\documentclass[english, 10pt]{amsart}

\usepackage{tikz}
\usepackage{aliascnt}
\usepackage{amsfonts}

\usepackage{amsthm}
\usepackage{amsmath}
\usepackage[all,poly,knot]{xy}
\usepackage{amsfonts}
\usepackage{color}
\usepackage{hyperref}
\usepackage{comment}
\usepackage{csquotes}
\usepackage{amscd}
\usepackage{mathtools}
\usepackage{dsfont}
\usepackage{enumitem}
\usepackage{calc}

\usepackage[charter]{mathdesign}

\usepackage{float}

\DeclareFontFamily{U}{dutchcal}{\skewchar\font=45 }
\DeclareFontShape{U}{dutchcal}{m}{n}{<-> s*[1.0] dutchcal-r}{}
\DeclareFontShape{U}{dutchcal}{b}{n}{<-> s*[1.0] dutchcal-b}{}
\DeclareMathAlphabet{\mathlcal}{U}{dutchcal}{m}{n}
\SetMathAlphabet{\mathlcal}{bold}{U}{dutchcal}{b}{n}

\usepackage{contour}
\usepackage{ulem}
\usepackage{amsmath,calligra,mathrsfs}
\DeclareMathOperator{\sheafhom}{\mathscr{H}\text{\kern -3pt {\calligra\large om}}\,}
\DeclareMathOperator{\sheafend}{\mathscr{E}\text{\kern -3pt {\calligra\large nd}}\,}

\theoremstyle{plain}

\newtheorem{thm}{Theorem}[section]

\newtheorem*{claim}{Claim}

\newtheorem*{conj*}{Conjecture}
\newtheorem*{theorem*}{Theorem}

\theoremstyle{definition}
\newtheorem{definition}[subsubsection]{Definition}

\newtheorem{definition-proposition}[subsubsection]{Definition-Proposition}

\newtheorem{theorem}[subsubsection]{Theorem}
\newtheorem{example}[subsubsection]{Example}
\newtheorem{remark}[subsubsection]{Remark}

\newtheorem{convention}[subsubsection]{Convention}

\newtheorem{proposition}[subsubsection]{Proposition}
\newtheorem*{proposition*}{Proposition}
\newtheorem{lemma}[subsubsection]{Lemma}
\newtheorem{corollary}[subsubsection]{Corollary}
\newtheorem{conjecture}[subsubsection]{Conjecture}

\newcommand{\Abb}{\mathbb{A}}
\newcommand{\Cbb}{\mathbb{C}}

\newcommand{\Gbb}{\mathbb{G}}

\newcommand{\Lbb}{\mathbb{L}}

\newcommand{\Pbb}{\mathbb{P}}
\newcommand{\Qbb}{\mathbb{Q}}
\newcommand{\Rbb}{\mathbb{R}}
\newcommand{\Sbb}{\mathbb{S}}

\newcommand{\Zbb}{\mathbb{Z}}
\newcommand{\Wbb}{\mathbb{W}}
\newcommand{\Mbb}{\mathbb{M}}

\newcommand{\mrm}{\mathrm{m}}

\newcommand{\Gbf}{\mathbf{G}}

\newcommand{\Vbf}{\mathbf{V}}

\newcommand{\gfrak}{\mathfrak{g}}

\newcommand{\GL}{{\mathrm{GL}}}

\newcommand{\Spec}{{\mathrm{Spec}}}

\newcommand{\End}[2]{\mathrm{End}_{#1}(#2)}

\newcommand{\Lie}{\mathrm{Lie}}

\newcommand{\ra}{\rightarrow}

\newcommand{\lra}{\longrightarrow}

\newcommand{\mono}{\hookrightarrow}

\newcommand{\inv}{{-1}}
\newcommand{\bsh}{\backslash}
\newcommand{\isom}{\simeq}

\newcommand{\Sym}{\mathrm{Sym}}

\newcommand{\Hom}{\mathrm{Hom}}

\newcommand{\Res}{\mathrm{Res}}
\newcommand{\id}{\mathrm{id}}

\newcommand{\Bbf}{\mathbf{B}}
\newcommand{\Cbf}{\mathbf{C}}

\newcommand{\Mbf}{\mathbf{M}}
\newcommand{\Nbf}{\mathbf{N}}
\newcommand{\Tbf}{\mathbf{T}}

\newcommand{\Gal}{\mathrm{Gal}}

\newcommand{\Image}{\mathrm{Im}}

\newcommand{\Ocal}{\mathcal{O}}
\newcommand{\Pcal}{\mathcal{P}}

\newcommand{\MT}{\mathbf{MT}}

\newcommand{\Hbf}{\mathbf{H}}

\newcommand{\nfrak}{\mathfrak{n}}

\newcommand{\Pbf}{\mathbf{P}}

\newcommand{\ad}{\mathrm{ad}}
\newcommand{\der}{\mathrm{der}}

\newcommand{\zbar}{\bar{z}}

\newcommand{\Vbb}{\mathbb{V}}

\newcommand{\ibf}{\mathbf{i}}

\newcommand{\Dcal}{\mathcal{D}}
\newcommand{\codim}{\mathrm{codim}}
\newcommand{\Zbf}{\mathbf{Z}}
\newcommand{\HL}{\mathrm{HL}}
\newcommand{\Ad}{\mathrm{Ad}}

\newcommand{\rig}{\mathrm{rig}}
\newcommand{\nonrig}{\mathrm{nonrig}}
\newcommand{\atyp}{\mathrm{atyp}}
\newcommand{\typ}{\mathrm{typ}}
\newcommand{\pos}{\mathrm{pos}}
\newcommand{\factorpos}{\mathrm{f-pos}}
\newcommand{\hfrak}{\mathfrak{h}}
\newcommand{\hodgecd}{\mathrm{H-cd}}
\newcommand{\moncd}{\mathrm{M-cd}}

\newcommand{\SU}{\mathrm{SU}}

\makeatletter
\newcommand{\namelabel}[1]{%
  \phantomsection
  \renewcommand{\@currentlabel}{#1}
  \label{#1}
}
\def\paragraph{\@startsection{paragraph}{4}%
  \z@\z@{-\fontdimen2\font}%
  {\normalfont\bfseries}}
\makeatother

\newcommand{\thistheoremname}{}
\newtheorem*{genericthm*}{\thistheoremname}
\newenvironment{namedthm*}[1]{\renewcommand{\thistheoremname}{#1}%
	\begin{genericthm*}}
	{\end{genericthm*}}

\theoremstyle{remark}

\numberwithin{equation}{section}

\def\Z{\mathbb Z}

\def\Spec{\mathrm{Spec}}

\def\log{\mathrm{log}\,}

\def\Spec{\mathrm{Spec}\,}

\usepackage[top=1.5in, bottom=1.5in, left=1.5in, right=1.5in]{geometry}

 \usepackage[hyperpageref]{backref}
 
\usepackage{xcolor}
\hypersetup{
	colorlinks,
	linkcolor={red!50!black},
	citecolor={blue!50!black},
	urlcolor={blue!80!black}
}

\contourlength{0.8pt}

\begin{document}
\title[On the distribution of non-rigid families]{On the distribution of non-rigid families in the moduli spaces}
\author[Ke Chen]{Ke Chen}
\address{School of Mathematics, Nanjing University}
\email{kechen@nju.edu.cn}

\author[Tianzhi Hu]{Tianzhi Hu}
\address{ School of Mathematics and Statistics, Wuhan University, Luojiashan, Wuchang, Wuhan, Hubei, 430072, P.R. China}
\email{hutianzhi@whu.edu.cn}

 \author{Ruiran Sun}
 \address{School of Mathematical Sciences, Xiamen University, Xiamen 361005, China}
\email{ruiransun@xmu.edu.cn}

\author[Kang Zuo]{Kang Zuo}
\address{ School of Mathematics and Statistics, Wuhan University, Luojiashan, Wuchang, Wuhan, Hubei, 430072, P.R. China; Institut f\"ur Mathematik, Universit\"at Mainz, Mainz, Germany, 55099}
\email{zuok@uni-mainz.de}

\begin{abstract}
This paper investigates the distribution of non-rigid families in a moduli space $\mathcal{M}$ of polarized projective manifolds for which the infinitesimal Torelli theorem holds. Guided by the analogy with unlikely intersection in Shimura varieties, we show that the image of any non-rigid classifying morphisms into $\mathcal{M}$ is contained in the Hodge locus as long as the derived Mumford-Tate group is $\Qbb$-simple and the period map is generically finite. If moreover the period domain is not Hermitian of rank at least 2, then the Hodge locus can be replaced by a closed subscheme, which yields a finiteness theorem of geometric Bombieri-Lang type.  Inspired by the Zilber-Pink conjecture, we also characterize the geometry of non-rigid locus by the specialness of bi-Hom schemes and the finiteness of ``structurally-atypical'' intersections. Finally, we specialize to the moduli spaces of polarized Calabi-Yau manifolds, formulate an unobstructedness conjecture for non-rigid maps which implies the specialness of bi-Hom schemes, prove a geometric Andr\'e-Oort theorem describing the Zariski closure of non-rigid locus, and test the theory and the conjecture for the explicit Viehweg-Zuo family of Calabi--Yau quintics in $\Pbb^4$.
\end{abstract}


\maketitle

\setcounter{tocdepth}{2}
\tableofcontents


\setcounter{tocdepth}{1}
\section{Introduction}\label{sec-1}
In their investigation of the Shafarevich conjecture, which concerns the finiteness of non-isotrivial families of genus $g$ curves over a fixed base curve, Parshin \cite{Par} and Arakelov \cite{Ara} separated the finiteness statement into two parts: {\it Boundedness} and {\it Rigidity}. For families of higher-dimensional varieties, Boundedness holds in a fairly general setting, whereas Rigidity fails, with numerous counterexamples.

A natural sequel is to study and classify non-rigid families of varieties. In this paper, we study the distribution of non-rigid families of polarized manifolds in the corresponding moduli space. Let $\mathcal{M}$ be the moduli stack of polarized projective manifolds with semi-ample canonical line bundle and with some fixed Hilbert polynomial. By \cite{Vieh95}, $\mathcal{M}$ admits a quasi-projective coarse moduli scheme $M$. Then for any family $f:\,V \to U$ in $\mathcal{M}(U)$, there is a classifying morphism $\varphi_f:\,U \to M$. We say that the family $f$ {\it lies in a general position of $\mathcal{M}$}, if the image $\varphi_f(U)$ passes through a general point of $M$. This paper is guided by two fundamental questions:

\begin{description}[labelwidth =\widthof{\bfseries9999}, leftmargin = !]
  \item[Question (A)] {\it If a family $f:\,V \to U$ lies at a general position of $\mathcal{M}$, is it rigid?}\\[-3mm]
  \item[Question (B)] {\it Can we characterize the Zariski closure of $\bigcup \varphi_f(U)$ in $M$, where $f:\,V \to U$ ranges over all the non-rigid families in $\mathcal{M}$?}
\end{description}

We denote the set $\bigcup \varphi_f(U)$ in \textbf{Question (B)} by $\mathrm{NRL}(\mathcal{M})$, the {\it Non-rigid Locus of $\mathcal{M}$}. Our investigation of $\mathrm{NRL}(\mathcal{M})$ is carried out under the assumption that the infinitesimal Torelli theorem holds on $\mathcal{M}$, so that the geometry of $\mathcal{M}$ is tightly controlled by the period map associated to the universal family.

\subsection{Previous and motivating results}
In the cases of $K3$ surfaces and abelian varieties, the rigidity or non-rigidity of families can be translated into the rigidity or non-rigidity of the associated period maps, facilitated by Torelli-type theorems and the dominance of period maps. Based on this, Faltings \cite{Fal} provided a condition for the rigidity of families of abelian varieties and verified the finiteness of families satisfying this condition.

Moreover, the moduli spaces $\mathcal{M}_{K3}$ and $\mathcal{A}_g$ are actually {\it Shimura varieties}. Noguchi's theorem then implies that the image $\varphi_f(U)$ of any non-rigid family $f$ is contained in a totally geodesic subvariety of $M=\Gamma\bsh\Dcal$, i.e. a {\it weakly special subvariety} of $M$. Consequently,
\[
\mathrm{NRL}(\mathcal{M}) \subset \{\textrm{weakly special subvarieties}\}
\]
for $\mathcal{M}=\mathcal{A}_g$ or $\mathcal{M}_{K3}$, confirming \textbf{Question (A)} immediately.

On the other hand, the existence of Hecke translates indicates an {\it``all-or-nothing''} principle for non-rigid families in Shimura varieties. In particular, the Zariski closure in \textbf{Question (B)} is simply the whole $M$ when $\mathcal{M}=\mathcal{M}_{K3}$ or $\mathcal{A}_g$. Therefore, \textbf{Question (B)} becomes truly interesting for moduli spaces which are {\it not} Shimura varieties, such as closed subvarieties in some Shimura variety or, more generally, in a Hodge variety (an arithmetic quotient of a period domain). In such settings, the characterization of Zariski closures in \textbf{Question (B)} is intimately related to the Zilber-Pink philosophy of unlikely intersections.

Saito and Zucker \cite{SZ, Saito} further employed Hodge theory to offer a complete classification of non-rigid families in $\mathcal{M}_{K3}$ and $\mathcal{A}_g$. These studies again rely on the dominance of period maps; however, outside the aforementioned cases, period maps are {\it never} dominant. Developing a systematic theory for non-dominant period maps is a central challenge addressed in this paper.

\subsection{Main results}
We now summarize the main theorems and conjectures of this paper, which address Questions (A) and (B) under various assumptions.

\subsubsection{Rigidity and finiteness}
Our first result answers \textbf{Question (A)} under certain assumptions.

\begin{thm}[=Theorem~\ref{thm-very-gen-rig}]\label{intro_very_gen_rig}
  Suppose the universal family $f:\,\mathcal{X} \to \mathcal{M}$ on the moduli stack carries a polarized $\mathbb{Z}$-VHS $\mathbb{V}$ satisfying the following conditions
  \begin{itemize}
    \item[(i)] the derived Mumford-Tate group is $\mathbb Q$-simple;
    \item[(ii)] the associated period map
          \[
        \Phi:\,\mathcal{M} \to \Gamma \backslash \mathcal{D}
          \]
          is generically finite onto its image.
  \end{itemize}
  Then $\mathrm{NRL}(\mathcal{M})$ is contained in $\mathrm{HL}(\mathbb{V})$, the {\rm Hodge locus} of $\mathbb{V}$.
\end{thm}
By the celebrated theorem of Cattani-Deligne-Kaplan \cite{CDK}, $\mathrm{HL}(\mathbb{V})$ is a countable union of closed irreducible algebraic subvarieties. Hence a family whose classifying morphism passes through a {\it very general point} of $\mathcal{M}$ is rigid. We say that $\mathcal{M}$ has the {\bf very general rigidity} property.

Even when very general rigidity holds, $\mathrm{NRL}(\mathcal{M})$ could still be Zariski dense in $M$; for example, this is already the case when we consider all Hecke translates of $\mathcal{A}_1 \times \mathcal{A}_{g-1}$ in $\mathcal{A}_g$. Our second result strengthens Theorem~\ref{intro_very_gen_rig} under a further restriction on the period domain.

\begin{thm}[=Theorem~\ref{thm-generic-rig}]\label{intro_thm-generic-rig}
  Let the moduli space $\mathcal{M}$ and the polarized $\mathbb{Z}$-VHS $\mathbb{V}$ be as in Theorem~\ref{intro_very_gen_rig}. If the Mumford-Tate domain $\mathcal{D}$ of $\mathbb{V}$ is not a bounded symmetric domain of rank $\geq 2$, then $\mathcal{M}$ has the {\bf generic rigidity} property: $\mathrm{NRL}(\mathcal{M})$ is contained in a finite union of closed irreducible algebraic subvarieties.
\end{thm}

Inspired by the Bombieri-Lang Conjecture, we obtain a geometric finiteness result in the spirit of Parshin-Arakelov.

\begin{thm}[=Corollary~\ref{geo-Bombieri-Lang}]\label{intro_geo-Bombieri-Lang}
  Let $\mathcal{M}$ be as in Theorem~\ref{intro_thm-generic-rig}. Assume further that $\mathcal{M}$ is either a fine moduli space or a Deligne-Mumford stack. Then there exists a Zariski open subset $M^o \subset M$ such that for any quasi-projective curve $U$, the set of nonconstant morphisms $\varphi: U \to \mathcal{M}$ with $\varphi(U) \cap M^o \neq \varnothing$ is finite.
\end{thm}
After the release of the first version of our paper on the arXiv, Engel, Lin, and Tayou, in their work \cite{ELT24}, proved a finiteness result for Hodge generic families of smooth projective hypersurfaces of a fixed degree by using a similar approach. For related results, we refer the reader to \S\ref{finiteness_hypersur}.

\subsubsection{Zariski closure of $\mathrm{NRL}(\mathcal{M})$}
We now turn to \textbf{Question (B)}. Guided by the analogy with the Zilber-Pink conjecture for Shimura varieties, we propose a series of conjectures on the structure of $\overline{\mathrm{NRL}(\mathcal{M})}^{\mathrm{Zar}}$.

Let $\Phi: M \hookrightarrow \Gamma\backslash\mathcal{D}$ be a quasi-finite period map. A product subvariety $Y_1\times Y_2\subset M$ obtained from the bi-Hom construction is called a {\it maximal non-rigid family} (Definition~\ref{max_nr}). The following conjecture posits that such maximal products are controlled by Hodge theory.

\begin{conjecture}[Specialness of bi-Hom schemes]\label{introNRLconj}
  Let $M$ be a moduli space satisfying the conditions in Theorem~\ref{intro_very_gen_rig}. For any maximal non-rigid family $H_1 \times H_2 \to M$, both $H_i$ are weakly special subvarieties with respect to the $\mathbb Z$-PVHS of the universal family. 
\end{conjecture}
This can be viewed as an analogue of Moonen's characterization of special subvarieties of Shimura varieties.

To describe the global structure of $\mathrm{NRL}(\mathcal{M})$, we introduce a finiteness conjecture on ``structurally-atypical intersections'':

Let $\Phi: M \hookrightarrow \Gamma\backslash\mathcal{D}$ be the period map above. By an $M$-weakly special subvariety, we mean a weakly special subvariety of $\Gamma\backslash\mathcal{D}$ that is contained in $M$ (Definition \ref{M-ws}).
\begin{conjecture}[Finiteness]\label{introfiniteness_atyp}
  There exists a finite set $\Sigma_{\mathrm{p}} \cup \Sigma_M$, where $\Sigma_{\mathrm{p}}$ consists of weakly special product subvarieties of $\Gamma\backslash\mathcal{D}$ and $\Sigma_M$ consists of $M$-weakly special subvarieties of $\Gamma\backslash\mathcal{D}$, such that for any maximal product subvariety $Y_1\times Y_2\subset M$, its weakly special closure $\langle Y_1\times Y_2\rangle_{\mathrm{ws}}$ either belongs to $\Sigma_{\mathrm{p}}$ or is contained in some $W\in\Sigma_M$. 
\end{conjecture}
 Assuming this finiteness conjecture, we obtain a structural description of $\overline{\mathrm{NRL}(\mathcal{M})}^{\mathrm{Zar}}$.

\begin{theorem}[=Theorem~\ref{condi_proof}, Structure of $\overline{\mathrm{NRL}(\mathcal{M})}^{\mathrm{Zar}}$]\label{introZar_closure}
Assume that Conjecture~\ref{introfiniteness_atyp} holds.
  Then the Zariski closure of the non-rigid locus decomposes as
  \[
  \overline{\mathrm{NRL}(\mathcal{M})}^{\mathrm{Zar}} = \bigcup^n_{i=1} P_i \cup \bigcup^m_{j=1} N_j,
  \]
  where each $P_i$ is a subvariety dominated by a fiber space whose generic fiber is a product variety, and each $N_j$ is birational to a Shimura variety of rank $\geq 2$.
\end{theorem}

We are able to prove a weaker but {\it unconditional} variant of this result, revealing the dichotomy imposed by Hodge theory.

\begin{theorem}[=Theorem~\ref{thm:snrl}, Zariski closure of weakly special non-rigid loci]\label{intro_thm:snrl}Assume the universal family of $\mathscr M$ carrying a polarized $\Z$-VHS with a generic finite period map. Let $$ I:=\{i\mid Y^{(i)}_1\times Y^{(i)}_2\text{ is a maximal product in } \mathscr M\text{ and it is weakly special in }\mathscr M\}.$$
  Let $Z$ be an irreducible component of $\overline{\bigcup_{i\in I}Y_1^{(i)}\times Y_2^{(i)}}^{\mathrm{Zar}}$ and let $\mathbf H_Z$ be the algebraic monodromy group of $\mathbb V|_Z$.
  \begin{enumerate}
    \item[(1)] If $\mathbf H_Z$ is $\mathbb Q$-simple, then $Z$ is birational to a Shimura variety of rank $\geq 2$.
    \item[(2)] If $\mathbf H_Z$ is not $\mathbb Q$-simple, then either $Z$ is a product variety, or there exists a non-trivial fibration $h:Z\to B$ whose fibres contain all special non-rigid loci and share a common algebraic monodromy group $\mathbf N\subsetneq \mathbf H_Z$.
  \end{enumerate}
\end{theorem}
This Theorem \ref{intro_thm:snrl} provides strong evidence for Conjecture~\ref{introZar_closure} and motivates the detailed study of Calabi-Yau moduli spaces in Section~\ref{sec-5}.

\subsection{A dictionary: Shimura varieties vs. moduli spaces and the Zilber-Pink philosophy}
The analogy between Shimura varieties and moduli spaces with a locally injective period map is a central theme of this paper. This dictionary translates deep arithmetic conjectures into geometric statements about moduli spaces and their non-rigid loci, and is particularly useful when viewed through the lens of the Zilber-Pink philosophy of unlikely intersections.

In the classical Shimura context, the Zilber-Pink conjecture predicts that the union of all atypical special subvarieties is not Zariski dense. Here, an atypical subvariety is one whose intersection with a Hodge generic subvariety has larger than expected codimension. This captures the idea that "unlikely intersections" -- intersections that are larger than what a naive dimension count would predict -- should be rare and organized in a precise way. The Andr\'e-Oort conjecture, a special case, predicts that a Zariski dense set of special points forces the ambient variety to be a Shimura variety.

Our work translates this entire philosophy to the setting of moduli spaces with a period map. The non-rigid locus $\mathrm{NRL}(\mathcal{M})$ is precisely the geometric manifestation of unlikely intersections: a maximal product $Y_1\times Y_2$ corresponds to an intersection of the moduli space $M$ with a product of period subdomains that is "structurally atypical" -- intersection preserving the product structure. The following table summarizes this dictionary, which guides our conjectures and theorems.
\footnotesize
\begin{table}[h]\label{tab_dict}
\centering
\begin{tabular}{|p{6.5cm}|p{6.5cm}|}
\hline
\textbf{Shimura variety context} & \textbf{Moduli space context} \\
\hline
$X$: a Shimura variety & $\Gamma \backslash \mathcal{D}$: a Hodge variety (arithmetic quotient of a period domain) \\
\hline
A closed subvariety $V \subset X$ & Moduli space $M$ with locally injective period map into $\Gamma \backslash \mathcal{D}$ \\
\hline
Maximal atypical subvariety & Maximal non-rigid family $Y_1\times Y_2\subset M$ (Definition~\ref{max_nr}) \\
\hline
Weakly special subvariety contained in $V$ & Hodge subvariety contained in $M$ (Definition~\ref{definition hodge locus and hecke orbit}) \\
\hline
Zilber-Pink Conjecture (finiteness of maximal atypical intersections) & Conjecture~\ref{introfiniteness_atyp} (Finiteness of weakly special closures of maximal products) \\
\hline
Moonen's characterization of special subvarieties & Conjecture~\ref{introNRLconj} (Specialness of bi-Hom schemes) \\
\hline
Geometric Andr\'e-Oort Conjecture (Zariski closure of positive dimensional weakly special subvarieties) & Conjecture \ref{introZar_closure} (Dichotomy for components of $\overline{\mathrm{NRL}(\mathcal{M})}^{\mathrm{Zar}}$) \\
\hline
\end{tabular}
\caption{Dictionary between Shimura varieties and moduli spaces.}
\end{table}
\normalsize

The dictionary is not merely formal: it predicts that the structure of $\overline{\mathrm{NRL}(\mathcal{M})}^{\mathrm{Zar}}$ is governed by the same dichotomy that appears in the Zilber-Pink conjecture. Components with $\mathbb Q$-simple monodromy are analogous to Shimura varieties in the sense that they are the "special" building blocks. Components with non-$\mathbb Q$-simple monodromy admit fibrations whose fibres are product-type, reflecting the fact that the underlying atypical intersection arises from a genuine product of Hodge data. Theorem~\ref{intro_thm:snrl} can thus be seen as a geometric incarnation of the Zilber-Pink dichotomy, providing unconditional evidence for the conjectural picture.

\subsection{Applications to Calabi-Yau moduli spaces}
In Section~\ref{sec-5}, we specialize to moduli spaces of polarized Calabi-Yau manifolds. The Hodge theory on Calabi-Yau varieties allows us to go further and provides concrete evidence for our conjectures.

Let $\mathscr M_h$ be a moduli space of polarized Calabi-Yau manifolds with a  fixed Hilbert polynomial $h$. We first formulate an unobstructedness conjecture for non-rigid families, which is crucial for establishing the specialness of bi-Hom schemes.

\begin{conjecture}[Unobstructedness]\label{introunob-conj}
  For a maximal extension of a non-rigid log map $\gamma:(H_1,S_1)\times(H_2,S_2)\to(\overline{\mathscr M}_h,D_\infty)$, the deformations of the restricted maps $(C,S_C)\times\{b\}\to(\overline{\mathscr M}_h,D_\infty)$ and $\{a\}\times(T,S_T)\to(\overline{\mathscr M}_h,D_\infty)$ are unobstructed for all $a\in H_1$, $b\in H_2$.
\end{conjecture}
Note that Conjecture~\ref{introunob-conj} can be regarded as a {\it relative version} of the celebrated Bogomolov-Tian-Todorov theorem. Assuming this conjecture, we prove that bi-Hom schemes are special, confirming Conjecture~\ref{introNRLconj} for Calabi-Yau moduli spaces.

\begin{proposition}[=Proposition~\ref{spe}, Specialness]\label{intro_spe}
  Assume Conjecture~\ref{introunob-conj}. Then for any maximal non-rigid family $H_1\times H_2\subset \mathscr M_h$, the subvariety $H_1\times H_2$ is special with respect to the $\mathbb Z$-PVHS $R^nf_*\mathbb Z_{\mathscr X}$.
\end{proposition}

A highlight of this section is the following geometric Andr\'e-Oort type theorem, which confirms Conjecture~\ref{introZar_closure} for Calabi-Yau moduli spaces.

\begin{theorem}[=Theorem~\ref{CY-AO}, Geometric Andr\'e-Oort for Calabi-Yau moduli]\label{intro_CY-AO}
  Assume $\mathscr M_h$ satisfies the conditions of Theorem~\ref{intro_very_gen_rig} and has a Zariski dense non-rigid locus. Then the period map $\Phi:\mathscr M_h\to\Gamma\backslash\mathcal D$ is dominant onto a Shimura variety.
\end{theorem}

Finally, we provide a detailed verification of our conjectures in the explicit Viehweg-Zuo example \cite{VZ, VZ-06}: a maximal non-rigid family of Calabi-Yau quintics in $\mathbb{P}^4$ obtained from cyclic covers. Through explicit computation of Hodge bundles and monodromy groups, we verify the unobstructedness conjecture and consequently establish the specialness of its bi-Hom scheme:

\begin{proposition}[Viehweg-Zuo example]\label{intro_VZ}
  In the Viehweg-Zuo example, Conjecture~\ref{introunob-conj} holds. Consequently, by Proposition~\ref{intro_spe}, the bi-Hom scheme is special.
\end{proposition}
See \S\ref{sec-VZ-eg} for details.
This provides a concrete and non-trivial test case for our general theory, demonstrating that the conjectural framework is not only consistent but also verifiable in a rich geometric setting.

\subsection{Organization of the paper}
The paper is organized as follows.

\begin{itemize}
    \item \textbf{Section \ref{sec-2} (Preliminaries)}: We recall the basic definitions of Hodge data, Hodge varieties, period maps, and the notions of Hodge locus, weakly special subvarieties, and (a)typicality. We also state the Geometric Zilber-Pink theorem of \cite{BKU} and the refined results of \cite{BU24}, which are essential tools for our analysis.

    \item \textbf{Section \ref{sec-3} (Rigidity and finiteness)}: This section contains our main results on the rigidity of moduli spaces. We introduce the non-rigid locus $\mathrm{NRL}(\mathcal{M})$ and prove the very general rigidity theorem (Theorem~\ref{thm-very-gen-rig}) and the generic rigidity theorem (Theorem~\ref{thm-generic-rig}). Using these, we derive a geometric Bombieri-Lang type finiteness result (Corollary~\ref{geo-Bombieri-Lang}) and give applications to $\overline{\mathcal{M}}_g$ and moduli spaces of hypersurfaces $\mathcal{M}_{d,n}$.

    \item \textbf{Section \ref{sec-4} (Zariski closures of non-rigid loci)}: Here we delve into the structure of the Zariski closure of $\mathrm{NRL}(\mathcal{M})$. We construct maximal non-rigid families using the bi-Hom scheme (Proposition~\ref{produce-max-non-rig}), relate them to non-rigid Hodge subvarieties (Lemma~\ref{lemma special subvarieties of factor type}), and formulate the finiteness and structure conjectures (Conjectures~\ref{finiteness_atyp} and \ref{Zar_closure}). We then prove Theorem~\ref{thm:snrl}, which describes the possible structures of an irreducible component of the closure.

    \item \textbf{Section \ref{sec-5} (Moduli spaces of polarized Calabi-Yau manifolds)}: We specialize to Calabi-Yau moduli. We state the unobstructedness conjecture (Conjecture~\ref{unob-conj}) and prove that under this conjecture, bi-Hom schemes are special (Proposition~\ref{spe}), providing evidence for Conjecture~\ref{NRLconj}. We then prove the geometric Andr\'e-Oort theorem for Calabi-Yau moduli (Theorem~\ref{CY-AO}), which confirms Conjecture~\ref{Zar_closure} in this setting. Finally, we present the Viehweg-Zuo example in detail, verifying the unobstructedness conjecture and thus the specialness of the bi-Hom scheme in this explicit case.
\end{itemize}

\subsection*{Acknowledgment}
The authors express their gratitude to Ngaiming Mok, Bruno Klingler, Xinyi Yuan, and Chenglong Yu for their interest in this work, and to Ariyan Javanpeykar for his valuable comments and suggestions, particularly for the helpful discussions on Deligne-Mumford stacks. R. Sun is grateful to Junyi Xie for very inspiring discussions on the proof of Theorem~\ref{introZar_closure}, and to Haohua Deng for helpful discussions on the Baily--Borel compactification of period images. K. Chen is supported by the National Natural Science Foundation of China Program 12471013, and K. Zuo by the National Natural Science Foundation of China (Key Program 12331002). Part of this work was completed during visits by K. Chen and R. Sun to the School of Mathematics and Statistics at Wuhan University. They extend their thanks to all the members there for their warm hospitality.

\section{Preliminaries}\label{sec-2}
In this section we discuss the distribution of Hodge loci along a period map and collect related materials for the main results. We follow \cite{BKU} closely.

As usual, PVHS stands for polarized variations of Hodge structures. $\Sbb=\Res_{\Cbb/\Rbb}\Gbb_\mrm$ is Deligne's torus, with the weight cocharacter $w: \Gbb_{\mrm,\Rbb}\ra\Sbb$ and the Hodge cocharacter $\mu:\Gbb_{\mrm,\Cbb}\ra\Sbb_\Cbb$. When $k$ is a commutative ring, (linear) $k$-groups always stand for affine algebraic group schemes over $\Spec{k}$.

\subsection{Hodge data and Hodge varieties}

We first recall some notions of Hodge data and Hodge varieties, following \cite{CPM} and \cite{Chen}.

\begin{definition}[Hodge datum and Hodge variety]\label{definition hodge datum and hodge variety} (1) A Hodge datum is a pair $(\Gbf,\Dcal)$ where
	\begin{itemize}
		\item $\Gbf$ is a connected reductive $\Qbb$-group;

		\item $\Dcal$ is the $\Gbf(\Rbb)$-conjugacy class of some $\Rbb$-group homomorphism $h: \Sbb\ra\Gbf_\Rbb$ such that \begin{itemize}
			\item $w\circ h: \Gbb_{\mrm,\Rbb}\ra\Gbf_\Rbb$ is central and defined over $\Qbb$;
			\item the conjugation by $h(\sqrt{-1})$ induces a Cartan involution of $\Gbf^\ad_\Rbb$, and $\Gbf^\ad$ admits NO compact factors defined over $\Qbb$.
		\end{itemize}

	\end{itemize}
	Take a compact open subgroup $K\subset\Gbf(\Abb_f)$, the double quotient $$\Gbf(\Qbb)\bsh (\Dcal\times \Gbf(\Abb_f)/K)$$ is called the Hodge variety associated to $(\Gbf,\Dcal)$ at level $K$. It is a complex analytic variety, which is smooth when $K$ is torsion-free. 

	(2) A connected Hodge datum is $(\Gbf,\Dcal^+)$ where for some Hodge datum $(\Gbf,\Dcal)$ we choose $\Dcal^+$ to be a connected component of $\Dcal$, equal to the $\Gbf(\Rbb)^+$-orbit of some $h\in\Dcal$. Choose further a congruence subgroup $\Gamma\subset\Gbf(\Qbb)^+$, we call the quotient space  $\Gamma\bsh\Dcal^+$ the connected Hodge variety associated to $(\Gbf,\Dcal^+)$ at level $\Gamma$. This is a complex analytic space, which is smooth when $\Gamma$ is torsion-free. 

	As is motivated by the case of Shimura varieties which are uniformized by period domains of Hermitian type, the evident map $$\pi_\Gamma: \Dcal^+\ra \Gamma\bsh\Dcal^+,\ h\mapsto \Gamma h$$ is called, by abuse of terminology,  the complex uniformization for $\Gamma\bsh\Dcal^+$.

	A Hodge variety in (1) is always a finite union of connected Hodge varieties associated to various congruence subgroups in $\Gbf(\Qbb)$.

	(3) A Hodge datum $(\Gbf,\Dcal)$ is of Shimura type if for any $h\in\Dcal$, the adjoint representation gives rise to a homomorphism $$\Sbb\overset{h}{\ra}\Gbf_\Rbb\ra\GL_{\gfrak,\Rbb}$$ which is a Hodge structure on $\gfrak=\Lie\Gbf$ of type $\{(-1,1),(0,0),(1,-1)\}$. The Hodge varieties thus obtained are Shimura varieties, which have been widely studied in algebraic geometry and arithmetic geometry.
\end{definition}

Unless stated otherwise, we only use connected Hodge data/varieties, and the adjective ``connected'' is omitted. In particular, in the expression $(\Gbf,\Dcal)$ for a Hodge datum, $\Dcal$ is understood to be connected.

\begin{definition}\label{definition hodge datum construction}
	(1) A morphism between Hodge data $(\Gbf_1,\Dcal_1)\ra(\Gbf_2,\Dcal_2)$ consists of \begin{itemize}
		\item a $\Qbb$-group homomorphism $f: \Gbf_1\ra\Gbf_2$;

		\item an $f$-equivariant map $f_*: \Dcal_1\ra\Dcal_2$, which sends any $h\in\Dcal_1$ to the composition $f\circ h: \Sbb\ra\Gbf_{2,\Rbb}$ lying in $\Dcal_2$.
	\end{itemize}

	(2) When $\Gamma_i\bsh\Dcal_i$ is a Hodge variety associated to $(\Gbf_i,\Dcal_i)$ at level $\Gamma_i$ ($i=1,2$), and $f: (\Gbf_1,\Dcal_1)\ra(\Gbf_2,\Dcal_2)$ is a morphism of Hodge data such that $f(\Gamma_1)\subset\Gamma_2$, then the evident map $$\Gamma_1\bsh\Dcal_1\ra \Gamma_2\bsh\Dcal_2,\ \Gamma_1 h\mapsto \Gamma_2 f_*(h)$$ is called the morphism between Hodge varieties associated to $f$. Only morphisms of this type are involved in this paper.

	(3) A Hodge subdatum of $(\Gbf,\Dcal)$ is given by a morphism $(\Gbf',\Dcal')\ra(\Gbf,\Dcal)$ of Hodge data such that $f:\Gbf'\ra\Gbf$ is an inclusion of $\Qbb$-subgroups and $f_*: \Dcal'\ra\Dcal$ is injective. For example, if $f: (\Gbf',\Dcal')\ra(\Gbf,\Dcal)$ is a morphism of Hodge data, then the image $(f(\Gbf'),f_*(\Dcal'))$ is a Hodge subdatum in $(\Gbf,\Dcal)$.

	If moreover $\Gamma\bsh\Dcal$ is a Hodge variety associated to $(\Gbf,\Dcal)$ with complex uniformization map $\pi_\Gamma$, then we call $\pi_\Gamma(\Dcal')$ the Hodge subvariety given by $(\Gbf',\Dcal')$. It is equal to the image of a suitably defined morphism between Hodge varieties $\Gamma'\bsh\Dcal'\ra\Gamma\bsh\Dcal$.

	Hodge subvarieties are also called special subvarieties, similar to the terminology in the setting of Shimura varieties.

	(4) If $(\Gbf_1,\Dcal_1)$ and $(\Gbf_2,\Dcal_2)$ are both Hodge data, then the pair $(\Gbf_1\times\Gbf_2,\Dcal_1\times\Dcal_2)$ is a well-defined Hodge datum, where we take $h_1\in\Dcal_1$ and $h_2\in\Dcal_2$ to form a homomorphism $h=(h_1,h_2): \Sbb\ra(\Gbf_1\times\Gbf_2)_\Rbb$ so that $\Dcal_1\times\Dcal_2$ is identified as the $(\Gbf_1\times\Gbf_2)(\Rbb)^+$-orbit of $h$. This is called the product of $(\Gbf_1,\Dcal_1)$ and $(\Gbf_2,\Dcal_2)$, which we also denoted as $(\Gbf_1,\Dcal_1)\times(\Gbf_2,\Dcal_2)$. If we further take congruence subgroups $\Gamma_1\subset\Gbf_1(\Qbb)^+$ and $\Gamma_2\subset\Gbf_2(\Qbb)^+$ respectively, then we obtain the product $\Gamma_1\bsh\Dcal_1\times \Gamma_2\bsh\Dcal_2$ of complex analytic varieties as the Hodge variety associated to $(\Gbf_1,\Dcal_1)\times(\Gbf_2,\Dcal_2)$ at level $\Gamma_1\times\Gamma_2$.

	(5) If $(\Gbf,\Dcal)$ is a Hodge datum and $\Nbf$ is a normal $\Qbb$-subgroup in $\Gbf$, then to the quotient group $\Gbf/\Nbf$ we associate a natural Hodge datum $(\Gbf/\Nbf,\Dcal')$ as the quotient of $(\Gbf,\Dcal)$ by $\Nbf$: if $\Dcal$ is the $\Gbf(\Rbb)^+$-conjugacy orbit of some $h: \Sbb\ra\Gbf_\Rbb$, then $\Dcal'$ is defined as the $(\Gbf/\Nbf)(\Rbb)^+$-conjugacy orbit of the composition $\Sbb\overset{h}{\ra}\Gbf_\Rbb\ra(\Gbf/\Nbf)_\Rbb$. This is independent of the choice of $h$.

	In particular, write $\Gbf^\ad$ for the adjoint quotient of $\Gbf$ (modulo the center $\Cbf_\Gbf$ of $\Gbf$), then in the associated Hodge datum $(\Gbf^\ad,\Dcal')$ as the quotient of $(\Gbf,\Dcal)$ by $\Cbf_\Gbf$ we have $\Dcal'=\Dcal$. In this case, take a congruence subgroup $\Gamma\subset\Gbf(\Rbb)^+$ of image $\Gamma'$ in $\Gbf^\ad(\Rbb)^+$, then the natural map $\Gamma\bsh\Dcal\ra\Gamma'\bsh\Dcal$ is an isomorphism of complex analytic space because $\Gamma$ acts on $\Dcal$ through $\Gamma'$, and we can identify Hodge subvarieties from both sides. 

	(6) A Hodge datum $(\Gbf,\Dcal)$ is of CM type if $\Gbf$ is a $\Qbb$-torus, in which case $\Dcal$ is reduced to a single point, and the Hodge subvarieties thus obtained are also called CM points or special points.
\end{definition}

\begin{definition}[Hodge locus and Hecke orbit]\label{definition hodge locus and hecke orbit}

	Let $M=\Gamma\bsh\Dcal$ be a Hodge variety associated to $(\Gbf,\Dcal)$ at level $\Gamma$.

	(1) We write $\HL(M)$ for the union of proper Hodge subvarieties in $M$, called the Hodge locus of $M$.

	(2) 	In a Hodge variety $M=\Gamma\bsh\Dcal$, a $\Cbb$-analytic subvariety $Z$ is Hodge generic if it is NOT contained in any Hodge subvariety $M'\subsetneq M$. Such varieties necessarily exist because the collection of Hodge subvarieties in $\Gamma\bsh\Dcal$ is countable. The set of Hodge generic points in $M$ is dense for the analytic topology.
	
	In particular, if $Z\subset M$ is a closed irreducible analytic subvariety, then there exists a minimal Hodge subvariety $M_Z$ containing $Z$, and $Z$ is Hodge generic in $M_Z$.

\end{definition}

We also need the notion of weakly special subvarieties in Hodge varieties.

\begin{definition}[weakly special subvarieties]\label{definitioin weakly special subvarieties} Let $M=\Gamma\bsh\Dcal$ be a Hodge variety associated to a Hodge datum $(\Gbf,\Dcal)$ at level $\Gamma$. A weakly special subvariety in $M$ is of the form $M'=f(M_1\times\{x_2\})$ where $f: (\Gamma_1\times\Gamma_2)\bsh (\Dcal_1\times\Dcal_2)\ra \Gamma\bsh \Dcal$ is a morphism between Hodge varieties from one defined by a Hodge datum of product type $(\Gbf_1\times\Gbf_2,\Dcal_1\times\Dcal_2)=(\Gbf_1,\Dcal_1)\times(\Gbf_2,\Dcal_2)$, in which $M_1=\Gamma_1\bsh\Dcal_1$ is of dimension $>0$ and $x_2\in M_2=\Gamma_2\bsh\Dcal_2$ is an arbitrarily chosen point.

	Note that the $M'$ thus defined is a Hodge subvariety if in the presentation $f(M_1\times\{x_2\})$ one may choose $x_2$ to be a CM point in $M_2$, or equivalently one may choose $(\Gbf_2,\Dcal_2)$ to be a Hodge datum of CM type.
	

\end{definition}

\subsection{Period maps and conjectures of Zilber-Pink type}

We proceed to Hodge loci associated to variations of Hodge structures. Note that when a variation of Hodge structures is given on a complex manifold $S$, we use the underlying local system to denote the variation itself.

\begin{definition}[period map and Hodge locus]\label{definition period map hodge locus} Let $S$ be a complex manifold carrying a $\Zbb$-PVHS $\Vbb$. Following \cite{CPM}, we choose a sufficiently general base point $o\in S$ and write $V=\Vbb_o\otimes_\Zbb\Qbb$, so that the generic Mumford-Tate group for $\Vbb$ over $S$ is identified as the Mumford-Tate group $\Gbf$ for the rational Hodge structure on $V$, which is a $\Qbb$-subgroup of $\GL_V$.

	(1) The Hodge structure $h_o: \Sbb\ra\GL_{V,\Rbb}$ (of image in $\Gbf_\Rbb$) gives rise  to the period domain $\Dcal=\Gbf(\Rbb)^+ h_o$.  We call $(\Gbf,\Dcal)$ the generic Hodge datum associated to $(S,\Vbb,o)$.

	In this case we have the period map $\Phi: S\ra \Gamma\bsh \Dcal$, deduced from a period map at the level of universal covering of $S$: $\widetilde{S}\overset{\widetilde{\Phi}}{\lra} \Dcal$; here $\Gamma$ is a suitable congruence subgroup in $\Gbf(\Qbb)^+$, which exists after possibly passing to a finite \'etale covering of $S$. 

	The period map is complex analytic, and we define special subvarieties in $S$ (with respect to $\Vbb$ and $\Phi$) to be irreducible analytic components of $\Phi^\inv(M')$ with $M'$ running over proper Hodge subvarieties in $\Gamma\bsh \Dcal$. The union of special subvarieties in $S$ is denoted as $\HL(S,\Vbb)$, called the Hodge locus of $S$ (with respect to $\Vbb$). 

	We also have the quotient Hodge datum $(\Gbf^\ad,\Dcal)$ from $(\Gbf,\Dcal)$ modulo the center $\Cbf_\Gbf$, and this results in an adjoint period map, as the composition $\Phi^\ad: S\ra\Gamma\bsh\Dcal\ra\Gamma'\bsh\Dcal$. The irreducible components of $(\Phi^\ad)^\inv(M')$ are exactly the special subvarieties in $S$ as $M'$ runs through Hodge subvarieties in $\Gamma'\bsh\Dcal$.

	(2) Note that the choice of base point $o$ defines a representation of $\pi_1(S,o)$ on the $\Qbb$-vector space $V$, and the Zariski closure of the image of $\pi_1(S,o)\ra\GL_V(\Qbb)$ is a linear $\Qbb$-subgroup of $\GL_V$ contained in $\Gbf$, whose neutral component $\Mbf$ is called the algebraic monodromy group for $\Vbb$. It is known that $\Mbf$ is a normal $\Qbb$-subgroup in $\Gbf$.

	In this case we also have $\Gbf^\ad\isom\Mbf^\ad\times\Nbf^\ad$ for some connected reductive $\Qbb$-group $\Nbf$, together with a decomposition of period domain $\Dcal\isom\Dcal_\Mbf\times\Dcal_\Nbf$, where $\Dcal_\Mbf$ is the $\Mbf^\ad(\Rbb)^+$-orbit of the composition $\Sbb\overset{h}{\ra}\Gbf_\Rbb\ra\Gbf^\ad_\Rbb\ra\Mbf^\ad_\Rbb$ giving rise to a Hodge datum $(\Mbf^\ad,\Dcal_\Mbf)$; the definition for $(\Nbf^\ad,\Dcal_\Nbf)$ is similar. The composition $\Phi_\Mbf: S\ra \Gamma\bsh\Dcal\ra\Gamma_\Mbf\bsh\Dcal_\Mbf$ is referred to as the reduced period map, with $\Gamma_\Mbf$ the image of $\Gamma$ in $\Mbf^\ad(\Qbb)^+$. According to \ref{lemma special subvarieties of factor type}, we can find, up to passing to suitable finite \'etale covering of $S$ and shrinking $\Gamma$, an isomorphism of Hodge varieties $$\Gamma_\Mbf\bsh\Dcal_\Mbf\times\Gamma_\Nbf\bsh\Dcal_\Nbf\isom\Gamma\bsh\Dcal$$ such that the image of the period map $\Phi$ is $S_\Mbf\times\{y_\Nbf\}$ where $S_\Mbf\subset\Gamma_\Mbf\bsh\Dcal_\Mbf$ is a Hodge generic subvariety, and $y_\Nbf\in\Gamma_\Nbf\bsh\Dcal_\Nbf$ is a Hodge generic point.

	Following \cite{BKU}, we refer to the induced map $\Phi_\Mbf: S\ra\Gamma_\Mbf\bsh\Dcal_\Mbf$ as the weak period map, where $\Gamma_\Mbf\bsh\Dcal_\Mbf$ is identified as a weakly special subvariety in $\Gamma\bsh\Dcal$ via the point $y_\Nbf$ in the remaining factor $\Gamma_\Nbf\bsh\Dcal_\Nbf$; $(\Mbf,\Dcal_\Mbf)$ is also called the weak Hodge datum associated to $(S,\Vbb)$. 

	We also notice that special subvarieties in $S$ are exactly the irreducible analytic components of $\Phi_\Mbf^\inv(Z)$ with $Z$ running over Hodge subvarieties in $\Gamma_\Mbf\bsh\Dcal_\Mbf$, and the same is true for the characterization of weakly special subvarieties. Therefore it is often convenient to work with the case where $\Mbf^\ad=\Gbf^\ad$, or equivalently where $\Mbf$ contains $\Gbf^\der$ (i.e. $\Mbf$ only differs from $\Gbf$ by some central $\Qbb$-torus).

\end{definition}

\begin{remark}
	In the setting above, the algebraic monodromy group $\Mbf$ inside the Mumford-Tate group $\Gbf$ contains NO compact $\Qbb$-factors. In fact, the algebraic monodromy group remains unchanged when we replace $S$ by a finite \'etale covering. The image of the monodromy representation of $\pi_1(S)$ inside $\Gbf(\Rbb)$ is a discrete subgroup, and will not meet any compact normal subgroups in $\Gbf(\Rbb)$ when we pass to a subgroup in $\pi_1(S)$ of finite index. Hence $\Mbf$ itself does not contain any compact $\Qbb$-factors.
\end{remark}

The notion of atypical special subvarieties is central in the study of \cite{BKU}:

\begin{definition}\label{definition atypical special subvarieties} 

	Let $\Vbb$ be a $\Zbb$-PVHS on a smooth quasi-projective algebraic variety $S$ over $\Cbb$, of generic Hodge datum $(\Gbf,\Dcal)$ and period map $\Phi: S\ra M=\Gamma\bsh\Dcal$ (regarding $S$ as a complex manifold). Let $Y\subset S$ be a closed irreducible algebraic subvariety over $\Cbb$.

	(1) There exists a smallest Hodge subvariety $M_Y$ in $M$ whose pre-image along $\Phi$ contains $Y$ as an irreducible component, and is referred to as the special closure of $\Phi(Y)$. The Hodge codimension of $Y$ with respect to $\Vbb$ is $$\hodgecd(Y,\Vbb):=\dim M_Y-\dim\Phi(Y).$$ We say $Y$ is atypical with respect to $\Vbb$ if either $Y$ is singular for $\Vbb$, or we have the inequality $$\hodgecd(Y,\Vbb)<\hodgecd(S,\Vbb),$$ namely $\codim_M(\Phi(Y))<\codim_M(\Phi(S))+\codim_M(M_Y).$ Otherwise $Y$ is said to be typical.

	(2) There exists a smallest weakly special subvariety $M'_Y$ in $M$ containing $\Phi(Y)$, which can be realized, up to shrinking $\Gamma$, as $M_1\times\{y_2\}$ lying in a Hodge subvariety in $M$ of the form $M_1\times M_2$ with $y_2\in M_2$ Hodge generic, following the description of Lemma \ref{lemma special subvarieties of factor type}. We thus define the monodromic codimension of $Y$ with respect to $\Vbb$ to be $$\moncd(Y,\Vbb):=\dim M'_Y-\dim\Phi(Y).$$ We say $Y$ is monodromically atypical with respect to $\Vbb$ if either $Y$ is singular for $\Vbb$, or we have the inequality $$\moncd(Y,\Vbb)<\moncd(S,\Vbb).$$ Otherwise $Y$ is said to be monodromically typical. Note that $\moncd(S,\Vbb)$ only depends on the weak period map for $(S,\Vbb)$.

It is clear that atypical subvarieties are monodromically atypical, and monodromically typical subvarieties are typical (see \cite[Lemma 5.6]{BKU}) .

\end{definition}
Note that when $S=M$ is a Hodge variety of Shimura type with the identity map as the period map, then every Hodge subvariety is typical. Atypicality only arise when the period map is non-trivial. We will be mainly concerned with the case where the period map is (generically) finite.
\bigskip

\begin{convention}\label{convention on subscripts for hodge loci}
	When $\Phi: S\ra M=\Gamma\bsh\Dcal$ is a period map associated to some $\Zbb$-PVHS on $S$, inside the Hodge locus $\HL(S,\Vbb)$ we have the following subsets of interest: $$\HL(S,\Vbb)_\atyp,\ \HL(S,\Vbb)_\typ,\ \HL(S,\Vbb)_\rig,\ \HL(S,\Vbb)_\nonrig$$ which are the union of atypical resp. typical resp. rigid resp. non-rigid proper special subvarieties in $S$ with respect to $\Vbb$. Here rigid resp. non-rigid special subvarieties in $S$ are irreducible analytic components of pre-images of rigid resp. non-rigid Hodge subvarieties in $M$ under $\Phi$.

	We also need other subsets in $\HL(S,\Vbb)$ following \cite{BKU}. Up to enlarging $S$, we may assume that $\Phi$ is proper, and up to shrinking $\Gamma$, we may assume that $\Dcal=\Dcal_1\times\cdots\Dcal_r$ is a direct product decomposition along the decomposition $\Gbf^\ad=\prod_j\Gbf_j$ into simple $\Qbb$-factors, and thus $\Gamma\bsh\Dcal\isom \prod_j\Gamma_j\bsh\Dcal_j$ for suitable congruence subgroups $\Gamma_j\subset\Gbf_j^\ad(\Qbb)^+$.

	\begin{itemize}
		\item $\HL(S,\Vbb)_\pos$ is the union of special subvarieties $Z$ in $S$ such that $\Phi(Z)$ is of dimension $>0$, i.e. of positive period dimension.

		\item $\HL(S,\Vbb)_\factorpos$ is the union of special subvarieties in $S$ such that the image of $Z$ in each $\Gamma_j\bsh\Dcal_j$ along the natural projection $\Gamma\bsh\Dcal\ra\Gamma_j\bsh\Dcal_j$ is of dimension $>0$.
	\end{itemize}
\end{convention}
The work \cite{BKU} proposed the following conjecture, motivated by the Zilber-Pink conjecture for Shimura varieties:

\begin{conjecture}[Zilber-Pink]\label{conjecture zilber pink} Let $\Phi: S\ra M=\Gamma\bsh\Dcal$ be the period map for a $\Zbb$-PVHS on an irreducible smooth quasi-projective algebraic variety $S$. Then the following equivalent conditions are true: \begin{itemize}
		\item[(a)] The subset $\HL(S,\Vbb)_\atyp$ is a finite union of maximal atypical special subvarieties of $S$;
		\item[(b)] The subset $\HL(S,\Vbb)_\atyp\subsetneq S$ is an algebraic subvariety in $S$;

		\item[(c)] The subset $\HL(S,\Vbb)_\atyp$ is NOT Zariski dense in $S$.
	\end{itemize}

\end{conjecture}

As an evidence toward the conjecture, in \cite{BKU} one finds the following results:

\begin{theorem}[Geometric Zilber-Pink]\label{GZP} Let $S$ be a smooth quasi-projective algebraic variety over $\Cbb$ carrying a $\Zbb$-PVHS $\Vbb$, and let $Z\subsetneq S$ be an irreducible component in the Zariski closure of $\HL(S,\Vbb)_\pos\cap\HL(S,\Vbb)_\atyp$ in $S$. Assume that $Z$ is NOT a maximal atypical special subvariety, whose generic Mumford-Tate group $\Gbf_Z$ gives rise to the Hodge datum $(\Gbf_Z,\Dcal_Z)$ for $\Vbb|_Z$ and the adjoint quotient $\Gbf_Z^\ad$ of $\Gbf_Z$ leads to the period map $\Phi_Z: Z\ra \Gamma_Z\bsh\Dcal_Z$.  Then up to replacing $S$ by a finite \'etale covering, we have a further decomposition $$\Phi_Z=(\Phi',\Phi''): Z\ra \Gamma_Z\bsh\Dcal_Z\isom\Gamma'\bsh\Dcal'\times\Gamma''\bsh\Dcal''$$ along $(\Gbf_Z^\ad,\Dcal_Z)\isom(\Gbf',\Dcal')\times(\Gbf'',\Dcal'')$, such that $Z$ contains a Zariski dense subset of atypical subvarieties whose images along $\Phi''$ are CM points in $\Gamma''\bsh\Dcal''$, and $Z$ is Hodge generic in a typical subvariety in $S$ with respect to $\Phi$.

\end{theorem}

We remark that the proof of this theorem in \cite{BKU} relies on the Ax-Schanuel theorem of VHS, proved in \cite{NPT} for Shimura varieties and \cite{BT} for general case. Moreover, the proof of the Ax-Schanuel theorem in \cite{NPT} cruically used the volume estimate of complex analytic subvarieties in Hermitian symmetric domain developed in \cite{HT}, and Bakker and Tsimerman generalized it to horizontal subvarieties in period domains in \cite{BT}.

Recently, Baldi and Urbanik \cite[Thm 1.9/Thm7.1/Lem7.5]{BU24} refines this Theorem \ref{GZP}. Here is their result:

\begin{theorem}\label{BU24}
    Given a $\mathbb Z$VHS over $S$, there are finitely many fibrations in $S$: $f_j:C_j\to B_j,\ j=1,2,\cdots,n$, whose fibers are all maximal monodromically atypical weakly special subvarieties of $S$. And each maximal monodromically atypical weakly special subvariety is exact a fiber of some fibration above.
\end{theorem}
\begin{remark}
  (1) The base of the fibration $f_j$ can be trivial, i.e. reduced to a point. In this case the whole space $C_j$ of $f_j$ is a monodromically atypical special subvariety of $S$. Hence we have a strict inclusion $C_j\subsetneq S$, see \cite[Theorem 6.1, case(a)]{BKU}.

  (2) If $f_j$ is a non-trivial fibration , one can prove that its fibers have the same algebraic monodromy group $\bf N$, which is a normal subgroup of the algebraic monodromy group $\Hbf_{C_j}$ of the whole space $C_j$ of $f_j$. Besides, one knows that $\Hbf_{C_1j}$ is not $\mathbb Q$-simple, see \cite[Theorem 6.1, case(b)]{BKU}.
\end{remark}

\subsection{Levels and related results}

Several imporant results from \cite{BKU} involve the notion of levels of Hodge data:

\begin{theorem}[Properties of (a)typical loci]
	Let $S$ be a smooth quasi-projective algebraic variety over $\Cbb$, carrying a $\Zbb$-PVHS $\Vbb$.

	(1) If $\HL(S,\Vbb)_\typ$ is non-empty, then $\HL(S,\Vbb)$ is dense in $S$ for the analytic topology.

	(2) When $\Vbb$ is of level at least 3,

	\begin{itemize} \item $\HL(S,\Vbb)_\factorpos$ is a finite union of maximal atypical special subvarieties, and thus it is algebraic;

		\item $\HL(S,\Vbb)_\typ$ is empty, and thus $\HL(S,\Vbb)=\HL(S,\Vbb)_\atyp$.\end{itemize}

	(3) When $\Vbb$ is of level 2 and $\Gbf^\ad$ is simple, then $\HL(S,\Vbb)_\pos\cap\HL(S,\Vbb)_\atyp$ is algebraic.

\end{theorem}

We recall definitions related to Hodge-Lie algebras and their levels. For conveniences we formulate them with coefficients in subfields of $\Rbb$. 

\begin{definition}\label{definition F-Hodge data}

	Let $F$ be a subfield in $\Rbb$. We already have the notion of $F$-Hodge structures following \cite{Andre}.

	(1) An $F$-Hodge structure is a pair $(V,\rho)$ where: \begin{itemize}
		\item $V$ is a finite dimensional $F$-vector space;
		\item $\rho: \Sbb\ra\GL_{V\otimes_F\Rbb}$ is a homomorphism of $\Rbb$-groups.
	\end{itemize} Namely it is an $\Rbb$-Hodge structure whose underlying vector space is defined over the subfield $F$. Other notions like weights, polarizations, variations of $F$-Hodge structures, etc. are understood similarly.

	(2) The Mumford-Tate group of an $F$-Hodge structure $(V,\rho)$ is defined to be the smallest $F$-subgroup $\Hbf$ of $\GL_V$ such that the image $\rho(\Sbb)$ is contained in $\Hbf_\Rbb$, and is denoted as $\MT(V,\rho)=\MT(V)$.

	Note that when $F=\Rbb$, the Mumford-Tate group is nothing but the image $\rho(\Sbb)$ itself as an $\Rbb$-subgroup. In general, $\MT(V,\rho)$ is not commutative.

	(3) We can also define (connected) $F$-Hodge data formally as a pair $(\Gbf, \Dcal)$ where: \begin{itemize}
		\item $\Gbf$ is a connected reductive $F$-group;

		\item $\Dcal$ is a $\Gbf(\Rbb)^+$-conjugacy orbit of homomorphisms $h: \Sbb\ra\Gbf_\Rbb$ satisfies the same condition as in \ref{definition hodge datum and hodge variety}; it is still required that $\Gbf$ contains NO normal $F$-factors whose associated Lie group is fixed by the Cartan involution induced by $h(\sqrt{-1})$.

	\end{itemize} Other notions like $F$-Hodge subdata are defined similarly as in \ref{definition hodge datum construction}.


\end{definition}

\begin{remark}
	The period map in the setting of \ref{definition F-Hodge data} should be treated carefully. When $F$ is a number field in $\Rbb$ and $S$ is a complex manifold carrying an $F$-PVHS, we fix a base point $s\in S$ and study the monodromy representation of $\pi_1(S,s)$ on $V:=\Vbb_s$. Then we have the algebraic monodromy group $\Hbf$ as the smallest $F$-subgroup in $\GL_V$ containing the image of the monodromy representation. Using the arguments in \cite{Andre}, we see that the $\Hbf_\Rbb$-action on $V_\Rbb$ is normalized by the image $\rho_s(\Sbb)$ inside $\GL_{V_\Rbb}$ given by the Hodge structure $(V,\rho_s)$ at $s$, and hence the $F$-subgroup $\Gbf$ generated by $\Hbf$ and $\MT(V,\rho_s)$ serves as the generic Mumford-Tate group for the $F$-PVHS and contains $\Hbf$ as a normal $F$-subgroup. In this way we also get an $F$-Hodge datum $(\Gbf,\Dcal)$ and a period map $\widetilde{S}\ra\Dcal$ at the level of universal covering space.

	In this setting the monodromy representation of $\pi_1(S,s)$ has its image in $\Hbf(\Rbb)$. If this image happens to be contained in a discrete subgroup $\Gamma$ of $\Hbf(\Rbb)$ acting on $\Dcal$ discontinuously, then we still obtain a period map of the form $S\ra\Gamma\bsh\Dcal$, or even an analogue only involving the monodromic part $$S\ra\Gamma_{\Hbf^\ad}(\Dcal_{\Hbf^\ad})$$ using the $F$-Hodge datum $(\Hbf^\ad,\Dcal_{\Hbf^\ad})$ as the quotient of $(\Gbf,\Dcal)$ modulo the centralizer of $\Hbf$ in $\Gbf$.
\end{remark}

\begin{definition}\label{definition hodge lie algebra}   Let $F$ be either a number field in $\Rbb$ or $\Rbb$ itself.

	(1) An $F$-Hodge-Lie algebra is a pair $(\gfrak,h)$ consisting of: \begin{itemize}
		\item a reductive Lie algebra $\gfrak$ over $F$, endowed with an $F$-Hodge structure of weight zero $$\gfrak_\Cbb=\bigoplus_{i\in\Zbb}\gfrak^i$$ with $\gfrak^i=\gfrak^{i,-i}$ in the usual sense;

		\item the Lie bracket $[\ ,\ ]: \bigwedge^2\gfrak\ra \gfrak$ is a morphism of $F$-Hodge structures, and the negative of the Killing form $B_\gfrak: \gfrak^\ad\otimes_F\gfrak^\ad\ra F$ is a polarization of the $F$-Hodge structure (upon the identification of $\gfrak^\ad$ with $\gfrak^\der$ as an $F$-Hodge structure summand inside $\gfrak$).

		(2) When $F=\Rbb$, the level of an $\Rbb$-Hodge-Lie algebra $(\gfrak,h)$ is the largest $i\in\Zbb$ such that $\gfrak^i\neq0$ in the Hodge decomposition along $h$. Clearly this only depends on the $\Rbb$-Hodge structure on $\gfrak^\ad$. 

		When $F$ is a number field in $\Rbb$, and $(\gfrak,h)$ is an  $F$-Hodge-Lie algebra, with $\gfrak$ noncommutative and simple as a Lie algebra over $F$, then $\gfrak\otimes_F\Rbb=\bigoplus_j\hfrak_j$ is a direct sum of finitely many simple Lie algebras $\hfrak_j$ over $\Rbb$, each of which carries a Hodge-Lie algebra structure induced by $h$. In this case we define the level of $(\gfrak,h)$ to be the maximum of the levels of the $\hfrak_j$'s.

		When $F$ is a number field and $(\gfrak,h)$ is an $F$-Hodge-Lie algebra, we define its level to be the minimum of the levels of the (non-commutative) simple $F$-factors in $\gfrak$.

	\end{itemize}

	(3)  Let $F$ be a number field in $\Rbb$. When $(\Gbf,\Dcal)$ is an $F$-Hodge datum, the adjoint representation of $\Gbf$ on $\gfrak=\Lie\Gbf$ induces, via any $h\in\Dcal$, an $F$-Hodge-Lie algebra $(\gfrak,h)$. We thus define the level of $(\Gbf,h)$ to be the level of $(\gfrak,h)$, which only depends on $\gfrak^\ad$ and $\Dcal$.

	When $\Vbb$ is an $F$-PVHS on a complex manifold $S$, we define the level of $\Vbb$ to be the level of the associated $F$-Hodge datum.

\end{definition}

\begin{remark} Let $F$ be a number field in $\Rbb$. For a general $F$-Hodge datum $(\Gbf,\Dcal)$ of adjoint quotient $(\Gbf^\ad,\Dcal)$, the polarizability condition on $h\in\Dcal$ forces an isomorphism $\Gbf^\ad=\prod_{j=1}^r\Gbf_j$ with $\Gbf_j=\Res_{F_j/F}\Hbf_j$ given by some  number field $F_j$ in $\Rbb$ containing $F$ and some absolutely simple $F_j$-group $\Hbf_j$, and this results in a further decomposition $$(\Gbf^\ad,\Dcal)\isom\prod_{j=1}^r(\Gbf_j,\Dcal_j)$$ of $F$-Hodge data. In this case the complex embeddings of $F_j$ extending $F\mono\Rbb$ have to be real. In particular, when $F=\Qbb$, we are led to simple $\Qbb$-factors defined by absolutely simple linear groups defined over totally real number fields. 
\end{remark}

In \cite{BKU} the following properties of levels are established

\begin{lemma}\label{lemma generation at level 1} Let $(\Gbf,\Dcal)$ be an $\Rbb$-Hodge datum, with $\Gbf$ semi-simple of Lie algebra $\gfrak$ with Hodge decomposition $\gfrak_\Cbb=\oplus_{j\in\Zbb}\gfrak^j$. Then the minimal Lie subalgebra in $\gfrak_\Cbb$ containing $\gfrak^1$ and $\gfrak^\inv$ equals $\gfrak_\Cbb$ itself.

\end{lemma}

\begin{proposition}\label{proposition consequences of level 3 and level 2}
	(1) Let $(\Gbf,\Dcal)$ be a Hodge pair of level 3, and assume that $(\Gbf',\Dcal')$ is a Hodge subpair of Lie algebra $\gfrak'$ such that $(\gfrak')^{j}=\gfrak^j$ as long as $|j|\geq 2$. Then $\Gbf'=\Gbf$.

	(2) Let $(\Gbf,\Dcal)$ be a Hodge datum of level 2, with $\Gbf$ simple as a linear $\Qbb$-group, and assume that $(\Gbf',\Dcal')$ is a Hodge subdatum, such that $(\Lie\Gbf')^{-2}=(\Lie\Gbf)^{-2}$. Then $\Gbf'$ is simple as a linear $\Qbb$-group. 
\end{proposition}

\section{Rigidity and finiteness}\label{sec-3}
This section addresses Question (A) by investigating the rigidity of polarized moduli spaces $\mathcal{M}_h$ and their non-rigid locus $\mathrm{NRL}(\mathcal{M}_h)$. We characterize \textit{very general rigidity} (containment in a countable union of proper subvarieties) and \textit{generic rigidity} (finiteness of the Zariski closure). Our results are threefold:
(1) Theorem~\ref{thm-very-gen-rig} establishes very general rigidity assuming a $\mathbb{Q}$-simple derived Mumford--Tate group and a generically finite period map; 
(2) Theorem~\ref{thm-generic-rig} establishes generic rigidity if the period domain further lacks factors of rank $\ge 2$; 
(3) Corollary~\ref{geo-Bombieri-Lang} provides a geometric Bombieri--Lang type finiteness for morphisms from a fixed curve. 
We conclude by showing that $\mathrm{NRL}$ is not Zariski dense for $\overline{\mathcal{M}}_g$ ($g \ge 10$) and for smooth hypersurfaces $\mathcal{M}_{d,n}$ of sufficiently high degree ($d \ge 5$ for $n=2$; $d \ge 6$ for $n \ge 3$).

\subsection{Rigidity of morphisms mapping into algebraic varieties}
Let $M$ be a normal variety. Consider a morphism
\[
\varphi:\, U \to M,
\]
where $U$ is a normal variety with $0 < \mathrm{dim}\,U < \mathrm{dim}\,M$, and $\varphi$ is {\it generically finite} onto its image in $M$.

\begin{definition}
  We say $\varphi:\,U \to M$ is {\it non-rigid}, if there exists a pointed scheme $(T,o)$ with $\mathrm{dim}\,(T)_{\mathrm{red}}>0$, such that $\varphi$ can be extended to a morphism
  \[
  \Phi:\, U \times T \to M,
\]
i.e. $\Phi|_{U \times \{o\}} = \varphi$, where $\Phi$ is generically finite onto its image.\\[.1cm]
Otherwise, $\varphi$ is said to be {\rm rigid}.
\end{definition}

\subsection{Non-rigid locus}\label{sec nrl}

For a normal variety $M$, consider the following set
\[
\mathrm{NRL}(M):= \bigcup_{\varphi} \varphi(U)
\]
where $\varphi$ ranges over all the {\it non-rigid} morphisms.
\begin{example}\label{exam-NRL}
  \begin{itemize}
	\item[(i)] Let $M= \mathbb{P}^n$ ($n \geq 2$). Note that through each point in $\mathbb{P}^n$ one can find a family of lines passing it. Thus $\mathrm{NRL}(\mathbb{P}^n) = \mathbb{P}^n$.
    \item[(ii)] Let $M$ be a projective surface which is a compact ball quotient. Then its cotangent bundle $\Omega^1_M$ is ample. Using deformation theory of morphisms and the negativity of $T_M$ it is easy to verify that every curve mapping into $M$ is rigid. Thus $\mathrm{NRL}(M) = \varnothing$.
  \end{itemize}
\end{example}
\begin{definition}\label{rig-prop-general}
  We say that $M$ has the {\rm \textbf{very general rigidity}} property if the non-rigid locus $\mathrm{NRL}(M)$ is contained in a {\rm countable} union of closed proper irreducible subvarieties of $M$.\\[.1cm]
We say that $M$ has the {\rm \textbf{generic rigidity}} property if the Zariski closure of $\mathrm{NRL}(M)$ in $M$ is a {\rm finite} union of closed proper irreducible subvarieties of $M$.
\end{definition}

\begin{remark}
Note that the notion of rigidity property in Definition~\ref{rig-prop-general} is different from the rigidity of $M$ itself {\rm as a scheme}. For instance, $\mathbb{P}^n$ is rigid as a scheme. However, it is far from having the very general/generic rigidity property since $\mathrm{NRL}(\mathbb{P}^n)=\mathbb{P}^n$ (see Example~\ref{exam-NRL} above).
\end{remark}

Let $Z \subset M$ be a proper closed subscheme. We formulate the {\it relative non-rigid locus} as the following set
\[
\mathrm{NRL}(M,Z):= \bigcup_{\varphi} \varphi(U)
\]
where $\varphi$ ranges over all the non-rigid morphisms whose images are {\it not} contained in $Z$.

\begin{lemma}\label{rel-NRL}
$M$ has the very general (resp. generic) rigidity property if and only if the relative non-rigid locus $\mathrm{NRL}(M,Z)$ is contained in a countable (resp. finite) union of closed proper irreducible subvarieties of $M$.
\end{lemma}
\begin{proof}
  It follows from the following chain of inclusions
  \[
 \mathrm{NRL}(M,Z) \subset \mathrm{NRL}(M) \subset \mathrm{NRL}(M,Z) \cup Z.
  \]
\end{proof}

We observe that properties of very general and generic rigidity are preserved under finite \'{e}tale covers.

\begin{lemma}\label{fet-cover}
Let $\psi:\,\tilde{M} \to M$ be a finite \'{e}tale cover. Then $M$ has the very general (resp. generic) rigidity property if and only if $\tilde{M}$ has the very general (resp. generic) rigidity property.
\end{lemma}

\begin{proof}
  Since $\psi$ maps $\mathrm{NRL}(\tilde{M})$ into $\mathrm{NRL}(M)$, the ``only if'' part is clear. To verify the ``if'' part, it suffices to show that the restriction of $\psi$ on $\mathrm{NRL}(\tilde{M})$ maps {\it surjectively} onto $\mathrm{NRL}(M)$.

  Consider a non-rigid morphism $Y_1 \to M$, which extends to a morphism
  \[
  Y:=Y_1 \times Y_2 \to M
\]
which is generically finite onto its image. Without loss of generality, one can assume that $Y_i$ are smooth quasi-projective curves.
By taking fiber product, we obtain the following cartesian square
\[
  \xymatrix{
    \tilde{Y} \ar[r] \ar[d]_p  & \tilde{M} \ar[d]_{\psi} \\
    Y=Y_1 \times Y_2 \ar[r] & M.
    }
\]
We will show that $\tilde{Y} \to \tilde{M}$ gives a {\it non-rigid} morphism mapping to $\tilde{M}$, which dominates the non-rigid morphism $Y_1 \times Y_2 \to M$. This will confirm the surjectivity of $\psi:\,\mathrm{NRL}(\tilde{M}) \to \mathrm{NRL}(M)$.

Note that $\mathrm{Im}[\pi_1(\tilde{Y}) \to \pi_1(Y_1) \times \pi_1(Y_2)]$ is a subgroup of finite index. Therefore, after replacing curves $Y_i$ by some finite \'{e}tale covers if necessary, one can assume that the projection
\[
\pi_1(\tilde{Y}) \to \pi_1(Y_1) \times \pi_1(Y_2) \twoheadrightarrow \pi_1(Y_i)
\]
is surjective for $i=1,2$.

We claim that the composed morphism $\tilde{Y} \xrightarrow{p} Y_1 \times Y_2 \xrightarrow{\mathrm{pr}_2} Y_2$ has {\it connected} fibers. Let $\bar{Y}_1$ be the projective completion of the first curve. One can find a partial compactification $\tilde{Y}^{\textasciicircum} \supset \tilde{Y}$ such that $p:\,\tilde{Y} \to Y$ extends to a finite ramified cover
\[
p^{\textasciicircum}:\, \tilde{Y}^{\textasciicircum} \to \bar{Y}_1 \times Y_2.
\]
Now we consider the Stein factorization of the following proper morphism
\[
\tilde{Y}^{\textasciicircum} \xrightarrow{p^{\textasciicircum}} \bar{Y}_1 \times Y_2 \xrightarrow{\mathrm{pr}_2} Y_2,
\]
which gives the following commutative square
\[
  \xymatrix{
    \tilde{Y}^{\textasciicircum} \ar[d]_-{p^{\textasciicircum}} \ar[r]^{\alpha} & \tilde{Y}_2  \ar[d]^{\beta} \\
    \bar{Y}_1 \times Y_2 \ar[r]^-{\mathrm{pr}_2} & Y_2.
      }
\]
Here $\alpha$ is a proper morphism with connected fibers and $\beta$ is a finite morphism. Thus the following maps between fundamental groups
\[
\pi_1(\tilde{Y}) \twoheadrightarrow \pi_1(\tilde{Y}^{\textasciicircum}) \to \pi_1(Y_2)
\]
actually factors through $\beta_*:\, \pi_1(\tilde{Y}_2) \to \pi_1(Y_2)$. Since $\pi_1(\tilde{Y}) \to \pi_1(Y_2)$ is surjective, we know that $\beta_*$ is also surjective. Consequently, $\tilde{Y}_2=Y_2$ and $\beta$ is the identity map.

Therefore, $\alpha:\,\tilde{Y}^{\textasciicircum} \to Y_2$, and thus $\tilde{Y} \to Y_2$, has connected fibers. We notice that $\tilde{Y} \to Y_2$ is also smooth, thus a {\it fiber bundle} under the analytic topology. Let $F$ be a general fiber of $\tilde{Y} \to Y_2$. One has the following short exact sequence
\[
1 \to \pi_1(F) \to \pi_1(\tilde{Y}) \to \pi_1(Y_2) \to 1.
\]
It is clear that the general fiber $F$ is the finite \'{e}tale cover of $Y_1$, corresponding to the subgroup
\[
\mathrm{Im}[\pi_1(F) \hookrightarrow \pi_1(\tilde{Y}) \twoheadrightarrow \pi_1(Y_1) ] \subset \pi_1(Y_1)
\]
of finite index. Let $\pi_1(Y_1) \to S_d$ be the corresponding finite quotient, where $S_d$ is the group of deck transformations of the topological cover $F \to Y_1$. By the Riemann Existence Theorem, the complex structure on $F$ is {\it uniquely} determined by $\pi_1(Y_1) \to S_d$. That means the fibration $\tilde{Y} \to Y_2$ is isotrivial, and thus gives a non-rigid morphism mapping to $\tilde{M}$. This completes the proof.
\end{proof}

In this paper, we are mainly interested in the rigidity properties on the {\it moduli spaces of polarized manifolds}. Let $h \in \mathbb{Q}[\alpha,\beta]$ be the fixed (double) Hilbert polynomial. Recall the moduli stack $\mathcal{M}_h$ whose $\Cbb$-points are described by the following set of isomorphic classes:
\[
\mathcal{M}_h(\mathbb{C}):=  \left\{ (X, \mathcal L)\,  \middle| \,  \xymatrix@C=3.0cm{\ar@{}[r]^-{ X\text{  projective manifold with semi-ample canonical line bundle $\omega_X$, }}_-{\mathcal L  \text{ ample line bundle and with fixed Hilbert polynomial}  \, h(\alpha,\beta)=\chi(\mathcal L^{\alpha} \otimes \omega^{\beta}_X )          }&}
  \middle\} \middle/_{\cong}.  \right.
\]
More precisely, for a scheme $T$, the set $\mathcal{M}_h(T)$ consists of pairs $(f:\,X_T \to T, \mathcal{L})$ whose fibers are all in $\mathcal{M}_h(\mathbb{C})$.

Viehweg proved that the moduli stack $\mathcal{M}_h$ is isomorphic to a separated quotient stack $[H/\mathrm{PGL}_r]$, where $H$ is a quasi-projective scheme of finite type. Moreover, he proved that
\begin{theorem}[Theorem~1.13 in \cite{Vieh95}]
There exists a quasi-projective scheme $M_h$ of finite type, which is the coarse moduli space of the stack $\mathcal{M}_h$.
\end{theorem}

Similarly, one can consider the \textbf{non-rigid locus on the moduli space}. Define the following set
\[
\mathrm{NRL}(\mathcal{M}_h):= \bigcup_{\varphi} \varphi(U)
\]
where $\varphi:\,U \to \mathcal{M}_h$ ranges over all the {\it non-rigid} morphisms mapping into the moduli stack $\mathcal{M}_h$. Note that the non-rigid morphism $\varphi:\,U \to \mathcal{M}_h$ can be identified with the classifying morphism of a {\it non-rigid family} $f:\,V \to U$.

\begin{definition}\label{rig-prop-moduli}
  We say that the moduli stack $\mathcal{M}_h$ has the {\rm \textbf{very general rigidity}} property if the non-rigid locus $\mathrm{NRL}(\mathcal{M}_h)$ is contained in a {\rm countable} union of closed proper irreducible subvarieties of $M_h$.\\[.1cm]
We say that $\mathcal{M}_h$ has the {\rm \textbf{generic rigidity}} property if the Zariski closure of $\mathrm{NRL}(\mathcal{M}_h)$ in $M_h$ is a {\rm finite} union of closed proper irreducible subvarieties of $M_h$.
\end{definition}

The first example of moduli spaces is $\mathcal{M}_g$, the moduli space of genus $g$ algebraic curves. Shafarevich's conjecture can be reformulated as the finiteness of the following set
\[
\{ \varphi:\,U \to \mathcal{M}_g\,|\, \text{$\varphi$ nonconstant} \},
\]
where $U$ is a fixed quasi-projective curve.

Parshin and Arakelov proved this conjecture by dividing the finiteness statement into two parts: {\it Boundedness} and {\it Rigidity}. In their strategy, one first verifies that the above set can be partitioned into finitely many deformation equivalence classes (``Boundedness''), and then proves that there is exactly one element in each deformation equivalence class (``Rigidity'').

Parshin-Arakelov's strategy is the original motivation of studying rigidity properties on moduli spaces. In particular, their results implies that
\[
\mathrm{NRL}(\mathcal{M}_g) = \varnothing.
\]

It is easy to see that the genuine rigidity no longer holds when the moduli spaces classify higher dimensional varieties. For instance, a product of an elliptic curve and an abelian variety of dimension $g-1$ can be deformed along two factors independently, which gives a natural embedding $\mathcal{A}_1 \times \mathcal{A}_{g-1} \hookrightarrow \mathcal{A}_g$. Therefore,
\[
\mathrm{NRL}(\mathcal{A}_g) \neq \varnothing.
\]

\subsection{Very general rigidity}
By the above discussion, it is natural to study the distribution of non-rigid loci on various moduli spaces. In this paper, we prove the following result.
\begin{theorem}[Very general rigidity]\label{thm-very-gen-rig}
  Suppose the universal family $f:\,\mathcal{X} \to \mathcal{M}_h$ on the moduli stack carries a polarized $\mathbb{Z}$-VHS satisfying
  \begin{itemize}
    \item[(i)] the derived Mumford-Tate group is $\mathbb Q$-simple;
    \item[(ii)] the associated period map
          \[
        \Phi:\,\mathcal{M}_h \to \Gamma \backslash \mathcal{D}
          \]
          is generically finite onto its image.
  \end{itemize}
 Then $\mathcal{M}_h$ has the {\rm very general rigidity} property.
\end{theorem}

There are plenty of examples of moduli spaces satisfying the conditions in Theorem~\ref{thm-very-gen-rig}, like the moduli space of principally polarized abelian $g$-folds $\mathcal{A}_g$, moduli spaces of polarized Calabi-Yau manifolds, moduli spaces of hypersurfaces of high degree in $\mathbb{P}^N$, etc.

Now we give the proof of Theorem~\ref{thm-very-gen-rig} by relating the non-rigid locus with the Hodge locus of the associated VHS.
\begin{proposition}[$\mathbb{Q}$-nonsimplicity]\label{Q-nonsimplicity}
Let $Y=Y_1 \times Y_2$ be a product variety ($\mathrm{dim}\,Y_i>0$). Let $\mathbb{V}$ be a polarized $\mathbb{Z}$-VHS on $Y$ such that the induced period map is generically finite onto its image. Then the generic derived Mumford-Tate group $\Gbf^{\rm der}$ of $\mathbb{V}$, as well as the algebraic monodromy group $\Hbf$ of $\mathbb{V}$, is not $\mathbb{Q}$-simple.
\end{proposition}
\begin{proof}
Let $\Hbf_1,\ \Hbf_2$ be the algebraic monodromy groups of $\mathbb V|_{Y_1\times\{{\rm a\ generic\ point\ in\ }Y_2\}}$ and $\mathbb V|_{\{{\rm a\ generic\ point\ in\ }Y_1\}\times Y_2}$ respectively. Let $$\rho:\pi_1(Y_1\times Y_2)=\pi_1(Y_1)\times \pi_1(Y_2)\to \Hbf$$ be the monodromy representation. Clearly, $\bf H_1$ and $\Hbf_2$ are $\mathbb Q$-Zariski closure of the image of the monodromy representations of $\pi_1(Y_1)$ and $\pi_1(Y_2)$ respectively. Thus $\rho(\pi_1(Y_1))\subset Z_\Hbf(\rho(\pi_1(Y_2))),$ where $Z_\Hbf(\cdot)$ is the centralizer group of a subgroup of $\Hbf.$ Note that $Z_\Hbf(\rho(\pi_1(Y_2)))=Z_\Hbf(\Hbf_2)$ and it is an algebraic group. Thus we have $\rho(\pi_1(Y_1))\subset Z_\Hbf(\Hbf_2)$ and we have $\Hbf_1\subset Z_\Hbf(\Hbf_2)$, which proves $\Hbf_1\times \Hbf_2$ is a subgroup of $\Hbf$. However, $\rho(\pi_1(Y_1\times Y_2))$ as a subset of $\Hbf_1\times \Hbf_2$ is Zariski dense in $\Hbf$ and thus we have $\Hbf=\Hbf_1\times \Hbf_2.$
\end{proof}

\begin{corollary}\label{cor-Q-nonsimplicity}
Let $Y=Y_1 \times Y_2$ and $\mathbb{V}$ be the same as in Proposition~\ref{Q-nonsimplicity}. Let $\psi:\,Y' \to Y$ be a finite \'{e}tale cover. Then the generic Mumford-Tate group of $\psi^*\mathbb{V}$, as well as the algebraic monodromy group of $\psi^*\mathbb{V}$, is not $\mathbb{Q}$-simple.
\end{corollary}
\begin{proof}
One only needs to notice that $\pi_1(Y')$ is a finite index subgroup of $\pi_1(Y)$.
\end{proof}

\begin{proof}[Proof of Theorem~\ref{thm-very-gen-rig}]\namelabel{Proof}
  We first recall the construction of the period map associated with the universal family over the moduli stack $\mathcal{M}_h$.

  Recall that $\mathcal{M}_h$ is isomorphic to the separated quotient stack $[H/\mathrm{PGL}_r]$, where $H$ is a quasi-projective scheme of finite type. Moreover, $H$ is a subscheme of the Hilbert scheme parametrizing closed subschemes of $\mathbb{P}^{r-1}$ with Hilbert polynomial $h$, and $H$ is $\mathrm{PGL}_r$-invariant subset parametrizing closed subschemes which are smooth and connected with a fixed polarization. Therefore, $H$ carries a universal family of smooth polarized manifolds. This family induces a period map
  \[
\Phi_H:\,H \to \Gamma \backslash \mathcal{D}
  \]
  which is {\it invariant} under the $\mathrm{PGL}_r$-action on $H$. Then $\Phi_H$ can be descended to a holomorphic map $\Phi_{\mathcal{M}_h}$ from the moduli stack $\mathcal{M}_h=[H/\mathrm{PGL}_r]$ mapping to the period domain. By the assumption, $\Phi_{\mathcal{M}_h}$ is generically finite onto its image. Let $Z \subset \mathcal{M}_h$ be the exceptional set of $\Phi_{\mathcal{M}_h}$, i.e. the locus where $\Phi_{\mathcal{M}_h}$ fails to be quasi-finite.

  Now we consider the relative non-rigid locus $\mathrm{NRL}(\mathcal{M}_h,Z)$. To show the very general rigidity property of $\mathcal{M}_h$, it suffices to show that $\mathrm{NRL}(\mathcal{M}_h,Z)$ is contained in a countable union of closed proper irreducible substacks of $\mathcal{M}_h$ by Lemma~\ref{rel-NRL}.

  By its definition, $\mathrm{NRL}(\mathcal{M}_h,Z)$ is generated by morphisms from product varieties:
  \[
 \varphi:\, Y_1 \times Y_2 \to \mathcal{M}_h,
\]
where $\varphi$ is generically finite onto its image and the image of $\varphi$ is not contained in $Z$. Consider the polarized $\mathbb{Z}$-VHS $\mathbb{V}_Y$ on $Y:=Y_1 \times Y_2$ induced from the following period map
\[
Y=Y_1 \times Y_2 \xrightarrow{\varphi} \mathcal{M}_h \xrightarrow{\Phi_{\mathcal{M}_h}} \Gamma \backslash \mathcal{D},
\]
which is still generically finite onto its image.

Denote by $Y_H$ the fiber product $Y \times_{\mathcal{M}_h}H$. Let $S$ be the desingularization of the quasi-projective scheme $H$. Let $Y_S$ be the fiber product $Y_H \times_H S$. Then we have the following commutative diagram
\begin{align}
  \label{diag_SHY}
  \xymatrix{
    Y_S \ar[d]_{p_S} \ar[r]^-{\varphi_S} & S \ar[d]_{\delta} \ar[rd]^-{\Phi} & \\
    Y_H \ar[d]_{p_H} \ar[r]^-{\varphi_H} & H \ar[d]_p \ar[r]^-{\Phi_H} & \Gamma \backslash \mathcal{D} \\
    Y  \ar[r]^-{\varphi} & \mathcal{M}_h \ar[ru]_{\Phi_{\mathcal{M}_h}} &
    }
\end{align}
where $\Phi:= \Phi_H \circ \delta$ is the induced period map on the smooth quasi-projective variety $S$. Then we have the isomorphism of polarized $\mathbb{Z}$-VHS $p^*_H\mathbb{V}_Y \cong \varphi^*_H \mathbb{V}_H$, where $\mathbb{V}_H$ is the polarized $\mathbb{Z}$-VHS associated to the period map $\Phi_H$.

Denote by $\mathbb{V}_S$ the polarized $\mathbb{Z}$-VHS associated with the period map $\Phi:\, S \to \Gamma \backslash \mathcal{D}$. From \eqref{diag_SHY}, we know that
\[
\varphi^*_S\mathbb{V}_S = \varphi^*_S \delta^*\mathbb{V}_H = p^*_S \varphi^*_H\mathbb{V}_H = p^*_Sp^*_H\mathbb{V}_Y
\]
According to Proposition~\ref{Q-nonsimplicity}, the generic Mumford-Tate group $\mathrm{MT}(\mathbb{V}_Y)$ is not $\mathbb Q$-simple. That implies $\mathrm{MT}(\varphi^*_S\mathbb{V}_S)$ is also not $\mathbb{Q}$-simple.

On the other hand, we know that the generic Mumford-Tate group of $\mathbb{V}_S$ is $\mathbb Q$-simple, since the generic Mumford-Tate group associated to $\Phi_{\mathcal{M}_h}$ is $\mathbb Q$-simple. Therefore, $\mathrm{MT}(\varphi^*_S\mathbb{V}_S)$ is a proper subgroup of $\mathrm{MT}(\mathbb{V}_S)$, which means that $\varphi_S(Y_S) \subset \mathrm{HL}(S,\mathbb{V}_S)$.

Now the celebrated theorem of Cattani-Deligne-Kaplan in \cite{CDK} tells us that the Hodge locus $\mathrm{HL}(S,\mathbb{V}_S)$ is a countable union of closed irreducible algebraic subvarieties of $S$. Then $\varphi_H(Y_H) \subset \delta(\mathrm{HL}(S,\mathbb{V}_S))$, and $\delta(\mathrm{HL}(S,\mathbb{V}_S))$ is a countable union of closed irreducible algebraic subvarieties of $H$.

We say that a closed subvariety $\Sigma \subset H$ is {\it saturated}, if for each point $x \in \Sigma$, the $\mathrm{PGL}_r$-orbit of $x$ is also contained in $\Sigma$.
Then each irreducible component of $\delta(\mathrm{HL}(S,\mathbb{V}_S))$ is saturated since the restriction of the period map $\Phi_H$ on a $\mathrm{PGL}_r$-orbit is constant.

Therefore, the inclusion $\delta(\mathrm{HL}(S,\mathbb{V}_S)) \hookrightarrow H$ is $\mathrm{PGL}_r$-equivariant. Consequently, $\delta(\mathrm{HL}(S,\mathbb{V}_S))$ is mapped to a countable union of closed irreducible substacks of $\mathcal{M}_h$ under the projection $p:\,H \twoheadrightarrow \mathcal{M}_h=[H/\mathrm{PGL}_r]$. Note that $\varphi(Y) \subset p(\delta(\mathrm{HL}(S,\mathbb{V}_S)))$ by \eqref{diag_SHY}. This implies that
\[
\mathrm{NRL}(\mathcal{M}_h,Z) \subset p(\delta(\mathrm{HL}(S,\mathbb{V}_S))),
\]
which confirms the very general rigidity property of $\mathcal{M}_h$.
\end{proof}

\subsection{Generic rigidity}
If we put further restriction on the Mumford-Tate domain associated with the polarized $\mathbb{Z}$-VHS, Theorem~\ref{thm-very-gen-rig} can be strengthened to the following form.

\begin{theorem}[Generic rigidity]\label{thm-generic-rig}
  Let the moduli space $\mathcal{M}_h$ and the polarized $\mathbb{Z}$-VHS $\varphi$ be the same as in Theorem~\ref{thm-very-gen-rig}. If the Mumford-Tate domain $\mathcal{D}$ of $\varphi$ is not a bounded symmetric domain of rank $\geq 2$, then $\mathcal{M}_h$ has the {\rm generic rigidity} property.
\end{theorem}
In this section, we prove Theorem~\ref{thm-generic-rig}. In addition to the techniques utilized in the proof of Theorem~\ref{thm-very-gen-rig}, we also conduct a case-by-case analysis of the different levels of the Hodge-Lie algebra.

We will see that both the level one case and the cases with levels greater than two are relatively straightforward to handle. Therefore, we begin by focusing on the analysis of the level two case.
\begin{theorem}[Level two case]\label{level-two-atyp}
  Let $\mathbb{V}$ be a polarized $\mathbb{Z}$-VHS on a smooth connected complex quasi-projective variety $S$, with the generic Hodge datum $(\Gbf,\Dcal)$. Assume that the Mumford-Tate group $\Gbf$ is $\mathbb{Q}$-simple. Assume that the level of the Hodge-Lie algebra of the algebraic monodromy group $\Hbf=\Gbf$ is two.

  Let $\varphi_S:\,Y_S \to S$ be a morphism from a quasi-projective variety $Y_S$ mapping to $S$. Assume that there is a surjective morphism from $Y_S$ to a product variety $Y=Y_1 \times Y_2$ such that $\varphi^*_S\mathbb{V}$ can be descended to a polarized $\mathbb{Z}$-VHS $\mathbb{V}_Y$ on $Y$, and the period map associated with $\mathbb{V}_Y$ is generically finite onto its image. Let $W \supset \varphi_S(Y_S)$ be a weakly special subvariety of $S$ containing the image of $\varphi_S$, such that $\varphi_S:\,Y_S\to W$ is monodormically generic (i.e., the algebraic monodromy group of $\varphi^*_S\mathbb{V}$ coincides with the algebraic monodromy group of $\mathbb{V}|_W$). Then $W$ is {\it monodromically atypical}.
\end{theorem}
\begin{proof}
  We argue by contradiction. Suppose that $W$ is monodromically typical.  From \cite[Proposition 7.2]{BKU}, we know that the algebraic monodromy group $\Hbf'$ of $\mathbb V|_W$ is $\mathbb Q$-simple. Consider the monodromy representation associated with $\varphi^*_S\mathbb{V}$:
  \[
 \rho_S:\, \pi_1(Y_S) \to \Hbf'(\mathbb{R}) .
\]
By the assumption on $\varphi^*_S\mathbb{V}$, $\rho_S$ factors through the following
\[
 \rho:\, \pi_1(Y) \to \Hbf'(\mathbb{R}) ,
\]
which is the monodromy representation associated with $\mathbb{V}_Y$. This is a contradiction by Corollary~\ref{cor-Q-nonsimplicity}, which states $\Hbf'$ is not $\mathbb Q$-simple.
\end{proof}

With the preceding discussion on the level two case and the application of the Geometric Zilber-Pink Theorem from \cite{BKU}, we are now prepared to prove Theorem~\ref{thm-generic-rig}.

\begin{proof}[Proof of Theorem~\ref{thm-generic-rig}]
  Note that we are in the same context as Theorem~\ref{thm-very-gen-rig}, with further requirement on the Mumford-Tate domain. Thus we use the notations from the \ref{Proof} of Theorem~\ref{thm-very-gen-rig} freely, especially the commutative diagram \eqref{diag_SHY}.

  Let $(\Gbf,\Dcal)$ be the Hodge datum associated with the period maps appeared in \eqref{diag_SHY}. The proof is divided into three parts based on the level of the Hodge-Lie algebra of $\Gbf$.

  \paragraph{\textbf{Level 1 case}:} The Hodge datum $(\Gbf,\Dcal)$ is of Shimura type in this case. In particular, $\Dcal$ is a bounded symmetric domain. The rank $\geq 2$ cases have already been ruled out by the assumption of Theorem~\ref{thm-generic-rig}. Thus we will focus on the rank one case.

  If the rank of $\Dcal$ is one, then by the classification of Hermitian symmetric spaces (cf. \cite[Ch. X]{Helg}) we know that $\Dcal$ is a complex ball. Consider a morphism $Y \to \mathcal{M}_h$ such that the image of $Y$ is not contained in $Z$. Then the induced period map
  \[
  \Phi:\, Y \to \Gamma \backslash \mathcal{D}
\]
is generically finite onto its image. Without loss of generality, we assume that $Y$ is a smooth quasi-projective curve.
  By the extension theorem of Borel in \cite{Borel}, $\Phi$ can be extended to a morphism
  \[
  \bar{\Phi}:\, \bar{Y} \to X:= \overline{\Gamma \backslash \mathcal{D}},
  \]
  where $\bar{Y}$ is the projective completion of $Y$ and $X$ is some smooth toroidal compactification of $\Gamma \backslash \mathcal{D}$. By virtue of Lemma~\ref{fet-cover}, we are free to replace the monodromy group $\Gamma$ by some finite index subgroup such that the monodromies of the VHS along the boundary $\bar{Y} \setminus Y$ are {\it uniponent}. Therefore, the hermitian metric $h$ on $\bar{\Phi}^*\Omega^1_X(\log D_{\infty})$ ($D_{\infty}$ is the boundary divisor of the compactification) induced from the left invariant metric on $\Gamma \backslash \mathcal{D}$ has only logarithmic singularities at $\bar{Y} \setminus Y$. Since $\mathcal{D}$ is a complex ball and $\Phi$ is nonconstant, we know that $(\bar{\Phi}^*\Omega^1_X(\log D_{\infty}),h)$ is Griffiths positive on $Y$. Therefore, by simple curvature estimate, we know that each quotient bundle of $\bar{\Phi}^*\Omega^1_X(\log D_{\infty})$ must have positive degree on $\bar{Y}$.

  On the other hand, if the morphism $Y \to \mathcal{M}_h$ is non-rigid, then by deformation theory there exists a nonzero section in $H^0(\bar{Y}, \bar{\Phi}^*T_X(-\log D_{\infty}))$. Taking dual of this section will produce a quotient
  \[
  \bar{\Phi}^*\Omega^1_X(\log D_{\infty}) \twoheadrightarrow \mathcal{O}_{\bar{Y}},
  \]
  which is absurd by the above discussion. Therefore, $\mathrm{NRL}(\mathcal{M}_h,Z) = \varnothing$.

  \paragraph{\textbf{Level 2 case}:} Consider a non-rigid morphism
  \[
  \varphi:\, Y:=Y_1 \times Y_2 \to \mathcal{M}_h
  \]
  which appears in $\mathrm{NRL}(\mathcal{M}_h,Z)$. By the same construction in the \ref{Proof} of Theorem~\ref{thm-very-gen-rig}, one obtains the diagram \eqref{diag_SHY}. Since we are in the level two case, by applying Theorem \ref{level-two-atyp} to the morphism $\varphi_S:\,Y_S \to S$, one concludes that $\varphi_S(Y_S)$ is contained in a {\it monodromically atypical weakly special} subvariety of $S$.

  By a variant of the Geometric Zilber-Pink Theorem \ref{GZP}, we know that the Zariski closure of $\mathrm{HL}(S,\mathbb{V})_{\mathrm{ws},w-atyp}$ is either a finite union of closed irreducible algebraic subvarieties of $S$, or the adjoint Mumford-Tate group $\Gbf^{\mathrm{ad}}$ decomposes as a non-trivial product. Since $\Gbf^{\mathrm{der}}$ is absolutely simple, we know that the second possibility will never happen.

  Using the same argument as in the \ref{Proof} of Theorem~\ref{thm-very-gen-rig}, we know that $\delta(\mathrm{HL}(S,\mathbb{V})_{\mathrm{ws},w-atyp})$ is {\it saturated} in $H$, and its Zariski closure is also saturated. Therefore, $p(\delta(\mathrm{HL}(S,\mathbb{V})_{\mathrm{ws},w-atyp}))$ is contained in a finite union of closed irreducible substacks of $\mathcal{M}_h$. Since each $\varphi_S(Y_S)$ is contained in $\mathrm{HL}(S,\mathbb{V})_{\mathrm{ws},w-atyp}$, we know that $\varphi(Y)$ is contained in $p(\delta(\mathrm{HL}(S,\mathbb{V})_{\mathrm{ws},w-atyp}))$. Consequently,
  \[
  \mathrm{NRL}(\mathcal{M}_h,Z) \subset p(\delta(\mathrm{HL}(S,\mathbb{V})_{\mathrm{ws},w-atyp})),
  \]
  which confirms the generic rigidity property of $\mathcal{M}_h$.

  \paragraph{\textbf{Level $\geq 3$ case}:} Denote by $\mathrm{HL}(S,\mathbb{V})_{\mathrm{pos}}$ the Hodge locus of positive period dimension. Since the level is at least three, one has
  \[
  \mathrm{HL}(S,\mathbb{V})_{\mathrm{pos}} = \mathrm{HL}(S,\mathbb{V})_{\mathrm{ws},w-atyp}
  \]
  by \cite[Theorem 7.1]{BKU}. That implies
  \[
  \varphi_S(Y_S) \subset \mathrm{HL}(S,\mathbb{V})_{\mathrm{ws},w-atyp}
  \]
 for each $\varphi_S$ induced from a non-rigid morphism $\varphi:\,Y=Y_1 \times Y_2 \to \mathcal{M}_h$. The rest of the proof is the same as the level two case.
\end{proof}

\subsection{Bombieri-Lang type finiteness results}
The purpose of studying the rigidity property on $\mathcal{M}_g$ is to derive the Shafarevich's finiteness conjecture. Similarly, we want to use Theorem~\ref{thm-generic-rig} to derive a geometric Bombieri-Lang type finiteness result.

\begin{corollary}[Geometric Bombieri-Lang]\label{geo-Bombieri-Lang}
  Let the moduli space $\mathcal{M}_h$ be the same as in Theorem~\ref{thm-generic-rig}. Assume further that
  \begin{itemize}
    \item[(a)] Either $\mathcal{M}_h$ is a {\rm fine} moduli space, i.e. the universal family can be descended to the coarse moduli scheme $M_h$;
    \item[(b)] Or $\mathcal{M}_h$ is a {\rm Deligne-Mumford} stack.
  \end{itemize}
  Then there exists a Zariski open subset $M^o_h \subset M_h$ such that for any quasi-projective curve $U$, the following set
  \[
  \{\varphi:\, U \to \mathcal{M}_h\,|\, \text{$\varphi$ nonconstant and } \iota(\varphi(U)) \cap M^o_h \neq \varnothing \,\}
  \]
is a finite set.
\end{corollary}


In fact, $M^o_h= M_h \setminus \overline{\mathrm{NRL}(\mathcal{M}_h)}^{\mathrm{Zar}}$. The proof follows a similar strategy as Parshin-Arakelov's, i.e. {\it Boundedness} plus {\it Generic Rigidity}, as mentioned in Section~\ref{sec nrl}.

\begin{proof}[Proof of Corollary~\ref{geo-Bombieri-Lang}]
  Consider the following set
  \[
  \Sigma:=\{\varphi:\, U \to \mathcal{M}_h\,|\, \text{$\varphi$ nonconstant and } \iota(\varphi(U)) \cap M^o_h \neq \varnothing \,\},
  \]
  where $M^o_h= M_h \setminus \overline{\mathrm{NRL}(\mathcal{M}_h)}^{\mathrm{Zar}}$. Note that Theorem~\ref{thm-generic-rig} guarantees that $M^o_h$ is a Zariski open subset of $M_h$. The finiteness of $\Sigma$ is a direct consequence of the following two statements.\\

  \paragraph{\textbf{Boundedness}} We will show that there are finitely many deformation equivalence classes of $\varphi$ in $\Sigma$.

  If $\mathcal{M}_h$ is a fine moduli space, then one can find a projective compactification $\bar{M}_h \supset M_h$ with normal crossing boundary $D_{\infty}$. Then $\Sigma$ can be regarded as a subset of the log Hom scheme $\mathrm{Hom}\left( (\bar{U},D_U), (\bar{M}_h,D_{\infty})  \right)$, where $D_U:= \bar{U}\setminus U$. By the Arakelov type inequality (cf. \cite[Theorem 0.3]{VZ-01}), we know that $\mathrm{Hom}\left( (\bar{U},D_U), (\bar{M}_h,D_{\infty})  \right)$ is of finite type.


  In general, $\mathcal{M}_h$ is a Deligne-Mumford stack: It is an Artin stack by the construction $\mathcal{M}_h=[H/\mathrm{PGL}_r]$. The Matsusaka-Mumford theorem guarantees that $\mathcal{M}_h$ is separated. The polarization on the objects parametrized by $\mathcal{M}_h$ implies that $\mathcal{M}_h$ has an affine diagonal (cf. for instance {\cite[Lemma~2.2]{Jav-Lou}}). Moreover, the geometric points of $\mathcal{M}_h$ are finite and reduced, which follows from {\cite[Lemma~7.6]{Vieh95}} and the fact that objects parametrized by $\mathcal{M}_h$ are defined over $\mathbb{C}$. This confirms that $\mathcal{M}_h$ is a Deligne-Mumford stack.

  Since $\mathcal{M}_h$ is a Deligne-Mumford stack, we know that it is {\it compactifiable} in the sense of \cite{KL11} as $\mathcal{M}_h=[H/\mathrm{PGL}_r]$ (cf. {\cite[Lemma~4.2]{KL11}}). To show that $\mathcal{M}_h$ is {\it weakly bounded} in the sense of \cite{KL11}, one needs to find a line bundle $\mathcal{L}$ on some coarse compactification of $\mathcal{M}_h$ such that $\mathcal{L}$ is ample over $M_h$, and verify the Arakelov type inequality for $\mathcal{L}$.

  In our setting, since the infinitesimal Torelli theorem holds on $\mathcal{M}_h$, it is natural to consider the {\it Griffiths line bundle} on $\mathcal{M}_h$, that is, $\bigotimes_p \mathrm{det}\,\mathcal{F}^p$, where $\{\mathcal{F}^p\}$ is the Hodge filtration of the polarized $\mathbb{Z}$-VHS associated with the universal family. Alternatively, the Griffiths line bundle can be written as $\bigotimes_p \left(\mathrm{det}\,\mathcal{E}^{p,q}\right)^{\otimes p}$, where $\mathcal{E}^{p,q}:= \mathcal{F}^p/\mathcal{F}^{p+1}$ is the $p$-th grading piece of the Hodge filtration. By virtue of {\cite[Lemma~2]{KV-04}}, the Griffiths line bundle can be descended to a $\mathbb{Q}$-line bundle $\mathcal{L}_{M_h}$ on the coarse moduli space $M_h$.

  Let $Y$ be the image of the quasi-finite period map $\Phi:\,\mathcal{M}_h \to \Gamma \backslash \mathcal{D}$. By {\cite[Theorem~6.14]{BBT-23}} we know that $Y$ is a quasi-projective scheme. Moreover, the argument in the proof of {\cite[Corollary~7.3]{BBT-23}} implies that the period map $\mathcal{M}_h \to Y \subset \Gamma \backslash \mathcal{D}$ factors through a quasi-finite morphism $\Psi:\,M_h \to Y$, and the Griffiths $\mathbb{Q}$-line bundle $\mathcal{L}_{M_h}$ is the pull-back of the Griffiths $\mathbb{Q}$-line bundle $\mathcal{L}_Y$ on the period image $Y$. It is known from {\cite[Theorem~6.14]{BBT-23}} that $\mathcal{L}_Y$ is ample on $Y$. Therefore, $\mathcal{L}_{M_h}=\Psi^*\mathcal{L}_Y$ is {\it ample} on $M_h$.

  Arakelov inequalities for Hodge bundles have been extensively studied (see, for example, \cite{Pet00, JZ02, VZ03-asianJ}). An Arakelov type inequality for the Griffiths line bundle can be established as follows: In {\cite[Proposition~2.1]{VZ03-asianJ}}, it is shown that the degree of $\mathrm{det}\,\mathcal{E}^{p,q}$ is bounded above by a uniform constant multiplied by the degree of the logarithmic canonical bundle of the base curve, provided that the VHS has unipotent monodromies. By summing these inequalities with appropriate weights, one obtains the required Arakelov inequality for the Griffiths line bundle in the unipotent case. The Arakelov inequality for the Griffiths line bundle in the general case has also been established in {\cite[Theorem~1.6]{Bru20}}. This completes the proof of the weak boundedness of $\mathcal{M}_h$, and the boundedness follows from \cite[Theorem 1.7]{KL11}.\\

  \paragraph{\textbf{Rigidity}} There is exactly one element in each deformation equivalence class in $\Sigma$. But this follows directly from the definition of $\Sigma$.
\end{proof}

\begin{remark}
Note that the boundedness result holds even without assuming the infinitesimal Torelli theorem. For example, the boundedness of families of canonically polarized manifolds can be established using the Arakelov inequality for direct image sheaves $\mathrm{det}\,f_*\omega^{\otimes m}_{\mathcal{X}/\mathcal{M}_h}$, often referred to as the {\it Viehweg line bundles}. For families of polarized manifolds with a semiample canonical line bundle, boundedness can still be achieved by applying the Arakelov inequality to certain direct image sheaves. However, since the primary focus of this paper is on \textbf{rigidity properties}, and the infinitesimal Torelli theorem is crucial to our approach in studying rigidity, we have opted for the most straightforward method to obtain the boundedness result within the framework of Corollary~\ref{geo-Bombieri-Lang}.
\end{remark}

\subsection{Examples and applications }
\subsubsection{Generic rigidity property of $\overline{\mathcal M_g}$}
Let $\overline{\mathcal M_g}$ be the partial compactification of $\mathcal M_g$. More precisely, the moduli space of $g$ genus stable curves of compact type (cf. \cite{GeoCurve}). The usual Torelli map $j^0:\mathcal M_g\ra\mathcal A_g$ extends to a morphism $j:\overline{\mathcal M_g}\ra\mathcal A_g$ (cf. \cite{M-O} p. 554 and \cite{CLZ}).

Recall that Shafarevich conjecture states that $\mathrm{NRL}(\mathcal M_g)=\varnothing$. It is natural to consider the distribution of the non-rigid locus $\mathrm{NRL}(\overline{\mathcal M_g})$. Here is our result:

\begin{proposition}

    $\overline{\mathcal M_g}$ has the generic rigidity property for $g\geq 10$, which means $\mathrm{NRL}(\overline{\mathcal M_g})$ is not Zariski dense.

\end{proposition}
\begin{proof}
    Let $j$ be the Torelli map of $\overline{\mathcal M_g}$ into $\mathcal A_g$ as above, which is known to be generically finite; at least injective on $\mathcal M_g$.  By Theorem~\ref{thm-very-gen-rig}, any non-rigid locus is contained in some non-rigid special subvariety of $\overline{\mathcal M_g}$, denoted by $j^{-1}(\Gamma_1\backslash \Dcal_1\times \Gamma_2\backslash \Dcal_2)$.

    Assume dim$(\Dcal_1)\leq$ dim$(\Dcal_2)$. Hence dim$(\Dcal_1)\leq$ $\frac{1}{2}$dim$(\mathcal A_g)$. Then we know any weakly special subvariety of the form $j^{-1}(\Gamma_1\backslash \Dcal_1\times \{p\})$ is monodromically atypical, where $p$ is a Hodge generic point in $\Gamma_2\backslash \Dcal_2$. In fact,
     \begin{align*}&\text{codim}_{\mathcal A_g}(\overline{\mathcal M_g})+\text{codim}_{\mathcal A_g}(\Gamma_1\backslash \Dcal_1)-\text{codim}_{\mathcal A_g}(j^{-1}(\Gamma_1\backslash \Dcal_1\times \{p\}))\\
    \geq &\text{dim}(\mathcal A_g)-\text{dim}(\mathcal M_g)-\text{dim}(\Dcal_1)\geq \frac{1}{2}\text{dim}(\mathcal A_g)-\text{dim}(\mathcal M_g)>0.\end{align*}

    Suppose all nonrigid loci are Zariski dense in $\overline{\mathcal M_g}$. Then we have a dense set of monodromically atypical weakly special subvarieties in $\overline{\mathcal M_g}$.

    Since the generic Mumford-Tate group of $\overline{\mathcal M_g}$ is $\mathbb Q$-simple, by Theorem \ref{GZP}, the set of monodromically atypical weakly special subvarieties cannot be Zariski dense in $\overline{\mathcal M_g}$. Thus $\mathrm{NRL}(\overline{\mathcal M_g})$ is not Zariski dense.
\end{proof}

\subsubsection{Generic rigidity property of $\mathcal M_{d,n}$}\label{finiteness_hypersur}

Let $\mathcal M_{d,n}$ be the moduli space of smooth degree $d$ hypersurfaces in $\mathbb P^{n+1}$. We shall explain that our main theorem applies here for almost all $d$ and $n$. Firstly, the universal family $f_{d,n}:\,\mathcal{X} \to \mathcal{M}_{d,n}$ carrying a polarized $\mathbb{Z}$-VHS satisfying the associated period map
        $ \varphi_{d,n}:\,\mathcal{M}_{d,n} \to \Gamma \backslash \mathcal{D}$ is locally injective except $n=2,d=3$ (cf. \cite{Griff68} and \cite{Green}). Next, let $\mathbf{MT}_{d,n}$ be the generic Mumford-Tate group of $\varphi_{d,n}$. Let $\Gbf_{d,n}$ be the automorphism group of $H^n(X,\mathbb Q)$ preserving the cup product for a generic degree $d$ hypersurface $X$ in $\mathbb P^{n+1}$. Thus $\Gbf_{d,n}$ is symplectic or orthogonal depending on the parity of $n$ and it is absolutely simple. Let $\Hbf_{d,n}$ be the algebraic monodromy group of $\varphi_{d,n}$. Clearly, $\Hbf_{d,n}\subset \mathbf{MT}_{d,n}^{\ad}\subset \Gbf_{d,n}$. In,fact, Deligne (\cite[Proposition 5.3 and Theorem 5.4]{Del}) and Beauville (\cite[Theorem 2 and Theorem 4]{Beauville}) proved that $\Hbf_{d,n}=\Gbf_{d,n}$. Thus $\mathbf{MT}^{\der}_{d,n}$ is absolutely simple. For $n=2$ and $d\geq 5$, one sees the level of $\varphi_{d,n}$ is $2$. For $n\geq3$ and $d\geq 6$, the level of $\varphi_{d,n}$ is at least $3$. Thus Theorem~\ref{thm-generic-rig} applies to $\mathcal M_{d,n}$:
\begin{proposition}
    For $n=2$, $d\geq 5$, or for $n\geq3$, $d\geq 6$, the moduli space $\mathcal M_{d,n}$ has the generic rigidity property.
\end{proposition}

\section{Zariski closures of non-rigid loci}\label{sec-4}
This section addresses Question (B) from the introduction, aiming to understand the structure of the Zariski closure of the non-rigid locus $\mathrm{NRL}(\mathcal{M}_h)$ in the moduli space $M_h$. We begin by introducing the notion of maximal non-rigid families (Definition~\ref{max_nr}) and develop the bi-Hom construction (Section~\ref{Bi-Hom}) as a systematic method to produce such families. This construction relies on the existence of suitable projective compactifications; we therefore review the Baily-Borel compactification for Shimura varieties and its generalization to period images via the recent work of Bakker-Filipazzi-Mauri-Tsimerman (Theorem~\ref{BB_comp}), which provides the necessary extension theorems and functoriality properties. Using these tools, we construct the moduli space of morphisms into $M$ and prove that the bi-evaluation morphism yields a maximal non-rigid family (Proposition~\ref{produce-max-non-rig}).

Parallel to this geometric construction, we study non-rigid Hodge subvarieties (Definition~\ref{definition special subvariety of rigid non-rigid type}) and characterize them as factors in a product of Hodge data (Lemma~\ref{lemma special subvarieties of factor type}). This group-theoretic perspective leads to the interpretation of maximal product subvarieties $Y_1\times Y_2\subset M$ as structurally-atypical intersections of $M$ with product period subdomains (Section~\ref{subsec-satyp}). Inspired by the Zilber-Pink philosophy, we formulate a finiteness conjecture for such intersections (Conjecture~\ref{finiteness_atyp}) and, assuming it, deduce a structural description of $\overline{\mathrm{NRL}(\mathcal{M}_h)}^{\mathrm{Zar}}$ as a finite union of two types of components: those dominated by fiber spaces whose generic fiber is a product variety, and those birational to Shimura varieties of rank $\ge 2$ (Conjecture~\ref{Zar_closure}). 

The remainder of the section is devoted to providing {\it unconditional} evidence for this conjectural picture. We prove that under the hypothesis that each bi-Hom scheme is special, an irreducible component $Z$ of $\overline{\mathrm{NRL}(\mathcal{M}_h)}^{\mathrm{Zar}}$ must satisfy a dichotomy: either its algebraic monodromy group is $\mathbb Q$-simple, forcing $Z$ to be birational to a Shimura variety of Hodge level $1$, or it admits a non-trivial fibration whose fibres contain all special non-rigid loci (Theorem~\ref{thm:snrl}). This result, which does not rely on the finiteness conjecture, already reveals the structure imposed by Hodge theory on the Zariski closure of the non-rigid locus and motivates the investigation on Calabi-Yau moduli spaces (cf. Section~\ref{sec-5}).

\subsection{Maximal non-rigid families}
We have particular interests in the ``maximal elements'' in $\mathrm{NRL}(\mathcal{M}_h)$:
\begin{definition}\label{max_nr}
  A non-rigid family $V \to U$ in $\mathcal{M}_h(U)$ is said to be {\it maximal}, if for its classifying morphism $\psi:\, U \to \mathcal{M}_h$, there exists a scheme $T$ and an extension of $\psi$,
  \[
  \Psi:\, U \times T \to \mathcal{M}_h,
  \]
  which is generically finite onto its image, such that the following property holds: for any factorization
  \[
    \xymatrix{
   U \times T \ar[r]^-{\Psi} \ar[d]^{(\iota_U,\iota_T)} & \mathcal{M}_h. \\
   U' \times T'  \ar@/_1pc/[ru]_-{\Psi'} &
    }
  \]
preserving the product structures (i.e. $\iota_U:\,U \to U'$ and $\iota_T:\,T \to T'$), one has that $\iota_U:\,U \to U'$ and $\iota_T:\,T \to T'$ are {\it birational morphisms}.
\end{definition}

For simplicity, we assume that $\mathcal{M}_h$ is a {\it fine moduli space} hereafter.

\subsection{Bi-Hom construction}\label{Bi-Hom}
The natural way of constructing maximal non-rigid families in $\mathcal{M}_h$ is to apply the Hom functor to classifying morphisms twice, and the maximality follows from the universal property of Hom functor.

However, most moduli spaces are {\it non-complete}. One has to compactify $M_h$ and deal with the boundary so as to employ Hom scheme. Consequently, it becomes important to choose a suitable projective compactification of $M_h$.

\subsubsection{Baily-Borel compactification}

\paragraph{Shimurian case}
When the moduli space is a Shimura variety, there is a canonical compactification. Let $\Omega:= \mathbf{G}(\mathbb{R})/K$ be a bounded symmetric domain, and let $\Gamma \subset \mathbf{G}(\mathbb{R})$ be an arithmetic subgroup. Then the quotient space, $X_{\Gamma}:= \Gamma \backslash \Omega$, admits a structure of quasi-projective variety.

By the work of Satake, Baily and Borel, the Shimura variety $X_{\Gamma}$ has a canonical projective compactification ${X_{\Gamma}}^{\mathrm{BB}}$, the {\it Baily-Borel compactification}. When $X_{\Gamma}=\mathcal{A}_g$, there is an explicit set-theoretical description of ${\mathcal{A}_g}^{\mathrm{BB}}$:
\[
{\mathcal{A}_g}^{\mathrm{BB}}= \mathcal{A}_g \sqcup \mathcal{A}_{g-1} \sqcup \cdots \sqcup \mathcal{A}_1 \sqcup \mathcal{A}_0,
\]
where $\mathcal{A}_1$ is the modular curve and $\mathcal{A}_0=\{\infty\}$.

The Baily-Borel compactification ${X_{\Gamma}}^{\mathrm{BB}}$ enjoys various nice properties. For instance, ${X_{\Gamma}}^{\mathrm{BB}}$ is the $\mathrm{Proj}$ of the graded ring of automorphic forms associated to $\Omega$ and $\Gamma$.
The most important property related to our construction is the following {\it extension property}, established by Borel in \cite{Borel72}.

\begin{theorem}[Borel Extension Theorem]
Let $X_{\Gamma}$ be a Shimura variety as above. Suppose that $\Gamma$ is torsion-free.
Then for every log smooth pair $(Z,D_Z)$, every holomorphic map $\varphi:\,Z^o:= Z \setminus D_Z \to X_{\Gamma}$ can be extended to a holomorphic map
\[
\bar{\varphi}:\, Z \longrightarrow {X_{\Gamma}}^{\mathrm{BB}}
\]
mapping into the Baily-Borel compactification.
\end{theorem}

\begin{remark}
Borel actually proved the above theorem in a more general setting by allowing $Z^o$ to be a punctured polydiscs $(\mathbb{D}^*)^a \times \mathbb{D}^b$. See \cite{DLSZ} for a generalization of Borel's extension theorem in the context of moduli spaces of polarized varieties.
\end{remark}

\paragraph{Griffiths Conjecture}
Most moduli spaces are not Shimurian. On the other hand, many moduli spaces have period maps mapping into (quotients of) period domains, i.e. homogeneous complex manifolds classifying polarized Hodge structures with given type.

Let $U$ be a quasi-projective variety. Let $\mathbb{V}=(V_{\mathbb{Z}},V_{\mathcal{O}},\nabla, \mathrm{F}^{\cdot}V_{\mathcal{O}},Q)$ be a $\mathbb{Z}$-PVHS over $U$. Here $V_{\mathbb{Z}}$ is the underlying $\mathbb{Z}$-local system, $(V_{\mathcal{O}},\nabla)$ is the corresponding flat bundle, $\mathrm{F}^{\cdot}V_{\mathcal{O}}$ is the Hodge filtration and $Q$ is the polarization.
These datum give a period map
\[
\varphi:\, U \to \Gamma \backslash \mathcal{D},
\]
where $\mathcal{D}$ is the corresponding period domain and $\Gamma$ is some arithmetic subgroup containing the monodromy group of $V_{\mathbb{Z}}$.

By a theorem of Griffiths, one could enlarge $U$ by some quasi-projective variety $U' \supset U$, and extend the period map as a {\it proper} holomorphic map $\varphi'$ from $U'$ mapping into $\Gamma \backslash \mathcal{D}$. Moreover, the image of $\varphi'$ is a closed complex analytic subvariety of $\Gamma \backslash \mathcal{D}$.

\begin{definition}[Period image]
  We call this closed complex analytic subvariety $\mathlcal{P}:=\mathrm{Im}(\varphi')$ the {\it period image} of $\varphi$. In particular, one has the following factorization
  \[
    \xymatrix{
 U \ar[rd]_{\varphi} \ar[rr]^{\varphi} & & \Gamma \backslash \mathcal{D} \\
   & \mathlcal{P}  \ar@{^{(}->}[ru]
    }
  \]
of the period map $\varphi$.
\end{definition}

Inspired by the Baily-Borel compactification of Shimura varieties, Griffiths in 1970s made the following influential conjecture:
\begin{conj*}[Griffiths]
  \begin{itemize}
    \item[(1)] The period image $\mathlcal{P}$ is a quasi-projective variety, and the period map $\varphi$ is induced from an algebraic morphism.
    \item[(2)] The Griffiths line bundle $L:= \bigotimes_p \mathrm{det}(\mathrm{F}^pV_{\mathcal{O}})$ is ample on $\mathlcal{P}$.
    \item[(3)] There exists a Baily-Borel type compactification $\mathlcal{P}^{\mathrm{BB}}$ of $\mathlcal{P}$, such that $\mathlcal{P}^{\mathrm{BB}}$ is the $\mathrm{Proj}$ of the ring of moderate growth sections of $L^{\otimes n}$.
  \end{itemize}
\end{conj*}
Note that (1) and (2) of the Griffiths conjecture has been proved in \cite{BBT-23}. Very recently, the Baily-Borel type compactification conjectured in (3) was constructed by Bakker-Filipazzi-Mauri-Tsimerman in \cite{BFMT25}.

\begin{theorem*}[Bakker-Filipazzi-Mauri-Tsimerman]\label{BB_comp}
  Let $\mathlcal{P}$ be a period image. Then there exists a functorial projective compactification $\mathlcal{P}^{\mathrm{BB}}$ such that the following properties hold:
  \begin{itemize}
    \item[(1)] For any log smooth algebraic space $(Z,D_Z)$, any morphism $Z \setminus D_Z \to \mathlcal{P}$ for which the induced period map is locally liftable extends to a morphism $Z \to \mathlcal{P}^{\mathrm{BB}}$.
    \item[(2)] $\mathlcal{P}^{\mathrm{BB}}$ is the $\mathrm{Proj}$ of the ring of moderate growth sections of the Griffiths bundle $L^{\otimes n}$, which is a finitely generated ring. Moreover, for any extension $Z \to \mathlcal{P}^{\mathrm{BB}}$, the pull-back of $L^{\otimes n}$ to $Z$ is exactly the Griffiths bundle on $Z$.
  \end{itemize}
\end{theorem*}

As a direct corollary of the above theorem, one could check the functoriality of the Baily-Borel compactification $\mathlcal{P}^{\mathrm{BB}}$:

\begin{corollary}[Functoriality]\label{funct_BB}
  \begin{itemize}
    \item[(1)] Let $\varphi_i:\,U_i \twoheadrightarrow \mathlcal{P}_i \subset \Gamma_i \backslash \mathcal{D}_i$ ($i=1,2$) be two period maps. Choose log smooth projective compactifications $(Y_i,D_i)$ of $U_i$, $i=1,2$. Denote by ${\mathlcal{P}_i}^{\mathrm{BB}}$ the Baily-Borel compactification of the corresponding period image. Then
          \[
        (\mathlcal{P}_1 \times \mathlcal{P}_2)^{\mathrm{BB}} = {\mathlcal{P}_1}^{\mathrm{BB}} \times {\mathlcal{P}_2}^{\mathrm{BB}}.
          \]
    \item[(2)] Let $\varphi:\,U \twoheadrightarrow \mathlcal{P} \subset \Gamma \backslash \mathcal{D}$ be a period map. Let $Z \subset U$ be a closed subvariety. Denote by $\varphi_Z:\,Z \twoheadrightarrow \mathlcal{P}_Z \subset \Gamma_Z \backslash \mathcal{D}_Z$ be the induced period map on $Z$, which fits into the following commutative diagram
          \[
          \xymatrix{
          Z \ar@{->>}[r]^{\varphi_Z} \ar@{^{(}->}[d]  &  \mathlcal{P}_Z \ar@{^{(}->}[d] \ar@{^{(}->}[r] & \Gamma_Z \backslash \mathcal{D}_Z \ar@{^{(}->}[d] \\
          U \ar@{->>}[r]^{\varphi} & \mathlcal{P} \ar@{^{(}->}[r] & \Gamma \backslash \mathcal{D}.
          }
          \]
          After replacing $U$ and $Z$ by some birational models $U'$ and $Z'$, one can find a log smooth projective compactification $(Y,D_Y)$ of $U'$ such that $(\overline{Z'},D_Z:= D_Y \cap \overline{Z'})$ is also log smooth (here $\overline{Z'}$ is the Zariski closure of $Z'$ in $Y$). Then ${\mathlcal{P}_Z}^{\mathrm{BB}} \subset {\mathlcal{P}}^{\mathrm{BB}}$. In particular, ${\mathlcal{P}_Z}^{\mathrm{BB}}$ is the Zariski closure of $\mathlcal{P}_Z$ in ${\mathlcal{P}}^{\mathrm{BB}}$.
  \end{itemize}
\end{corollary}

\paragraph{Baily-Borel compactification of moduli schemes}
Bakker-Filipazzi-Mauri-Tsimerman have constructed Baily-Borel type compactification of any moduli space of polarized varieties with a local Torelli theorem. They used the ring of moderate growth sections of the Griffiths bundle on moduli space to construct the projective compactification. Such strong condition will not be used in our setting, and thus we provide here an alternative construction of the ``minimal compactification'' such that the extended period map remains to be quasi-finite.

Consider the period map on the fine moduli scheme $M$ and its period image $\mathlcal{P}$. Let $\delta:\,U \to M$ be the desingularization and let $\overline{U} \supset U$ be a log smooth projective compactification. By Theorem~\ref{BB_comp}, we have the following diagram:
\[
  \xymatrix{
 U \ar[r]^{\delta} \ar@{^{(}->}[d] &   M \ar@{->>}[r]^{\varphi}  & \mathlcal{P} \ar@{^{(}->}[r] \ar@{^{(}->}[d] & \Gamma \backslash \mathcal{D} \\
    \overline{U} \ar[rr]^{{\varphi}_1} & &  {\mathlcal{P}}^{\mathrm{BB}},  &
    }
\]
where ${\varphi}_1$ is the extended period map.

Consider the Stein factorization of the proper morphism ${\varphi}_1$:
\[
\overline{U} \to \overline{U}' \xrightarrow{\bar{\varphi}} {\mathlcal{P}}^{\mathrm{BB}}.
\]
If the period map $\varphi$ is finite, then $\mathrm{Exc}(\varphi_1) \cap U$ must be contained in $\mathrm{Exc}(\delta)$. In particular, $M \setminus \Delta$ is contained in $\overline{U}'$, where $\Delta$ is the singular locus. This induces a birational map $M \dashrightarrow \overline{U}'$, and we want to show that it can be extended over $\Delta$. If it is not extendable, then there will be an exceptional divisor mapped into $\overline{U}'$. But this is impossible due to the finiteness of $\varphi$ and $\bar{\varphi}$, as well as the following commutative diagram
\[
  \xymatrix{
    M \ar@{->>}[r]^{\varphi} \ar@{-->}[d]  & \mathlcal{P} \ar@{^{(}->}[r] \ar@{^{(}->}[d] & \Gamma \backslash \mathcal{D} \\
    \overline{U}' \ar[r]^{\bar{\varphi}}  &  {\mathlcal{P}}^{\mathrm{BB}}.  &
    }
\]
Consequently, we know that $\overline{U}'$ is also a projective compactification of $M$. Note that the extended period map $\bar{\varphi}$ on $\overline{U}'$ is a {\it finite} morphism. By abuse of notions, we still call $\overline{U}'$ the {\it Baily-Borel compactification} of $M$, and we denote it as $M^{\mathrm{BB}}$.

Note that ${M}^{\mathrm{BB}}$ could be highly singular, and the boundary $D_{\infty}:= {M}^{\mathrm{BB}} \setminus M$ could be a closed subvariety of higher codimension.

\subsubsection{Extension theorem}
Let $M$ be a fine moduli space with a {\it quasi-finite} period map $\varphi:\, M \twoheadrightarrow \mathlcal{P} \subset \Gamma \backslash \mathcal{D}$. Let $M^{\mathrm{BB}}$ be the Baily-Borel compactification of $M$ constructed in the previous section. In particular, the extended period map $\bar{\varphi}:\,M^{\mathrm{BB}} \to \mathlcal{P}^{\mathrm{BB}}$ is also quasi-finite.

\begin{theorem}[Extension Theorem]\label{ext_thm}
  Let $U$ and $H$ be two smooth quasi-projective varieties. Let $\Phi:\,U \times H \to M$ be a morphism. Then for any smooth projective compactifications $\overline{U} \supset U$ and $\overline{H} \supset H$ with normal crossing boundaries, $\Phi$ can be extended to a morphism $\bar{\Phi}:\,\overline{U} \times \overline{H} \to {M}^{\mathrm{BB}}$.
\end{theorem}

\begin{proof}
  Let $\widehat{\overline{U} \times \overline{H}}$ be the resolution of indeterminacy of the rational map $\overline{U} \times \overline{H} \dashrightarrow M^{\mathrm{BB}}$ induced by $\Phi$. If the indeterminacy locus is nonempty, one can reduce the general case to the case that both $U$ and $H$ are curves. We will obtain a contradiction in this case. Denote by $\delta:\, \widehat{\overline{U} \times \overline{H}} \to \overline{U} \times \overline{H}$ the blowing up of indeterminacy points and let $\Psi:\, \widehat{\overline{U} \times \overline{H}} \to M^{\mathrm{BB}}$ be the morphism.

After some finite \'etale base change, one has the following factorization
\[
  \xymatrix{
U \times H \ar[r]^-{\Phi} \ar[d]_{(\varphi_U,\varphi_H)}  & M \ar[d]^{\varphi} \\
\Gamma_U \backslash \mathcal{D}_U \times \Gamma_H \backslash \mathcal{D}_H \ar@{^{(}->}[r] & \Gamma \backslash \mathcal{D},
  }
\]
where $\mathcal{D}_U$, $\mathcal{D}_H$ are the corresponding period subdomains.

Let $\mathlcal{P}_U$ (resp. $\mathlcal{P}_U$) be the period image of $\varphi_U$ (resp. $\varphi_H$). By virtue of Corollary~\ref{funct_BB}, we have the following commutative diagram
\small
\begin{align}
  \label{big-comm-extension}
    \xymatrix@C=4em{
U \times H \ar[ddd]_{(\varphi_U, \varphi_H)} \ar[rrrr]^{\Phi} \ar@{^{(}->}[rdd]  & & & & M \ar@{_{(}->}[ldd] \ar[ddd]^{\varphi} \\
   & & \widehat{\overline{U} \times \overline{H}} \ar[ld]_{\delta} \ar[rd]^{\Psi}  & &  \\
   & \overline{U} \times \overline{H}  \ar@{-->}[rr]^{\Phi} \ar[dd]_{(\overline{\varphi_U}, \overline{\varphi_H})} & & {M}^{\mathrm{BB}} \ar[dd]^{\bar{\varphi}} & \\
   \mathlcal{P}_U \times \mathlcal{P}_H \ar[rrrr] \ar@{^{(}->}[rd] & & & & \mathlcal{P} \ar@{_{(}->}[ld] \\
   & {\mathlcal{P}_U}^{\mathrm{BB}} \times {\mathlcal{P}_H}^{\mathrm{BB}} \ar@{^{(}->}[rr] & & {\mathlcal{P}}^{\mathrm{BB}}  &
    }
  \end{align}
\normalsize

Let $E$ be the exceptional divisor of $\delta$. If the rational map $\Phi:\,\overline{U} \times \overline{H} \dashrightarrow M^{\mathrm{BB}}$ is non-extendable, then $\Psi(E)$ is a positive dimensional subvariety in $M^{\mathrm{BB}}$.
Since $\bar{\varphi}:\, M^{\mathrm{BB}} \to \mathlcal{P}^{\mathrm{BB}}$ is quasi-finite, we know that $\bar{\varphi}(\Psi(E))$ has positive dimension.

On the other hand, $\delta(E)$ are indeterminacy points and thus will be mapped to points in ${\mathlcal{P}_U}^{\mathrm{BB}} \times {\mathlcal{P}_H}^{\mathrm{BB}}$. This gives a contradiction to the diagram \eqref{big-comm-extension}.
\end{proof}

\paragraph{Moduli space of morphisms into $M$}
To use the Hom scheme to parametrize morphisms, one usually needs to compactify the target space first. Since now we have the canonical projective compactification $M^{\mathrm{BB}}$ of the moduli space $M$, it is possible to define the moduli space morphisms into $M$ without referring to a particular chosen compactification. We follow the argument in \cite{Nogu} closely in this subsection.

Let $U$ be a smooth quasi-projective variety with a log smooth projective compactification $\overline{U}$. Let $\mathrm{Mor}(U,M)$ be the set of morphisms from $U$ to $M$. Let $\mathrm{Hom}(\overline{U},M^{\mathrm{BB}})$ be the usual Hom scheme. Then by Theorem~\ref{ext_thm}, one can regard $\mathrm{Mor}(U,M)$ as a subset of the Hom scheme $\mathrm{Hom}(\overline{U},M^{\mathrm{BB}})$.

\begin{theorem}
  \begin{itemize}
    \item[(i)] $\mathrm{Mor}(U,M)$ is a countable union of quasi-projective schemes, and there is a universal evaluation morphism
          \[
          U \times \mathrm{Mor}(U,M) \to M.
          \]
    \item[(ii)](Universal property) Let $T$ be a scheme and let $\varphi_T:\,U \times T \to M$ be a morphism. Then there exists a classifying morphism $T \to \mathrm{Mor}(U,M)$ such that $\varphi_T$ factors through the universal evaluation morphism.
  \end{itemize}
\end{theorem}

\begin{proof}
  Let $\varphi \in \mathrm{Mor}(U,M)$ and let $\bar{\varphi} \in \mathrm{Hom}(\overline{U},M^{\mathrm{BB}})$ be the extension. Let $Z$ be the irreducible component of $\mathrm{Hom}(\overline{U},M^{\mathrm{BB}})$ containing $\bar{\varphi}$. Then $\bar{\varphi}$ can be deformed to a family of morphisms
  \[
 \Phi:\, \overline{U} \times Z \to M^{\mathrm{BB}}.
  \]
  For each $z \in Z$, we denote by $\Phi_z$ the restriction of $\Phi$ on $\overline{U} \times \{z\}$, which is a proper morphism. Thus $\Phi^{-1}_z(D_{\infty})$ is a closed subscheme in $\overline{U}$, and $\Phi^{-1}(D_{\infty}) \to Z$ is a family of closed subschemes in $\overline{U}$. This induces a morphism
  \[
 Z \to \mathrm{Hilb}(\overline{U}),
  \]
  which sends $z$ to the class $[\Phi^{-1}_z(D_{\infty})]$.

  Denote by $D_U:= \overline{U} \setminus U$ the boundary divisor. Then $\mathrm{Hilb}(D_U)$ can be regarded as a closed subscheme of $\mathrm{Hilb}(\overline{U})$. Define $Z_U$ as the following fiber product:
  \[
    \xymatrix{
Z_U \ar@{^{(}->}[r] \ar[d]  & Z \ar[d] \\
\mathrm{Hilb}(D_U)  \ar@{^{(}->}[r] & \mathrm{Hilb}(\overline{U}).
    }
  \]
  Then $Z_U$ is a countable union of quasi-projective schemes, and for each $z \in Z_U$ one has $\Phi_z(U) \subset M$. Thus
  \[
\Phi|_{U \times Z_U}: \, U \times Z_U \to M.
  \]

  Note that $Z_U$ is nonempty since $\varphi \in Z_U$. It is clear that $\mathrm{Mor}(U,M)$ is a countable union of such $Z_U$.

  Now we check the universal property (ii). Without loss of generality, we assume that $T$ is irreducible. By Theorem~\ref{ext_thm} we can extend $\varphi_T$ and thus get a morphism $T \to \mathrm{Hom}(\overline{U},M^{\mathrm{BB}})$. It is then clear that $T \to \mathrm{Hom}(\overline{U},M^{\mathrm{BB}})$ factors through some $Z_U$.
\end{proof}

\subsubsection{Construction of bi-Hom schemes}
For a non-rigid morphism $\varphi:\, U \to M$, one can construct a maximal non-rigid morphism $\Phi:\,H_1 \times H_2 \to M$ in the following way.

Note that $[\varphi] \in \mathrm{Mor}(U,M)$. Since $\varphi$ is non-rigid, we know that the irreducible component $Z_1$ of $\mathrm{Mor}(U,M)$ containing $[\varphi]$ has positive dimension. Let $H_1$ be the desingularization of $Z_1$. We also get the induced evaluation morphism
\[
\varphi_1:\, U \times H_1 \to M.
\]
Choose a general point $u \in U$. Then the restricted morphism
\[
\psi:=\varphi_1|_{\{u\} \times H_1} \to M
\]
is again non-rigid. Let $Z_2 \subset \mathrm{Mor}(H_1,M)$ be the corresponding irreducible component, and let $H_2$ be the desingularization of $Z_2$. This gives the bi-evaluation morphism
\begin{align}
  \label{bi-eva}
\Phi:\, H_1 \times H_2 \to M.
\end{align}

Before we check that $\Phi:\, H_1 \times H_2 \to M$ is {\it maximal} in the sense of Definition~\ref{max_nr}, we will first recall the notion of {\it 1-pointed rigidity}.

\paragraph{One-pointed rigidity}
Let $U$ be a smooth quasi-projective variety as before. Fix marked points $u \in U$ and $x \in M$. Define the {\it moduli space of 1-pointed morphisms}
\[
\mathrm{Mor}\left( (U,u),(M,x) \right):=\{ \varphi \in \mathrm{Mor}(U,M) \,:\,\varphi(u)=x\}
\]
which is a subscheme of $\mathrm{Mor}(U,M)$.

\begin{definition}\label{1-pointed-rigidity}
The moduli space $M$ is said to have \textbf{1-pointed rigidity} property if for any pointed scheme $(U,u)$ and any marked point $x \in M$, every irreducible component of $\mathrm{Mor}\left( (U,u),(M,x) \right)$ is a zero-dimensional scheme.
\end{definition}

\begin{theorem}[Dimension bound]\label{dim-bound}
  Suppose that $M$ has the 1-pointed rigidity property, and $M$ is of log general type. Then for each irreducible component $H$ of $\mathrm{Mor}(U,M)$, the following hold:
  \begin{itemize}
    \item[(i)]  For each $u \in U$,the restriction of the evaluation map on $H \times \{u\}$ is {\it generically finite} onto its image;
    \item[(ii)] $\mathrm{dim}\,H \leq \mathrm{dim}\,M-1$.
  \end{itemize}
\end{theorem}
\begin{proof}
When $\mathrm{dim}\,U=1$ this has been proved in {\cite[Theorem~6.4]{JLSZ}}. The general case can be reduced to the curve case by taking some complete intersection curve in $U$.
\end{proof}

\paragraph{Maximality of $\Phi:\, H_1 \times H_2 \to M$}
Now we come back to the construction of maximal non-rigid families in $M$.
\begin{proposition}\label{produce-max-non-rig}
The bi-evaluation morphism $\Phi:\, H_1 \times H_2 \to M$ constructed in \eqref{bi-eva} is maximal in the sense of Definition~\ref{max_nr}.
\end{proposition}

\begin{proof}
  Recall that $\Phi$ is constructed from a non-rigid morphism $\varphi:\,U \to M$, and $H_1$ is the desingularization of some irreducible component of $\mathrm{Mor}(U,M)$. Then $U \to Z_2 \subset \mathrm{Mor}(H_1,M)$.
  One has the following commutative diagram
  \[
    \xymatrix{
  H_1 \times U' \ar[r] \ar[d] & M \ar@{=}[d] \\
  H_1 \times H_2 \ar[r] & M,
    }
  \]
where $H_2$ is the desingularization of $Z_2$ and $U'$ is the fiber product $U \times_{Z_2}H_2$.

By applying the same process to $H_2 \times H_1 \to M$, one obtains some irreducible component $Z_3 \subset \mathrm{Mor}(H_2,M)$ with the classifying morphism $H_1 \to Z_3$, and the following diagram
\begin{align}
  \label{squ_H_123}
    \xymatrix{
  H_2 \times H_1 \ar[r] \ar[d] & M \ar@{=}[d] \\
  H_2 \times Z_3 \ar[r] & M.
    }
\end{align}
  We will show that the natural morphism $H_1 \to Z_3$ is the identity map. Then the maximality of $\Phi$ follows from the universal property of $\mathrm{Mor}(\bullet,M)$.

  Choose a general point $u \in U$. Then by Theorem~\ref{dim-bound} (i) we know that the restricted morphism $H_1 \times \{u\} \to M$ is generically finite onto its image. Let $h_2$ be the image of $u$ in $H_2$. Then from \eqref{squ_H_123} we have the factorization
  \[
\{h_2\} \times H_1 \to \{h_2\} \times Z_3 \to M.
  \]
  Thus $H_1 \to Z_3$ is generically finite onto its image.

  On the other hand, the natural map $U' \to H_2$ induces the following morphism
  \begin{align}
    \label{pull_back}
 \mathrm{Mor}(H_2, M) \to \mathrm{Mor}(U', M).
  \end{align}
Since $M$ is hyperbolic and $U$ is smooth, we know that $\mathrm{Mor}(U', M) = \mathrm{Mor}(U, M)$, and thus \eqref{pull_back} gives a section of $H_1 \to Z_3$. Since both $H_1$ and $Z_3$ are irreducible, we know that $H_1=Z_3$.
\end{proof}

\subsection{Non-rigid Hodge subvarieties}
The purpose of this subsection is to construct the maximal non-rigid objects in the category of Hodge subvariety, in the sense of Definition~\ref{max_nr}.

As we have seen in Definition~\ref{definitioin weakly special subvarieties}, a weakly special subvariety which is NOT special necessarily arise in a ``family'', whose members behave like a factor in a Hodge subvariety of product type.

\begin{definition}\label{definition special subvariety of rigid non-rigid type} Let $(\Gbf,\Dcal)$ be a Hodge datum with associated Hodge variety $M=\Gamma\bsh\Dcal$, and let $(\Gbf',\Dcal')$ be a Hodge subdatum in $(\Gbf,\Dcal)$.

	(1) We say that $(\Gbf',\Dcal')$ in $(\Gbf,\Dcal)$ is rigid (or of non-factor type) if the centralizer $\Zbf_{\Gbf^\der}\Gbf'^\der$ is compact modulo center, i.e. does not contain any non-trivial non-compact normal $\Qbb$-subgroups. Otherwise it is said to be non-rigid (or of factor type).

	(2) A Hodge subvariety in $M$ is said to be rigid resp. non-rigid it it is defined by a Hodge subdatum in $(\Gbf,\Dcal)$ which is rigid resp. non-rigid.


\end{definition}

Here we prefer the terminology ``(non-)rigid'' Hodge subvariety in accordance with the notions of rigidity introduced in Section \ref{sec-2}. The more traditional terminology ``of (non-)factor type'' is already used in \cite{Ull}, which leads to the following characterization:

\begin{lemma}\label{lemma special subvarieties of factor type}
	Let $M=\Gamma\bsh\Dcal$ be a Hodge variety associated to some Hodge datum $(\Gbf,\Dcal)$, and let $M_0$ be a Hodge subvariety associated to some Hodge subdatum $(\Gbf_0,\Dcal_0)$. Then $M_0$ is non-rigid if and only if there exists a morphism of Hodge data $f: (\Gbf_1,\Dcal_1)\times(\Gbf_2,\Dcal_2)\ra(\Gbf,\Dcal)$
	such that \begin{itemize}
		\item $f: \Gbf_1^\der\times\Gbf_2^\der\ra \Gbf^\der$ is of finite kernel and $f_*: \Dcal_1\times\Dcal_2\ra \Dcal$ is finite onto its image, with $\Dcal_1$ and $\Dcal_2$ both of strictly positive dimension;
		\item there exists some CM subdatum $(\Tbf, h)$ in $(\Gbf_2,\Dcal_2)$ and some congruence subgroups $\Gamma'\subset(\Gbf_1\times\Gbf_2)(\Qbb)^+$ with $\Gamma''=\Gamma'\cap(\Gbf_1\times\Tbf)(\Qbb)^+$, so that the image of the morphism $(\Gbf_1,\Dcal_1)\times(\Tbf,h)\ra(\Gbf,\Dcal)$ is a Hodge subdatum in $(\Gbf,\Dcal)$ defining $M_0$.
	\end{itemize}

\end{lemma}
The idea of this lemma is the same as the case of Shimura varieties treated in \cite[Section 3]{Ull}, and we only present an outline.
\begin{proof}[Sketch of the proof]

	``$\Leftarrow$'': When $f: (\Gbf_1,\Dcal_1)\times(\Gbf_2,\Dcal_2)\ra(\Gbf,\Dcal)$ and $(\Tbf,h)\subset(\Gbf_2,\Dcal_2)$ are given as in the statement, the image $(\Gbf_0,\Dcal_0)$ of $(\Gbf_1,\Dcal_1)\times(\Tbf,h)$ is the subdatum in $(\Gbf,\Dcal)$ defining $M_0$. Since $\Gbf_1^\der\times\Gbf_2^\der\ra\Gbf^\der$ is of finite kernel, the $\Qbb$-subgroup $\Gbf_0^\der$ equals the image   $f(\Gbf_1^\der\times\{e_{\Gbf_2}\})$ and is centralized by the non-commutative $\Qbb$-group $f(\{e_{\Gbf_1}\}\times\Gbf_2^\der)$ whose associated Lie group is non-compact. Thus $M_0$ is non-rigid.\bigskip

	``$\Rightarrow$'': We are reduced to the case $\Gbf=\Gbf^\ad=\Gbf^\der$. The neutral component $\Nbf$ of the normalizer $\Nbf_\Gbf(\Gbf_0^\der)$ is a connected reductive $\Qbb$-subgroup of $\Gbf$ containing $\Gbf_0$, and admits a decomposition into almost direct product: $$\Nbf=\Cbf\cdot\Hbf_{c}\cdot \Hbf_{nc}$$ where \begin{itemize}
		\item $\Cbf$ is the connected center of $\Nbf^\circ$;
		\item $\Hbf_c$ is generated by simple $\Qbb$-factors in $\Nbf^\der$ whose associated Lie groups are compact;
		\item $\Hbf_{nc}$ is generated by simple $\Qbb$-factors in $\Nbf^\der$ whose associated Lie groups are non-compact.
	\end{itemize}

	Note that $\Gbf_0^\der$ is a normal $\Qbb$-subgroup strictly contained in $\Hbf_{nc}$, so we have a further decomposition into an almost direct product $\Hbf_{nc}=\Gbf_0^\der\Hbf'$ and $\Gbf_0$ is contained in $\Cbf\cdot\Hbf_{nc}$. Take a point $h_0\in\Dcal_0$ whose image is contained in a $\Qbb$-torus $\Tbf$ of $\Gbf_0$, and thus we may write $\Tbf=\Tbf_\Cbf\cdot\Tbf'\cdot\Tbf_0$ where \begin{itemize}
		\item $\Tbf_\Cbf$ is the neutral component of $\Tbf\cap\Cbf$;
		\item $\Tbf'$ is the neutral component of $\Tbf\cap\Hbf'$;
		\item $\Tbf_0$ is the neutral component of $\Tbf\cap\Gbf_0^\der$.
	\end{itemize} We thus put $\Gbf_2$ to be the $\Qbb$-subgroup in $\Nbf$ generated by $\Hbf'$ and $\Tbf$, and let $\Dcal_2$ be the $\Gbf_2(\Rbb)^+$-conjugacy class of $h$. Then $(\Gbf_2,\Dcal_2)$ is a Hodge subdatum itself inside $(\Gbf,\Dcal)$, with $\Gbf_2^\der=\Hbf'$ centralizing $\Gbf_0^\der$ in $\Gbf$.

	The two subdata $(\Gbf_0,\Dcal_0)$ and $(\Gbf_2,\Dcal_2)$ inside $(\Gbf,\Dcal)$ produce a map of Hodge data $$f: (\Gbf_0,\Dcal_0)\times(\Gbf_2,\Dcal_2)\ra(\Gbf,\Dcal)$$ so that the $\Qbb$-group homomorphism $\Gbf_0^\der\times\Gbf_2^\der\ra\Gbf^\der$ is injective of image $\Hbf_{nc}=\Gbf_0^\der\Hbf'$. The image of $(\Gbf_0,\Dcal_0)\times(\Tbf,h)$ in $(\Gbf,\Dcal)$ is simply $(\Gbf_0,\Dcal_0)$ defining $M_0$.
\end{proof}

\begin{remark}\label{remark special subvariety of factor type}

	Note that in the lemma above one may move $M_0$ into subvarieties uniformized by $\Dcal_1\times\{h'\}$ with $h'$ moving in $\Dcal_2$ not necessarily corresponding to a subdatum of CM type. This is exactly the description of totally geodesic subvarieties in Shimura varieties characterized in \cite{Moo} and they are equivalently known as weakly special subvarieties.
\end{remark}

The lemma above shows that a non-rigid Hodge subvariety $M_0$ in $M=\Gamma\bsh\Dcal$ defined by $(\Gbf_0,\Dcal_0)$ is realized, after suitably shrinking $\Gamma$, as a factor in a product: there exists an injective map of Hodge varieties $M_1\times M_2\ra M$ induced by $(\Gbf_1,\Dcal_1)\times(\Gbf_2,\Dcal_2)\ra (\Gbf,\Dcal)$ at the level of Hodge data, such that $M_0\isom M_1\times\{x_2\}$ for some special point $x_2\in M_2$, and $\Dcal_1\times\Dcal_2\ra\Dcal$ is finite onto its image. When $x_2$ moves along $M_2$, the resulting weakly special subvarieties in $M$ are all isomorphic to $M_1$ (and to $M_0$), and we obtain an injective map $\iota_{x_2}: M_1\ra M$ of complex analytic varieties, which is compatible with the group-theoretic data, in the rough sense that it lifts to an embedding $\lambda_{x_2}:\Dcal_1\ra\Dcal$ equivariant with respect to the Lie group homomorpism $\Gbf_1^\der(\Rbb)^+\ra\Gbf^\der(\Rbb)^+$. For the moment we write $\Hom(M_1,M)$ for the space of such embeddings, and endow it with the subspace topology from the Hom space between two complex analytic varieties. Then the natural map $M_2\ra\Hom(M_1,M)$ is well-defined as a continuous embedding. We may thus think of $M_2$ in an ad hoc way as a space of ``good'' embeddings from $M_1$ to $M$.

Since $M_2$ is itself a Hodge variety and can be embedded into $M$ by the choice of base points in $M_1$, we can continue the same group-theoretic procedure. Choose a special point $x_1$ in $\Dcal_1$, we can form a Hodge datum $(\Tbf_1,\{x_1\})\times(\Gbf_2,\Dcal_2)$, whose image in $(\Gbf,\Dcal)$ is a subdatum $(\Gbf',\Dcal')$ inducing a Hodge subvariety isomorphic to $M_2$. The non-compact $\Qbb$-factors in $\Zbf_{\Gbf}(\Gbf'^\der)$ generate a $\Qbb$-subgroup $\Hbf_3$, and we obtain further a subdatum $(\Gbf'',\Dcal'')$ such that $\Gbf''^\der=\Hbf_3\Gbf'^\der$, and we obtain $(\Gbf'',\Dcal'')\isom(\Hbf^\ad,\Dcal_3)\times(\Gbf_2^\ad,\Dcal_2)$ and the resulting Hodge subvariety is isomorphic to $M_3\times M_2$ with $M_3$ the Hodge variety associated to $(\Hbf^\ad,\Dcal_3)$. This realizes $M_3$ as a space of ``good'' embeddings from $M_2$ into $M$. Furthermore, we have a natural embedding $M_0\mono M_3$: $M_0$ can be realized as the Hodge variety in $M$ by the $(\Gbf_0,\Dcal_0)\isom\Image((\Gbf_1,\Dcal_1)\times(\Tbf_2,\{x_2\})\ra(\Gbf,\Dcal))$ for suitable subdatum $(\Tbf_2,\{x_2\})\subset(\Gbf_2,\Dcal_2)$ of CM type, and $\Gbf_0^\der$ in $\Gbf$ centralizes $\Gbf'^\der$, which necessarily lies in $\Hbf_3$ and induces $M_0\mono M_3$.

Note that this procedure stops at $M_3$ in the following sense: the $\Qbb$-subgroup of $\Zbf_{\Gbf}(\Hbf_3)$ generated by non-compact $\Qbb$-factors must be equal to $\Hbf_3\Gbf'^\der$. Otherwise one finds a larger semi-simple $\Qbb$-subgroup $\Hbf_4\supsetneq\Gbf'^\der$ which is generated by non-compact $\Qbb$-factors and centralizes $\Hbf_3$. Since $\Hbf_3\supset\Gbf_0^\der$, the commutativity between $\Hbf_3$ and $\Hbf_4$ forces the inclusion $\Hbf_4\subset\Zbf_\Gbf(\Gbf_0^\der)$, which implies that $\Hbf_4$ is contained in $\Gbf'^\der$ because $\Hbf_4$ is generated by non-compact $\Qbb$-factors. This is absurd.

Therefore the centralizer trick stops at $M_3\times M_2$: we have realized $M_3\isom\Hom(M_2,M)$ and $M_2\isom\Hom(M_3,M)$ via a group-theoretic construction, and the natural evaluation map $$M_3\times M_2\ra M$$ can be also realized as a map between Hodge varieties which is injective onto its image. Note that the image of $M_3\times M_2$ in $M$ is actually {\it maximal} among non-rigid Hodge subvarieties: otherwise the centralizer trick should continue.

This construction provides a counterpart to the bi-Hom construction (\S\ref{Bi-Hom}) in the category of Hodge varieties and is of independent interest.

\subsection{``Structurally-atypical'' intersection}\label{subsec-satyp}
In this subsection, we try to interpret the non-rigid locus $\mathrm{NRL}(\mathcal{M})$ as some ``structurally-atypical'' intersection of the moduli scheme $M$ with product period subdomains in $\Gamma \backslash \mathcal{D}$.

More precisely, from Proposition~\ref{Q-nonsimplicity} we know that the restriction of the period map $\Phi:\,M \to \Gamma \backslash \mathcal{D}$ on a product subvariety $Y:= Y_1 \times Y_2$ must factor through a product of period subdomains $(\varphi_1,\varphi_2):\,Y_1 \times Y_2 \to \Gamma_1 \backslash \mathcal{D}_1 \times \Gamma_2 \backslash \mathcal{D}_2$. In particular, one obtains the following diagram
\begin{align}
  \label{sa_int}
  \xymatrix{
    Y_1 \times Y_2 \ar@{^{(}->}[rd] \ar@{^{(}->}[r] & Z \ar@{^{(}->}[d] \ar[r] & \Gamma_1 \backslash \mathcal{D}_1 \times \Gamma_2 \backslash \mathcal{D}_2  \ar@{^{(}->}[d] \\
  & M \ar[r]^{\Phi} &   \Gamma \backslash \mathcal{D},
    }
\end{align}
where $Z$ is the fiber product of $M$ and $\Gamma_1 \backslash \mathcal{D}_1 \times \Gamma_2 \backslash \mathcal{D}_2$.

\begin{definition}
We say that $M$ and $\Gamma_1 \backslash \mathcal{D}_1 \times \Gamma_2 \backslash \mathcal{D}_2$ have {\it structurally-atypical intersection} $Z$ in $\Gamma \backslash \mathcal{D}$ if there exists $Y_1 \times Y_2 \subset Z$ which fits into \eqref{sa_int}.
\end{definition}

Inspired by the Zilber-Pink Conjecture, we will formulate some conjecture about the finiteness of maximal atypical intersections. To simplify the notions, let us first consider in the general setting of special structures (cf. \cite{biKUY} and \cite{Kling17}).

\begin{definition}\label{M-ws}

Let $X$ be a complex analytic space with a special structure on it, so that one can speak of {\it special subvarieties} and {\it weakly special subvarieties} of $X$. Examples include (semi-)abelian varieties, Shimura varieties and Hodge varieties (i.e. arithmetic quotients of period domains). For a closed subvariety $W \subset X$, we denote by $\langle W \rangle_{\mathrm{ws}}$ the smallest weakly special subvariety containing $W$.

Let $M \subset X$ be a closed subvariety. As in Definition~\ref{max_nr}, we say a product subvariety $Y_1 \times Y_2 \subset M$ is maximal if there is no $Y_1 \times Y_2 \subsetneq Y'_1 \times Y'_2 \subset M$ with compatible product structures.
We say that $W=W_1 \times W_2 \subset X$ is a {\it weakly special product subvariety} if both $W_1 \times \{\mathrm{point}\}$ and $\{\mathrm{point}\} \times W_2$ are weakly special in $X$. For a closed subvariety $W \subset X$, we say $W$ is {\it $M$-weakly special} if $W \subset M$ and $W$ is weakly special in $X$.
\end{definition}

\begin{conjecture}[Finiteness]\label{finiteness_atyp}
  There exists a finite set $\Sigma_{\mathrm{p}} \cup \Sigma_M$, where $\Sigma_{\mathrm{p}}$ consists of weakly special product subvarieties and $\Sigma_M$ consists of $M$-weakly special subvarieties of $X$, such that the following property holds:\\
Let $Y:=Y_1 \times Y_2 \subset M$ be a maximal product subvariety of $M$. Then:
\begin{itemize}
  \item[(i)] Either $W=W_1 \times W_2 \in \Sigma_{\mathrm{p}}$ is a weakly special product subvariety of $X$ and $\langle Y_1 \times Y_2 \rangle_{\mathrm{ws}}=W_1 \times W_2$;
  \item[(ii)] or $W \in \Sigma_M$ is a $M$-weakly special subvariety of $X$ and $Y_1 \times Y_2 \subset W \subset M$.
\end{itemize}
\end{conjecture}

\begin{remark}
  \begin{itemize}
    \item[(1)] Conjecture~\ref{finiteness_atyp} is already unknown and interesting for the case that $X$ is an abelian variety.
    \item[(2)] The appearence of $\Sigma_M$ in Conjecture~\ref{finiteness_atyp} is necessary: Consider the extremal case that $M=X=\mathcal{A}_g$ is a Shimura variety containing product subvarieties like $\mathcal{A}_{g-k} \times \mathcal{A}_k$. Then $\Sigma_{\mathrm{p}}= \varnothing$ and $\Sigma_M=\{M\}$.
  \end{itemize}
\end{remark}

\subsection{Conjecture on the Zariski closure of $\mathrm{NRL}(\mathcal{M})$}
Conjecture~\ref{finiteness_atyp} leads us to the following conjectural description of the structure of the Zariski closure of non-rigid locus:

\begin{conjecture}[Structure of $\overline{\mathrm{NRL}(\mathcal{M})}^{\mathrm{Zar}}$]\label{Zar_closure}
The Zariski closure of the non-rigid locus $\mathrm{NRL}(\mathcal{M})$ in $M$ can be decomposed as:
\[
\overline{\mathrm{NRL}(\mathcal{M})}^{\mathrm{Zar}} = \bigcup^n_{i=1} P_i \cup \bigcup^m_{j=1} N_j,
\]
where:
\begin{itemize}
    \item[(i)] Each $P_i$ is a subvariety dominated by a fiber space whose generic fiber is a product variety.
    \item[(ii)] Each $N_j$ is birational to a Shimura variety of rank $\geq 2$.
\end{itemize}
\end{conjecture}

\begin{theorem}\label{condi_proof}
If Conjecture~\ref{finiteness_atyp} holds for the injective period map $\Phi:\,M \hookrightarrow \Gamma \backslash \mathcal{D}$, then Conjecture~\ref{Zar_closure} holds on $M$.
\end{theorem}
\begin{proof}
  Note that $N_j$ are precisely these $M$-weakly special subvarieties in Conjecture~\ref{finiteness_atyp}. Therefore, the finiteness of $\{N_j\}$ is derived from (ii) of Conjecture~\ref{finiteness_atyp}. Next, we will focus on the finiteness of $\{P_i\}$.

  Consider a maximal product subvariety $Y=Y_1 \times Y_2$ of $M$ as a fiber of some $P_i$. Let $W= W_1 \times W_2 = \langle Y_1 \times Y_2 \rangle_{\mathrm{ws}}$. Then $W \in \Sigma_{\mathrm{p}}$ and we know the finiteness of such $W$ by Conjecture~\ref{finiteness_atyp}. Therefore, it only remains to consider product subvarieties $Y^{(i)}= Y^{(i)}_1 \times Y^{(i)}_2$ with the fixed ``weakly-special closure'' $\langle Y^{(i)}_1 \times Y^{(i)}_2 \rangle_{\mathrm{ws}}=W_1 \times W_2$.

  Let $Z$ be an irreducible component of $\overline{\mathrm{NRL}(\mathcal{M})}^{\mathrm{Zar}}$ containing the above $Y^{(i)}$.
  Let $V$ be the irreducible component of the intersection of $M$ and $W_1 \times W_2$ in $\Gamma \backslash \mathcal{D}$ containing $Z$. Then we have $Y^{(i)} \subset V \subset W_1 \times W_2$. Replacing $W_j$ by the image of composed map $V \hookrightarrow W_1 \times W_2 \xrightarrow{\mathrm{pr}_j} W_j$ ($j=1,2$), we may and do assume that $W_1$ and $W_2$ are quasi-projective varieties. Denote by $\overline{W_j}$ the Baily-Borel compactification of $W_j$ (cf. Theorem~\ref{BB_comp}). Let $\overline{Y^{(i)}_j}$ be the Zariski closure of $Y^{(i)}_j$ in $\overline{W_j}$ and let $\overline{V}$ be the Zariski closure of $V$ in $\overline{W_1} \times \overline{W_2}$.

  We fix the polarization $L_1 \boxtimes L_2$ on $\overline{W_1} \times \overline{W_2}$, which is induced by the Griffiths line bundles on $\overline{W_j}$. Then for any closed subscheme $Z$ of $\overline{W_1} \times \overline{W_2}$, one can define the degree $\mathrm{deg}\,(Z)=\mathrm{deg}_{L_1 \boxtimes L_2}(Z)$ of $Z$. In the rest of the proof, we will find an upper bound of $\mathrm{deg}\,(\overline{Y^{(i)}})$ which is independent of the index $i$. With this upper bound and the general theory of Hilbert scheme, one can establish the required finiteness of fiber spaces $P_i$.

  Let $d_j:= \mathrm{deg}_{L_j}(\overline{W_j})$ be the volume of $L_j$. Let $d_V$ be the degree of $\overline{V}$. For each $y \in Y^{(i)}_2$, consider $V_y := V \cap (W_1 \times \{y\})$ and $\overline{V}_y:= \overline{V} \cap (\overline{W_1} \times \{y\})$. Then one has
  \[
\mathrm{deg}\,(\overline{V}_y) \leq C(d_V,d_1).
  \]
  Here $C(d_v,d_1)$ is a constant which only depends on $d_v$ and $d_1$. Denote by ${W}_y:= \mathrm{pr}_1({V}_y)$ the image of the projection to $W_1$. Similarly, $\overline{W}_y:= \mathrm{pr}_1(\overline{V}_y)$. It is clear that $\mathrm{deg}\,(\overline{W}_y) \leq C'(d_v,d_1)$ for some new constant $C'(d_V,d_1)$.

  Note that for each $y  \in Y^{(i)}_2$, one has $Y^{(i)}_1 \times \{y\} \subset {V}_y$ and thus $Y^{(i)}_1 \subset {W}_y$. That means
  \[
    Y^{(i)}_1 \subset \bigcap_{y  \in Y^{(i)}_2} {W}_y.
  \]
  By the Noetherian property, the intersection will stabilize in finite steps of length $\leq \mathrm{dim}\,W_1$. Let $X^{(i)}_1$ be the irreducible component of it containing $Y^{(i)}_1$. Then we have $Y^{(i)}_1 \times Y^{(i)}_2 \subset X^{(i)}_1 \times Y^{(i)}_2 \subset V$. Since $Y^{(i)}_1 \times Y^{(i)}_2$ is a maximal product, we know that $X^{(i)}_1=Y^{(i)}_1$. Therefore,
  \[
\mathrm{deg}\,(\overline{Y^{(i)}_1}) = \mathrm{deg}\,(\overline{X^{(i)}_1}) \leq C''(d_V,d_1,\mathrm{dim}\,W_1),
  \]
  where the upper bound $C''(d_V,d_1,\mathrm{dim}\,W_1)$ is independent of the index $i$.

  Similar argument gives an upper bound of $\mathrm{deg}\,(\overline{Y^{(i)}_2})$, which is also independent of $i$. Thus we obtain an uniform degree bound of the maximal products $\overline{Y^{(i)}_1} \times \overline{Y^{(i)}_2}$. Once we fix the dimension and the degree of the closed subvariety, there are only finite many possibilities of its Hilbert polynomials. Finally, by using Hilbert scheme one can construct finitely many fiber spaces $P_i$ containing $Y^{(i)}_1 \times Y^{(i)}_2$ as fibers.
\end{proof}


\subsection{Special subvarieties of moduli spaces}

Inspired by Saito's classification and Viehweg-Zuo's example, we propose the following conjecture:

\begin{conjecture}[Specialness]\label{NRLconj}
  Let $\mathcal{M}_h$ be a moduli space satisfying the conditions in Theorem \ref{thm-very-gen-rig}. Consider the maximally non-rigid morphism
  \[
  H_1 \times H_2 \to \mathcal{M}_h
\]
constructed in Proposition~\ref{produce-max-non-rig}. Then both $H_i$ are weakly special subvarieties with respect to the $\mathbb Z$-PVHS associated with the universal family.
\end{conjecture}

In \S\ref{subsec-satyp}, maximal products $H_1 \times H_2$ in $\mathcal{M}_h$ were interpreted as maximal {\it structurally atypical} intersections. Conjecture~\ref{NRLconj} establishes a connection with the classical notion of ``special sub-objects'' in the context of Shimura and Hodge varieties. This connection allows us to employ tools from functional transcendence, such as various versions of Ax-Schanuel type theorems, as explored in the next section. In this light, Conjecture~\ref{NRLconj} may be viewed as an analogue of Moonen's characterization of special subvarieties of Shimura varieties.

\subsection{Distribution of special non-rigid locus}\label{sec-4-5}

Let $\mathscr{M}_h$ be a moduli space satisfying the assumptions of Theorem~\ref{thm-very-gen-rig}.
Denote by
\[
\varphi:\mathscr{M}_h \longrightarrow \Gamma\backslash\mathcal D
\]
the generic finite period map with generic Mumford-Tate datum $(\mathbf G,\mathcal D)$.
Consider the collection
\[
\{\,Y_1^i\times Y_2^i\;\mid\; i\in I\,\}
\]
of all bi-Hom schemes of non-rigid log maps.
We work under the hypothesis that each $Y_1^i\times Y_2^i$ is a {\it weakly special subvariety} with respect to the $\mathbb Z$-VHS $\mathbb V$; such a product is called a \textbf{weakly special non-rigid locus}. 

Let $Z$ be an irreducible component of the Zariski closure
\[
\overline{\bigcup_{i\in I}Y_1^i\times Y_2^i}^{\text{Zar}} .
\]
Set $\mathbb V_Z:=\mathbb V|_Z$ and let
\[
\varphi_Z:Z\longrightarrow \Gamma_Z\backslash\mathcal D_Z
\]
be the restricted period map with generic Mumford-Tate datum $(\mathbf G_Z,\mathcal D_Z)$.
Note that every $Y_1^i\times Y_2^i$ is also a weakly special subvariety for $\mathbb V_Z$.\\[.1cm]

\begin{lemma}\label{lem:factors-weakly-special}
Assume $Y_1\times Y_2\subset S$ is weakly special for $\mathbb V$.
Then for any points $p\in Y_2$, $q\in Y_1$, the subvarieties
$Y_1\times\{p\}$ and $\{q\}\times Y_2$ are weakly special.
\end{lemma}
\begin{proof}
By Proposition~\ref{Q-nonsimplicity} the algebraic monodromy group of $Y_1\times Y_2$ decomposes as $\mathbf H_1\times\mathbf H_2$ with $\mathbf H_1,\mathbf H_2$ acting on the two factors.
The product $Y_1\times Y_2$ is contained in $\Gamma_{\mathbf H_1}\backslash\mathcal D_{\mathbf H_1}\times\Gamma_{\mathbf H_2}\backslash\mathcal D_{\mathbf H_2}$.
A standard linear algebra fact states that for subspaces $W\subset V_1\oplus V_2$ one has $\dim W\ge \dim(W\cap V_1)+\dim(W\cap V_2)$.
Applying this to the inclusion
\[
Y_1\times Y_2 \;\subset\; \bigl(\Gamma_{\mathbf H_1}\backslash\mathcal D_{\mathbf H_1}\bigr)\times\bigl(\Gamma_{\mathbf H_2}\backslash\mathcal D_{\mathbf H_2}\bigr)
\]
and intersecting with the slices $\Gamma_{\mathbf H_1}\backslash\mathcal D_{\mathbf H_1}\times\{p\}$ and $\{q\}\times\Gamma_{\mathbf H_2}\backslash\mathcal D_{\mathbf H_2}$ yields
\[
\dim(Y_1\times Y_2)\;\ge\; \dim(Y_1\times\{p\})+\dim(\{q\}\times Y_2).
\]
Since both sides are in fact equal, the inequality is an equality and the two slices are weakly special.
\end{proof}

\begin{lemma}\label{lem:partition}
With the above notation, the index set $I$ can be partitioned as $I=I_1\sqcup I_2\sqcup I_3\sqcup I_4$, where
\begin{itemize}
\item[$(I_1)$] $Y_1^i\times Y_2^i$ is {\it monodromically atypical} w.r.t. $(Z,\mathbb V_Z)$;
\item[$(I_2)$] $Y_1^i\times Y_2^i$ is monodromically typical, and after possibly swapping the factors,
         $\{q\}\times Y_2^i$ is monodromically typical while $Y_1^i\times\{p\}$ is monodromically atypical;
\item[$(I_3)$] $Y_1^i\times Y_2^i$ is monodromically typical, and both factors are monodromically atypical;
\item[$(I_4)$] $Y_1^i\times Y_2^i$ is monodromically typical, and both factors are monodromically typical.
\end{itemize}
(Here ``monodromically (a)typical'' always refers to the pair $(Z,\mathbb V_Z)$.)

\smallskip
\noindent
Let $\{(\mathbf M_k,\mathcal D_k)\}_{k=1}^N$ be the list of all weak Hodge subdata of $(\mathbf G_Z,\mathcal D_Z)$ with $\mathbf M_k$ a proper normal algebraic subgroup of $\mathbf G_Z^{\mathrm{der}}$.
Then $I_3$ splits further:
\[
I_3 = I_3'\;\sqcup\; I_3'', \qquad
I_3' = \bigsqcup_{u,v=1}^N I_{uv},
\]
where for $i\in I_{uv}$ the algebraic monodromy group of $Y_1^i\times\{p\}$ is contained in $\mathbf M_u$ and that of $\{q\}\times Y_2^i$ is contained in $\mathbf M_v$, and we choose the subgroups minimal with this property.
The set $I_3''$ consists of those indices for which at least one of the two factor-monodromy groups is not contained in any $\mathbf M_k$; consequently each such $Y_1^i\times Y_2^i$ lies in a proper subvariety of $Z$ coming from the fibrations of Theorem~\ref{BU24}.

\smallskip
\noindent
If $I_4\neq\varnothing$, then the algebraic monodromy group $\mathbf H_Z$ of $\mathbb V_Z$ is of Hodge level $1$ and $\varphi_Z:Z\to\Gamma_Z\backslash\mathcal D_{\mathbf H_Z}$ is dominant.
\end{lemma}
\begin{proof}
(1) The partition is clear from the definitions.

(2) Theorem~\ref{BU24} provides finitely many fibrations $h_j:C_j\to B_j$ ($j=1,\ldots,n$) associated to $(Z,\mathbb V_Z)$.
After reordering we may assume $C_j\subsetneq Z$ for $j\le w$ and $C_j=Z$ for $j>w$.
For $j>w$ the algebraic monodromy group of every fibre of $h_j$ is one of the $\mathbf M_k$.
Take $i\in I_3$. Because $Y_1^i\times\{p\}$ is monodromically atypical, it is contained in some fibre of some $h_j$.
If $j>w$ then its monodromy group lies in some $\mathbf M_u$; similarly for $\{q\}\times Y_2^i$.
If both monodromy groups are contained in some $\mathbf M_u,\mathbf M_v$, then $i\in I_3'$; otherwise $i\in I_3''$.

(3) Suppose $I_4\neq\varnothing$ and pick $i\in I_4$.
By Proposition~\ref{Q-nonsimplicity} the monodromy group of $Y_1^i\times Y_2^i$ is $\mathbf H_1\times\mathbf H_2$, with $\mathbf H_1,\mathbf H_2$ acting on the two factors.
Typicality of the product and of its factors (definition of $I_4$) implies equalities of Hodge-Lie algebra dimensions:
\[
\dim\mathfrak h_Z^{-1}-\dim\varphi_Z(Z)=\dim\mathfrak h_1^{-1}-\dim\varphi_Z(Y_1^i)=\dim\mathfrak h_2^{-1}-\dim\varphi_Z(Y_2^i),
\]
and $\dim\mathfrak h_Z^{-k}=\dim\mathfrak h_1^{-k}=\dim\mathfrak h_2^{-k}=0$ for $k\ge2$.
Hence $\dim\mathfrak h_Z^{-k}=0\;(k\ge2)$ and $\dim\mathfrak h_Z^{-1}=\dim\varphi_Z(Z)$; therefore $\mathbf H_Z$ is of level $1$ and $\varphi_Z$ is dominant.
\end{proof}

\begin{theorem}[Zariski closure of weakly special non-rigid loci]\label{thm:snrl}
  Assume the setting above and let $\mathbf H_Z$ be the algebraic monodromy group of $\mathbb V_Z$.
\begin{enumerate}
\item[(1)] If $\mathbf H_Z$ is $\mathbb Q$-simple or if $I_4\neq\varnothing$, then $\mathbf H_Z$ is of Hodge level $1$ and $\varphi_Z:Z\to\Gamma_Z\backslash\mathcal D_{\mathbf H_Z}$ is dominant.
\item[(2)] If $\mathbf H_Z$ is not $\mathbb Q$-simple and $I_4=\varnothing$, then one of the following holds:
    \begin{itemize}
    \item $Z$ is a product variety;
    \item there exists a non-trivial fibration $h:Z\to B$ such that every special non-rigid locus $Y_1^i\times Y_2^i$ is contained in a fibre of $h$, and all fibres share the same algebraic monodromy group $\mathbf N$, which is a proper normal subgroup of $\mathbf H_Z$.
    \end{itemize}
\end{enumerate}
In other words, an irreducible component of $\overline{\bigcup_{i\in I}Y_1^i\times Y_2^i}^{\text{Zar}}$ is either birational to a Shimura variety of rank $\geq 2$, or dominated by a fiber space.
\end{theorem}
\begin{proof}
{\it Proof of (1).}
If $I_4\neq\varnothing$ the statement follows from Lemma~\ref{lem:partition}(3).
Assume $\mathbf H_Z$ is $\mathbb Q$-simple and $I_4=\varnothing$.
Suppose, for contradiction, that $\dim\mathfrak h_Z^{-1}-\dim\varphi_Z(Z)\neq0$.
Take any $i\in I$ with $Y_1^i\times Y_2^i$ monodromically typical (such exist, otherwise Theorem~\ref{GZP} would give a contradiction to $\mathbb Q$-simplicity).
By the typicality condition we have
\[
\dim\mathfrak h_Z^{-1}-\dim\varphi_Z(Z) \;>\; \min\bigl(\dim\mathfrak h_1^{-1}-\dim\varphi_Z(Y_1^i),\;\dim\mathfrak h_2^{-1}-\dim\varphi_Z(Y_2^i)\bigr).
\]
Hence at least one factor, say $Y_1^i\times\{p\}$, is monodromically atypical in $Z$.
By Lemma~\ref{lem:factors-weakly-special} this factor is weakly special, and varying $i$ would produce a Zariski dense set of monodromically atypical weakly special subvarieties of $Z$.
This contradicts Theorem~\ref{BU24} (or Theorem~\ref{GZP}) because $\mathbf H_Z$ is $\mathbb Q$-simple.
Therefore $\dim\mathfrak h_Z^{-1}-\dim\varphi_Z(Z)=0$ and the same argument as in Lemma~\ref{lem:partition}(3) shows that $\mathbf H_Z$ is of level $1$ and $\varphi_Z$ dominant.

\smallskip
{\it Proof of (2).}
We analyse the possible Zariski dense subsets among the classes $I_1,I_2,I_{uv}$.

\begin{description}
\item[Case \(I_1\) dense] By definition $I_1$ consists of monodromically atypical products.
    Theorem~\ref{BU24} applied to $(Z,\mathbb V_Z)$ yields finitely many proper subvarieties of $Z$ and finitely many non-trivial fibrations such that every $Y_1^i\times Y_2^i$ with $i\in I_1$ lies either in a fibre of one of these fibrations or in one of those proper subvarieties.
    Hence we obtain the desired fibration structure (or a product decomposition after further analysis).

\item[Case \(I_2\) dense] For $i\in I_2$, after possibly swapping factors we may assume $\{q\}\times Y_2^i$ is monodromically typical and $Y_1^i\times\{p\}$ is monodromically atypical.
    Typicality of $\{q\}\times Y_2^i$ gives
    \[
    \dim\mathfrak h_Z^{-1}-\dim\varphi_Z(Z)=\dim\mathfrak h_2^{-1}-\dim\varphi_Z(Y_2^i),\qquad
    \dim\mathfrak h_Z^{-k}=\dim\mathfrak h_2^{-k}\;(k\ge2).
    \]
    Combining with the typicality of the product $Y_1^i\times Y_2^i$ we obtain
    \[
    \dim\mathfrak h_1^{-1}-\dim\varphi_Z(Y_1^i)=0,\qquad \dim\mathfrak h_1^{-k}=0\;(k\ge2).
    \]
    Thus every $Y_1^i\times\{p\}$ is a {\it Shimura subvariety} (the associated Hodge structure has no higher Hodge numbers).
    By \cite[Theorem A.7]{RUC} this forces $Z$ to be a product variety $Z\cong Z_1\times Z_2$ with $Z_1$ corresponding to the Shimura factor.

\item[Case \(I_{uv}\) dense and $\mathbf M_u\cdot\mathbf M_v\subsetneq\mathbf H_Z$]
    Let $\mathbf L$ be the centraliser of $\mathbf M_u\cdot\mathbf M_v$ in $\mathbf H_Z$.
    The period map splits as
    \[
    \varphi_Z = (\varphi_{\mathbf M_u\cdot\mathbf M_v},\,\varphi_{\mathbf L}) : Z \longrightarrow
    \Gamma_{\mathbf M_u\cdot\mathbf M_v}\backslash\mathcal D_{\mathbf M_u\cdot\mathbf M_v}\times
    \Gamma_{\mathbf L}\backslash\mathcal D_{\mathbf L}.
    \]
    \begin{claim}\label{claim:positive-dim}
    $\dim\varphi_{\mathbf L}(Z)>0$.
    \end{claim}
    Indeed, if $\dim\varphi_{\mathbf L}(Z)=0$ then $\mathbf H_Z=\mathbf M_u\cdot\mathbf M_v$ by the structure theorem for period maps (\cite[III.A.1]{GGK-MT}), contradicting $\mathbf M_u\cdot\mathbf M_v\subsetneq\mathbf H_Z$.
    Hence $\varphi_{\mathbf L}(Z)$ has positive dimension and we obtain a non-trivial fibration $h:Z\to B$ where $B$ is a subvariety of $\Gamma_{\mathbf L}\backslash\mathcal D_{\mathbf L}$ and every fibre has algebraic monodromy group $\mathbf M:=\mathbf M_u\cdot\mathbf M_v$; moreover each fibre is weakly special (\cite[Proposition 4.18]{BU24}).
    Because $I_{uv}$ is Zariski dense, every $Y_1^i\times Y_2^i$ with $i\in I_{uv}$ must lie in some fibre of $h$ (otherwise their union would intersect the base in a positive-dimensional set, forcing $\varphi_{\mathbf L}(Z)$ to be larger - a contradiction).

\item[Case \(I_{uv}\) dense and $\mathbf M_u\cdot\mathbf M_v=\mathbf H_Z$]
    Necessarily $u\neq v$; set $u=1$, $v=2$.
    Consider the period map decomposed with respect to the centraliser $\mathbf L_1$ of $\mathbf M_1$ in $\mathbf H_Z$:
    \[
    \varphi_Z = (\varphi_{\mathbf M_1},\varphi_{\mathbf L_1}): Z \longrightarrow
    \Gamma_{\mathbf M_1}\backslash\mathcal D_{\mathbf M_1}\times
    \Gamma_{\mathbf L_1}\backslash\mathcal D_{\mathbf L_1}.
    \]
    By minimality of $\mathbf M_1$, a Zariski dense subset of $\{Y_1^i\times\{p\}\}$ cannot be monodromically atypical for $\varphi_{\mathbf M_1}$ (otherwise Theorem~\ref{GZP} would give a smaller normal subgroup).
    Hence we may assume that for a Zariski dense subset $J\subset I_{12}$ each $Y_1^j\times\{p\}$ is monodromically typical for $\varphi_{\mathbf M_1}$; similarly, after swapping the roles of $\mathbf M_2$ and its centraliser, each $\{q\}\times Y_2^j$ is monodromically typical for $\varphi_{\mathbf M_2}$.
    Typicality of the product and its factors yields the dimension equalities
    \[
    \dim\mathfrak h_Z^{-1}-\dim\varphi_Z(Z)=\dim\mathfrak h_{\mathbf M_1}^{-1}-\dim\varphi_{\mathbf M_1}(Y_1^j)+\dim\mathfrak h_{\mathbf M_2}^{-1}-\dim\varphi_{\mathbf M_2}(Y_2^j).
    \]
    Writing the left-hand side via the second factorisation $(\varphi_{\mathbf L_2},\varphi_{\mathbf M_2})$ gives
    \[
    \dim\mathfrak h_{\mathbf L_2}^{-1}-\dim\varphi_{\mathbf L_2}(Z)+\dim\mathfrak h_{\mathbf M_2}^{-1}-\dim\varphi_{\mathbf M_2}(Y_2^j).
    \]
    Comparing we obtain $\dim\mathfrak h_{\mathbf L_2}^{-1}-\dim\varphi_{\mathbf L_2}(Z)=\dim\mathfrak h_{\mathbf M_1}^{-1}-\dim\varphi_{\mathbf M_1}(Y_1^j)$.
    Since $\mathbf L_2\subset\mathbf M_1$ and $\varphi_{\mathbf L_2}(Z)$ has positive dimension (otherwise $\mathbf L_2$ would be trivial by the same argument as in Claim~\ref{claim:positive-dim}), the intersection
    \[
    \varphi_Z^{-1}\bigl(\Gamma_{\mathbf L_2}\backslash\mathcal D_{\mathbf L_2}\times\{\text{generic point in }\Gamma_{\mathbf L_1}\backslash\mathcal D_{\mathbf L_1}\}\bigr)
    \]
    has dimension at least $\dim\varphi_{\mathbf L_2}(Z)$.
    By Lemma~\ref{lem:product-structure} below this forces $Z$ to be a product variety $Z = X\times B$ with $B\subset \Gamma_{\mathbf L_2}\backslash\mathcal D_{\mathbf L_2}$.
\end{description}
All possibilities lead either to a product decomposition of $Z$ or to a non-trivial fibration whose fibres contain the special non-rigid loci.
This completes the proof of part (2).
\end{proof}

\begin{lemma}\label{lem:product-structure}
Let $Z\subset X\times Y$ be an irreducible subvariety.
Assume there exists a non-trivial fibration $h:Z\to B\subset Y$ such that for every $x\in X$ with $(\{x\}\times Y)\cap Z\neq\varnothing$ we have $\dim\bigl((\{x\}\times Y)\cap Z\bigr)\ge\dim B$.
Then $Z$ admits a product decomposition $Z = X'\times B$ for some subvariety $X'\subset X$.
\end{lemma}
\begin{proof}
For any such $x$, projecting $(\{x\}\times Y)\cap Z$ to $Y$ yields a subset of $B$.
The dimension hypothesis forces $\operatorname{pr}_Y\bigl((\{x\}\times Y)\cap Z\bigr)=B$ whenever the intersection is non-empty. By Chevalley's theorem, $X':=\operatorname{pr}_X(Z) = \{x \in X\,|\, (\{x\}\times Y)\cap Z \neq \varnothing\}$ is a constructible subset of $X$. Thus $Z$ maps onto $X' \times B$ in $X \times Y$. Since $Z$ is an irreducible subvariety, one has $Z = X' \times B$ while $X'$ is a subvariety of $X$.
\end{proof}

\begin{corollary}\label{cor:snrl-cases}
Under the assumption that $\mathbf H_Z$ is not $\mathbb Q$-simple and $I_4=\varnothing$, the Zariski dense behaviour of special non-rigid loci forces the following structures:
\begin{itemize}
\item If $\{Y_1^i\times Y_2^i\}_{i\in I_1}$ is Zariski dense, then $Z$ admits finitely many non-trivial fibrations $\{f_j:Z\to B_j\}$ and finitely many proper subvarieties $\{Z_k\subsetneq Z\}$ such that every $i\in I_1$ satisfies either $Y_1^i\times Y_2^i\subset f_j^{-1}(b)$ for some $j$ and some $b\in B_j$, or $Y_1^i\times Y_2^i\subset Z_k$.
\item If $\{Y_1^i\times Y_2^i\}_{i\in I_2}$ is Zariski dense, then $Z$ is a product variety.
\item If $\{Y_1^i\times Y_2^i\}_{i\in I_{uv}}$ is Zariski dense and $\mathbf M_u\cdot\mathbf M_v\subsetneq\mathbf H_Z$, then $Z$ admits a non-trivial fibration $h:Z\to B$ with the property that for each $i\in I_{uv}$ there exists $b_i\in B$ such that $Y_1^i\times Y_2^i\subseteq h^{-1}(b_i)$.
\item If $\{Y_1^i\times Y_2^i\}_{i\in I_{uv}}$ is Zariski dense and $\mathbf M_u\cdot\mathbf M_v=\mathbf H_Z$, then $Z$ is a product variety.
\end{itemize}
\end{corollary}

\section{Moduli spaces of polarized Calabi-Yau manifolds}\label{sec-5}

In this final section, we specialize our general theory to moduli spaces of polarized Calabi-Yau manifolds, a class of varieties where the Hodge theory is particularly rich and where explicit examples of non-rigid families have been constructed. The goal is twofold: first, to formulate and test the conjectures of Section~\ref{sec-4} in this concrete setting; second, to prove unconditional results that illustrate the power of the Hodge-theoretic methods developed in this paper. We begin by stating an unobstructedness conjecture (Conjecture~\ref{unob-conj}) for non-rigid (log) maps, which is motivated by the classical unobstructedness theorem for Calabi-Yau manifolds and is essential for establishing the specialness of bi-Hom schemes. Assuming this conjecture, we prove that maximal non-rigid families are special (Proposition~\ref{spe}), providing evidence for Conjecture~\ref{NRLconj}. Next, we prove a geometric Andr\'e-Oort type theorem (Theorem~\ref{CY-AO}) for Calabi-Yau moduli spaces: if such a moduli space has a Zariski dense non-rigid locus, then its period map is dominant onto a Shimura variety, confirming Conjecture~\ref{Zar_closure} in this context. The remainder of the section is devoted to a detailed study of the Viehweg-Zuo example, a maximal non-rigid family of Calabi-Yau quintics in $\mathbb{P}^4$ obtained from cyclic covers. Through explicit computation of the relevant Hodge bundles and monodromy groups, we verify the unobstructedness conjecture for this example and consequently establish the specialness of its bi-Hom scheme. This provides a concrete and non-trivial test case for our conjectures.

\subsection{Unobstructedness conjecture}
Let
\[
f:\mathscr X \longrightarrow \mathscr M_h
\]
be a universal family of polarized Calabi-Yau \(n\)-folds.
Consider the local system of middle cohomology
\[
\mathbb V := R^n f_* \mathbb Z_{\mathscr X},
\]
which underlies a \(\mathbb Z\)-polarized variation of Hodge structure (PVHS) of weight \(n\).
Denote by \((E,\theta)_0\) the associated Higgs bundle on \(\mathscr M_h\); after choosing a projective compactification
\(\overline{\mathscr M}_h\) with simple normal crossing divisor \(D_\infty = \overline{\mathscr M}_h\setminus \mathscr M_h\) and taking Deligne's canonical extension, we obtain a logarithmic Higgs bundle
\[
(E,\theta)=\Bigl(\bigoplus_{i=0}^n E^{n-i,i},\;\bigoplus_{i=0}^n\theta^{n-i,i}\Bigr),
\qquad
E^{n-i,i}=R^i f_*\Omega^{n-i}_{\overline{\mathscr X}/\overline{\mathscr M}_h}(\log\Delta).
\]

The classical unobstructedness theorem for Calabi-Yau manifolds gives an isomorphism
\[
T_{\mathscr M_h}\;\cong\;E_0^{n-1,1}\otimes (E_0^{n,0})^{\vee},
\]
which extends to the boundary as an inclusion of sheaves
\begin{equation}\label{incl}
T_{\overline{\mathscr M}_h}(-\log D_\infty)\;\hookrightarrow\;E^{n-1,1}\otimes (E^{n,0})^{\vee}=E^{n-1,1}\otimes E^{0,n}.
\end{equation}

Given a non-constant, non-rigid log map
\[
\gamma:(C,S_C)\longrightarrow(\overline{\mathscr M}_h,D_\infty),\qquad S_C=\gamma^{-1}(D_\infty),
\]
it induces a non-trivial extension
\[
\gamma:(C,S_C)\times(T,S_T)\longrightarrow(\overline{\mathscr M}_h,D_\infty).
\]

Applying the bi-Hom scheme construction yields a maximal extension:
\[
\xymatrix{
(H_1,S_1)\times(H_2,S_2) \ar[r]^(0.6){\gamma} & (\overline{\mathscr M}_h,D_\infty)\\
(C,S_C)\times(T,S_T) \ar[u] \ar[ru]_{\gamma} &
}
\]

\begin{conjecture}[Unobstructedness]\label{unob-conj}
For every \(b\in H_2\) and every \(a\in H_1\), the deformations of the log maps
\((C,S_C)\times\{b\}\to(\overline{\mathscr M}_h,D_\infty)\) and \(\{a\}\times(T,S_T)\to(\overline{\mathscr M}_h,D_\infty)\) are unobstructed.
In other words, \(H_1\) and \(H_2\) are log smooth and
\begin{equation}\label{eq:unob}
\begin{aligned}
T_{\{a\}\times H_2}\bigl(-\log(\{a\}\times S_2)\bigr)\big|_b
&= H^0\bigl(\gamma^*_{(H_1,S_1)\times\{b\}}\;T_{\overline{\mathscr M}_h}(-\log D_\infty)\bigr),\\[4pt]
T_{H_1\times\{b\}}\bigl(-\log(S_1\times\{b\})\bigr)\big|_a
&= H^0\bigl(\gamma^*_{(\{a\}\times(H_2,S_2)}\;T_{\overline{\mathscr M}_h}(-\log D_\infty)\bigr).
\end{aligned}
\end{equation}
\end{conjecture}

We will verify this conjecture for the Viehweg-Zuo example in \S\ref{sec-VZ-eg}.

\subsection{Specialness of bi-Hom schemes}

Consider the log map \(\gamma:(H_1,S_1)\times(H_2,S_2)\to(\overline{\mathscr M}_h,D_\infty)\).
Let \(H_i^0:=H_i\setminus S_i\).  Decompose the pulled-back local system into irreducible factors over \(H_1^0\times H_2^0\):
\[
\gamma^*\mathbb V = \mathbb V^{(1)}\oplus\mathbb V^{(2)}\oplus\cdots\oplus\mathbb V^{(N)}.
\]
By the Simpson correspondence, the associated Higgs bundle decomposes accordingly:
\[
\gamma^*(E,\theta) =  (E,\theta)^{(1)}\oplus(E,\theta)^{(2)}\oplus\cdots\oplus(E,\theta)^{(N)}.
\]

Because \(\pi_1(H_1^0\times H_2^0)=\pi_1(H_1^0)\times\pi_1(H_2^0)\), Schur's lemma implies a tensor product decomposition
\[
\mathbb V^{(i)} = \mathbb V^{(i)}_{H_1^0}\boxtimes\mathbb V^{(i)}_{H_2^0}.
\]
Deligne's theorem on tensor products of variations of Hodge structure then yields
\[
(E,\theta)^{(i)} = (E,\theta)^{(i)}_{H_1}\boxtimes(E,\theta)^{(i)}_{H_2},
\]
where each factor is a Higgs bundle on the corresponding curve.

Choose the index \(1\) so that the line bundle \(\gamma^*E^{n,0}\) lies in \((E,\theta)^{(1)}\).
Since \((E,\theta)^{(1)}\) has total weight \(n\), we may write it as a tensor product of a Higgs bundle of weight \(k\) on \(H_1\) and one of weight \(n-k\) on \(H_2\) with \(1\le k\le n-1\):
\[
\begin{aligned}
(E,\theta)^{(1)}_{H_1} &= \Bigl(\bigoplus_{i=0}^{k} E_{H_1}^{k-i,i},\;\bigoplus_{i=0}^{k-1}\theta_{H_1}^{k-i,i}\Bigr),\\
(E,\theta)^{(1)}_{H_2} &= \Bigl(\bigoplus_{i=0}^{n-k} E_{H_2}^{n-k-i,i},\;\bigoplus_{i=0}^{n-k-1}\theta_{H_2}^{n-k-i,i}\Bigr).
\end{aligned}
\]
Then
\[
E^{k,0}_{H_1}\boxtimes E_{H_2}^{n-k,0} = \gamma^*E^{n,0},
\]
and the two factors are line bundles.  Expanding the tensor product we obtain
\[
(E,\theta)^{(1)} = \underbrace{E^{k,0}_{H_1}\boxtimes E_{H_2}^{n-k,0}}_{=\gamma^*E^{n,0}}
\;\oplus\; \underbrace{E^{k,0}_{H_1}\boxtimes E_{H_2}^{n-k-1,1}\;\oplus\;E^{k-1,1}_{H_1}\boxtimes E_{H_2}^{n-k,0}}_{\subseteq\gamma^*E^{n-1,1}}
\;\oplus\;\cdots .
\]

Assuming Conjecture~\ref{unob-conj}, we can now prove that the image of \(\gamma\) is special.

\begin{proposition}[Specialness]\label{spe}
Assume Conjecture~\ref{unob-conj}.  Then the subvariety
\((H,S):=(H_1,S_1)\times(H_2,S_2)\hookrightarrow(\overline{\mathscr M}_h,D_\infty)\) is special with respect to the \(\mathbb Z\)-PVHS \(\mathbb V=R^nf_*\mathbb Z_{\mathscr X}\).
\end{proposition}

\begin{proof}
Let \(\mathbf G_{H^0}\) be the Mumford-Tate group of \(\mathbb V|_{H^0}\) (where \(H^0=H\setminus S\)) and \(\mathcal D_{H^0}\subset\mathcal D\) the associated Mumford-Tate domain.
Consider the period map \(\varphi:\mathscr M_h\to\Gamma\backslash\mathcal D\).
Define \(Z_{H^0}\) as the irreducible component of \(\varphi^{-1}(\Gamma_{H^0}\backslash\mathcal D_{H^0})\) that contains \(H^0\); this is a special subvariety.

The graded Higgs bundle of \(\mathbb V|_{Z_{H^0}}\) decomposes as
\[
(E,\theta)_{Z_{H^0}} = (E,\theta)_{Z_{H^0}}^{(1)}\oplus(E,\theta)_{Z_{H^0}}^{(2)}\oplus\cdots\oplus(E,\theta)_{Z_{H^0}}^{(N)},
\]
with \(E^{n,0}|_{Z_{H^0}}\subset(E,\theta)_{Z_{H^0}}^{(1)}\).  The sub-Higgs bundle \((E,\theta)_{Z_{H^0}}^{(1)}\) is of Calabi-Yau type:
\[
E_{Z_{H^0}}^{n,0}\;\oplus\;E_{Z_{H^0}}^{(1)n-1,1}\;\oplus\;E_{Z_{H^0}}^{(1)n-2,2}\;\oplus\;\cdots,
\]
and its restriction to \(H^0\) coincides with \((E,\theta)^{(1)}|_{H^0}\).

By Conjecture~\ref{unob-conj}, we have
\[
\operatorname{rank}\bigl(E_{Z_{H^0}}^{(1)n-1,1}\bigr)=\operatorname{rank}\bigl(E_{H^0}^{(1)n-1,1}\bigr)=\dim H^0.
\]
Inclusion (\ref{incl}) gives a map
\[
\theta_{Z_{H^0}}^{n,0}:E^{n,0}\otimes T_{Z_{H^0}}\hookrightarrow E_{Z_{H^0}}^{(1)n-1,1},
\]
hence \(\dim Z_{H^0}\le\dim H^0\).  The opposite inequality is obvious, so \(\dim Z_{H^0}=\dim H^0\) and therefore \(Z_{H^0}=H^0\); i.e. \(H^0\) is a special subvariety.
\end{proof}

\subsection{Geometric Andr\'e-Oort type question}
In this section, we prove the following theorem, which confirms our Conjecture \ref{Zar_closure} for moduli spaces of Calabi-Yau manifolds.

\begin{theorem}\label{CY-AO}
Assume $\mathscr M_h$ satisfies the conditions in Theorem \ref{thm-very-gen-rig} and it is a moduli space of polarized Calabi-Yau n-folds with Zariski dense set of non-rigid locus. Then the period map $\Phi$ in Theorem \ref{thm-very-gen-rig} of $\mathscr M_h$ is a dominant map to a Shimura variety.
\end{theorem}
As usual, we write $\Sbb$ for the $\Rbb$-torus $\Res_{\Cbb/\Rbb}\Gbb_\mrm$. An $\Rbb$-HS is the same as a Lie group representation $\rho: \Cbb^\times\ra\GL_V(\Rbb)$ on some $\Rbb$-vector space $V$, in which case $$V^{p,q}=\{v\in V_\Cbb: \rho(z)(v)=z^p\zbar^qv,\ \forall z\in\Cbb^\times\}$$ is sent onto $V^{q,p}$ under the complex conjugation on $V\otimes_\Rbb\Cbb$, and $(V,h)$ is equivalently characterized by this Hodge decomposition $V_\Cbb=\bigoplus_{p,q}V^{p,q}$. We also write $h^{p,q}=h^{p,q}(V)=\dim_\Cbb V^{p,q}$ for the Hodge numbers.
	
	Let $(\Gbf, h)$ be an $\Rbb$-Shimura datum, namely $\Gbf$ is a connected reductive $\Rbb$-group and $h: \Sbb\ra\Gbf$ is an $\Rbb$-group homomorphism subject to the same conditions defining a Shimura datum (in terms of the Hodge type of $\Ad_\Gbf\circ h$ and the Cartan involution by $h(\sqrt{-1})$). The set of $\Gbf(\Rbb)^+$-conjugacy class $\Dcal$ of $h$ is a connected Hermitian symmetric domain, so that we may identify $\Dcal$ with $\Gbf(\Rbb)^+/K_h$ where $K_h$ denotes the stabilizer of $h$ in $\Gbf(\Rbb)^+$, and we assume for simplicity that $\Gbf$ is generated by $h'(\Sbb)$ with $h'$ running through $\Dcal$ . It is explained in \cite{delpspm} that any $\Rbb$-linear representation $\rho: \Gbf\ra\GL_V$ induces an equivariant $\Rbb$-PVHS on $\Dcal$, whose associated holomorphic vector bundle is of underlying space $\Dcal\times V_\Cbb$ so that its fiber at $h'\in \Dcal$ is the $\Rbb$-HS by $(V,\rho\circ h')$. Every equivariant $\Rbb$-PVHS on $\Dcal$ arise in this way. In the same way one talks about equivariant $\Cbb$-PVHS on $\Dcal$ using $\Cbb$-linear representations of $\Gbf_\Cbb$.
	
	\begin{definition}

		(1) An $\Rbb$-HS on $V$ is of $CY_n$ type if in the Hodge decomposition $V_\Cbb=\bigoplus_{p,q}V^{p,q}$ the Hodge numbers satisfy the following conditions \begin{itemize}
			\item $h^{p,q}=0$ for $p+q\neq n$;
			\item $h^{n,0}=1$ and $h^{p,q}=0$ for $p\notin[0,n]$.
		\end{itemize}
		
		(2) An $\Rbb$-PVHS on some complex manifold $M$ is of $CY_n$ type if the $\Rbb$-HS on any of its fibers is of $CY_n$ type.
		
		In the same way one talks about $\Cbb$-PVHS of $CY_n$ type using $\Cbb$-linear representations of $\Gbf_\Cbb$ with the same constraints on Hodge numbers.
	\end{definition}
	
	Thus equivariant $\Rbb$-PVHS and $\Cbb$-PVHS on the domain $\Dcal$ associated to $(\Gbf,h)$ are completely characterized in terms of representations. The classification of such PVHS over irreducible Hermitian symmetric domains is given in \cite{FL}, following previous works by \cite{Gro} and \cite{SZuo}: \cite{Gro} and \cite{SZuo} classify the $\Cbb$-PVHS of $CY_n$ type associated to the canonical representations of $\Gbf$, and \cite{FL} show further that in any symmetric power of the canonical representation (up to a suitable half twist if one considers $\Cbb$-PVHS) there exists a unique summand giving rise to a representation of $CY_n$ type, and every PVHS of $CY_n$ type arise in this way.

	We are interested in the following property of PVHS:
	
	\begin{lemma}
		Let $\Dcal$ be an irreducible Hermitian symmetric domain associated to some $\Rbb$-Shimura datum $(\Gbf,h)$ as above, and let $(V,\rho)$ be an irreducible $\Cbb$-linear representation of $\Gbf$ giving rise to a $\Cbb$-HS of $CY_n$ type. Then $\dim \Dcal$ does not exceed the Hodge number $h^{n-1,1}$ of $(V,\rho\circ h)$ as a $\Cbb$-HS.
	\end{lemma}
	
	\begin{proof}
		
		Being an irreducible Hermitian symmetric domain, $\Dcal$ and $(\Gbf,h)$ are exhausted by the four families of classical domains $I_{p,q},\ II_m,\ III_m,\ IV_m$ and two exceptional ones $E_6, E_7$. For each of these irreducible domains, there exists a canonical $\Cbb$-linear representation $(V_D,\rho_D)$ which is irreducible and gives rise to a $\Cbb$-HS of $CY_{n_D}$ type: $n_D$ and $(V_D,\rho_D)$ are explicitly computed in \cite{Gro} and \cite{SZuo}. It is further showed in \cite{FL} that a general irreducible $\Cbb$-linear representation $V$ of $CY_n$ type of $\Gbf$ has to be a summand in some symmetric power $\Sym^k(V_D,\rho_D)$ with $n=kn_D$.
		
		In particular, if we write $\mu$ for the Hodge cocharacter associated to $h: \Sbb\ra\Gbf$, then $\mu$ corresponds to a special node in the Dynkin diagram of $\Gbf_\Cbb$, and corresponds further to an cominuscule dominant weight $\lambda$ in the root datum of $\Gbf_\Cbb$ with respect to some Borel pair $(\Bbf,\Tbf)$ defined over $\Cbb$. It is shown in \cite{FL} that an irreducible $\Cbb$-linear representation $(V,\rho)$ of $CY_n$ type mentioned above has to be associated to the dominant weight $n\lambda$, and is generated by a $\Cbb$-linear line $L$ on which $\Tbf$ acts through $n\lambda$. The stabilizer of  $L$ in $\Gbf_\Cbb$ thus equals the parabolic $\Cbb$-subgroup $\Pbf_\mu$ in $\Gbf_\Cbb$ uniquely determined by $\mu$. Note that $L$ is the component $V^{n,0}$ in the Hodge decomposition for $V$.
		
		At the level of Lie algebras we have $\Lie\Gbf_\Cbb=\nfrak_-\oplus\Lie\Pbf_\mu$, where $\nfrak_-$ is the unipotent radical of the opposite to the parabolic $\Lie\Pbf_\mu$. Note that $\Gbb_\mrm$ acts via $\mu$ on $\nfrak_-$ by weight $-1$ and on $\Lie\Pbf_\mu$ by weight $0$ and $1$. The irreducible action $\Gbf_\Cbb\times V\ra V$ induces an action of Lie algebra $\Lie\Gbf_\Cbb\times V\ra V$. This very map is bilinear and compatible with the Hodge structures, inducing an injective map $\nfrak_-\otimes_\Cbb L\ra V$, whose image necessarily lies in $V^{n-1,1}$. This shows that $$h^{n-1,1}\geq \dim\nfrak_-=\dim\Gbf_\Cbb-\dim\Pbf_\mu=\dim \Dcal$$ where the last equality follows from Borel embedding, namely the open immersion of $\Dcal$ into the dual flag variety $\Gbf_\Cbb/\Pbf_\mu$.

	\end{proof}
	
	\begin{corollary}\label{dim-ineq}
		Let $\Dcal$ be an Hermitian symmetric domain associated to an $\Rbb$-Shimura datum $(\Gbf,h)$, and let $V$ be an irreducible $\Cbb$-linear representation of $\Gbf$ of $CY_n$ type for some $n$. If each simple factor of $\Lie\Gbf^\der$ acts on $V$ faithfully, then holds the inequality $\dim D\leq h^{n-1,1}$ where $h^{n-1,1}$ is the Hodge number of type $(n-1,1)$ for $V$.
	\end{corollary}
	
	\begin{proof}
		We have a decomposition $\Dcal=\Dcal_1\times\cdots\times \Dcal_s$ of $\Dcal$ as a product of irreducible Hermitian symmetric domains, with $\Dcal_j$ associated to some $\Rbb$-Shimura datum $(\Gbf_j,h_j)$. In this case $\Gbf^\der$ is necessarily an almost direct product $\Gbf_1^\der\tilde{\times}\cdots\tilde{\times}\Gbf_s^\der$ with $\Lie\Gbf^\der=\bigoplus_j\Lie\Gbf_j^\der$. The action of $\Lie\Gbf^\der$ on $\Vbf$ is irreducible, and each factor $\Lie\Gbf_j^\der$ acts faithfully. Hence we obtain a decomposition $V=\bigotimes_jV_j$ where $V_j$ is an irreducible representation of $\Gbf_j$, carrying a $\Cbb$-HS of weight $n_j$ with $n_1+\cdots+n_s=n$. The Hodge types are additive with respect to tensor products. Since $V$ is of $CY_n$ type, the constraints on Hodge numbers forces each $V_j$ is of $CY_{n_j}$ type, and we have $$h^{n-1,1}(V)=\sum_jh^{n_j-1,1}(V_j)\geq\sum_j\dim \Dcal_j=\dim \Dcal.$$
	\end{proof}

	\begin{remark} The classifications in \cite{Gro} and \cite{SZuo} show that when $\Dcal$ associated to $(\Gbf,h)$ as above is irreducible and $V$ is the irreducible $\Cbb$-linear representation giving rise to the canonical $\Cbb$-PVHS of $CY_n$ type ($n=n_D$), then holds the equality $\dim \Dcal=h^{n-1,1}(V)$.
		
		\end{remark}
Now we prove our main Theorem \ref{CY-AO} in this section.
\begin{proof}[Proof of Theorem~\ref{CY-AO}]
By Theorem \ref{thm-generic-rig}, $\Dcal$ is a Hermitian symmetric domain of rank $\geq2$. Let $\Gbf$ be the associated generic Mumford-Tate group. Let $\mathbb W$ be the local system on $\Gamma\backslash \Dcal$ constructed in \cite[Theorem 2.2]{SZuo}. Then we infer that the local system $\mathbb V$ is equal to $\Phi^\ast \mathbb W$. Consider the associated Higgs bundle $(F,\phi)=(\bigoplus_{i=0}^nE^{n-i,i},\bigoplus_{i=0}^n\phi^{n-i,i})$ on $\Gamma\backslash \Dcal$ of $\mathbb W$ and the associated Higgs bundle $(E,\theta)=(\bigoplus_{i=0}^nE^{n-i,i},\bigoplus_{i=0}^n\theta^{n-i,i})$ on $\mathscr M_h$ of $\mathbb V$. Then rank$(F^{n-1,1})=$rank$(E^{n-1,1})$. And by Corollary \ref{dim-ineq}, we know that $\dim \Dcal\leq \text{rank}(F^{n-1,1})$. However, by the unobstructed Theorem of Calabi-Yau manifolds, we have $\dim \mathcal M_h=\text{rank}(E^{n-1,1})$. Thus we obtain that $\dim \Dcal\leq\dim \mathcal M_h$ and this proves that the period map $\Phi$ is dominant.
\end{proof}

\subsection{An explicit maximal non-rigid family of Calabi-Yau quintics in $\mathbb{P}^4$}\label{sec-VZ-eg}

In \cite{VZ} and \cite{VZ-06} Viehweg and Zuo constructed and studied several examples of rigid and non-rigid families of Calabi-Yau varieties using cyclic coverings of hypersurfaces. We recall here the construction that is most relevant for our purposes.

\subsubsection{The construction}\label{VZ-example-M_i}

Consider the following two moduli spaces:
\begin{itemize}
\item $M_1$: the moduli space of quintic cyclic covers of $\mathbb P^1$ branched over five distinct points.
      On $M_1$ there is a natural $\mathbb Q$-PVHS $\mathbb V$ of weight $1$: for a point $x\in M_1$ representing a curve $C_x$, the fibre is $H^1(C_x,\mathbb Q)$.

\item $M_2$: the moduli space of quintic cyclic covers of $\mathbb P^2$ branched over a normal-crossing quintic curve (e.g. a curve coming from $M_1$).
      On $M_2$ there is a natural $\mathbb Q$-PVHS $\mathbb W$ of weight $2$; for a point $y\in M_2$ representing a surface $S_y$, the fibre is a certain $\mathbb Q$-linear subspace of $H^2(S_y,\mathbb Q)$ (the primitive part, to be explained below).
\end{itemize}
(One can continue this procedure to obtain an iterated sequence as in \cite{VZ}.)

After passing to suitable finite coverings we obtain families
\[
g_1:Z_1\longrightarrow M_1,\qquad g_2:Z_2\longrightarrow M_2,
\]
whose fibres are the above cyclic covers.

\begin{theorem*}[Viehweg-Zuo, \cite{VZ, VZ-06}]
The product family $(g_1,g_2):Z_1\times Z_2\to M_1\times M_2$ admits a fibrewise $\mathbb Z_5$-action.
The quotient fits into a diagram
\[
\xymatrix{
(Z_1\times Z_2)/\mathbb Z_5 \ar[rrd]^{(g_1,g_2)} &&
\widehat{(Z_1\times Z_2)/\mathbb Z_5} \ar[ll]_{\text{blow up}} \ar[rr]^{\text{blow down}} \ar[d] &&
\mathcal Z \ar[lld]^{h} \\
&& M_1\times M_2 .
}
\]
The resulting family $h:\mathcal Z\to M_1\times M_2$ is a family of Calabi-Yau quintic hypersurfaces in $\mathbb P^4$.
Moreover, $h$ is a {\it maximal non-rigid family} in the sense of Definition~\ref{max_nr}.
\end{theorem*}

\subsubsection{The local systems on $M_1$ and $M_2$}
Now We take a closer look at the local systems appeared in the above construction:

\begin{itemize}

	\item A point $x$ in $M_1$ can be realized as a curve $C=C_x$ together with an evident finite map $\phi: C\ra\Pbb^1$, so that $H^1(C,\Qbb)=H^1(\Pbb^1,\phi_*\Qbb_{C})$.

	\item A point $y$ in $M_2$ can be realized as a surface $S=S_y$ together with an evident finite map $\psi: S\ra\Pbb^2$, so that $H^2(S,\Qbb)=H^2(\Pbb^2,\psi_*\Qbb_{S})$.

\end{itemize}

Both $\phi$ and $\psi$ are branched coverings which are generically Galois coverings of automorphism group $\mu_5=\{z\in\Cbb^\times: z^5=1\}$, so that on these cohomology spaces we have an evident action of $\Qbb(\mu_5)$. We can decompose $H^1(C,\Cbb)$ and $H^2(S,\Cbb)$ along characters of $\mu_5$. Note that these characters actually take value in the subfield $\Qbb(\mu_5)$ of $\Cbb$.

We thus put $F=\Qbb(\mu_5)=\Qbb(t)\subset\Cbb$ with $t=\exp{\frac{2\pi\ibf}{5}}$. This is a CM number field, of cyclic Galois group $\{\tau,\tau^2,\tau^3,\tau^4=\id\}\isom(\Zbb/5)^\times\isom\Zbb/4$ over $\Qbb$, with $\tau(t)=t^2$ and $\tau^2$ equal to the restriction of complex conjugation to $F$. The real part $F^+=F^{\tau^2}$ equals $\Qbb(u)$ with $u=t+t^\inv$ of minimal equation $u^2+u-1=0$ and $F^+=\Qbb(\sqrt{5})$, with $\Gal(F/F^+)=\{\sigma=\tau^2,\id\}$. We also write $\{\chi,\chi^2,\chi^3,\chi^4,\chi^5=1\}$ for the character group $\Hom(\mu_5,\Cbb^\times)\isom\Hom(\mu_5,F^\times)$, with $\chi^j(t)=t^j$, so that $\Gal(F/\Qbb)$ acts on $\Hom(\mu_5,F^\times)$ naturally with equalities like $\chi\circ\tau=\chi^2$, $\chi\circ\tau^2=\chi^4$, $\chi\circ\tau^3=\chi^3$, etc. These characters allow us to decompose the cohomology spaces. 

The discussions below are special cases of \cite{Rhode} and \cite{yu_zheng}, which find their origin in the Deligne-Mostow theory \cite{deligne_mostow}.

\subsubsection{Calculations of Mumford-Tate groups and Mumford-Tate domains}

\subsubsection*{The case of $M_1$}


The constant sheaf $\Cbb_C$ on $C$ gives $\phi_{*}\Cbb_C=\bigoplus_{i=0}^4L_{C,i}$ with $L_{C,i}$ corresponding to the characters $\chi^i$ for the cyclic covering. Thus we obtain $$H^1(C,\Qbb)\otimes_\Qbb \Cbb=H^1(C,\Cbb)=H^1(\Pbb^1,\phi_{*}\Cbb_C)=\bigoplus_{i=1}^4H^1(\Pbb^1,L_{C,i})$$ with no contribution from $L_{C,0}=\Cbb_{\Pbb^1}$ because $H^1(\Pbb^1,\Cbb)=0$. Since these characters take values in $F$, the decomposition actually descends over $F$.

On $M_1$ we have the natural local system $\Vbb$ whose fiber at $x$ is $H^1(C,\Qbb)$. The discussion affirms a decomposition $\Vbb\otimes_\Qbb F=\bigoplus_{i=1}^4\Lbb_i$ into local system of $F$-vector spaces along the $F$-valued characters $\chi^i$'s.

\begin{proposition}\label{MT-M_1}
  Let $x$ be a general point in $M_1$. Then the Hodge decomposition for $H^1:=\Vbb_x$ is of the form $H^1\otimes_\Qbb\Cbb=H^{1,0}\oplus H^{0,1}$ with \begin{itemize}
		\item $H^{1,0}=\bigoplus _{i=1}^4 H^0(\Pbb^1,\Ocal(3-i))=H^0(\Pbb^1,\Ocal(2))\oplus H^0(\Pbb^1,\Ocal(1))\oplus H^0(\Pbb^1,\Ocal)$

		\item $H^{0,1}=\bigoplus_{i=1}^4H^1(\Pbb^1,\Ocal(-i))=H^1(\Pbb^1,\Ocal(-2))\oplus H^1(\Pbb,\Ocal(-3))\oplus H^1(\Pbb^1,\Ocal(-4))$, which is further isomorhphic to $ H^0(\Pbb^1,\Ocal)\oplus H^0(\Pbb,\Ocal(1))\oplus H^0(\Pbb^1\Ocal(2))$ by Serre duality;

	\end{itemize} and summands in $H^1\otimes_\Qbb \Cbb=\bigoplus_{i=1}^4 H^1_i$ are of dimension 3, refined as \begin{itemize}
		\item $H^1_i\otimes_F\Cbb= H^0(\Pbb^1,\Ocal(3-i))\oplus H^1(\Pbb^1,\Ocal(-i))$ of dimension 3, $i=1,2,3,4$;
	\end{itemize}
	The polarization on $H^1$ induces an Hermitian structure on $H^1_1\oplus H^1_4$ of signature $(3,0)$, and an Hermitian structure on $H^1_2\oplus H^1_3$ of signature $(2,1)$, with respect to $F/F^+$.

	The Mumford-Tate group $\Gbf_1$ for $H^1$ is a reductive $\Qbb$-group, whose derived part equals the algebraic monodromy group for $\Vbb$ over $M_1$, isomorphic to $\Res_{F^+/\Qbb}\Hbf_1$ with $\Hbf_1$ the special unitary group of some Hermitian form $h_1: F^3\times F^3\ra F$ of signature $(3,0)$ and $(2,1)$ along the two different real embeddings $F^+\mono\Rbb$, so that $\Gbf_1^\der(\Rbb)\isom \SU(3,0)\times \SU(2,1)$. We also have $\End{\Gbf_1}{H^1}\isom F$.

\end{proposition}

\begin{proof}[Sketch of proof]

	This is a special case of the computations in \cite[Chapter 3, 4, 5]{Rhode}, and we only present an outline. Notice that the Serre duality implies the following table:
	\begin{center}
		\begin{tabular}{ c| c| c }
			& $H^{1,0}$ & $H^{0,1}$ \\
			\hline
			$i=1$ & 3 & 0 \\
			\hline
			$i=2$ & 2 & 1 \\
			\hline
			$i=3$ & 1 & 2 \\
			\hline
			$i=4$ & 0 & 3
		\end{tabular}
	\end{center} and the polarization paring induces an Hermitian form on $H^{1,0}_1\oplus H^{0,1}_4$ resp. on $H^{1,0}_2\oplus H^{0,1}_3$ of signature $(3,0)$ resp. $(2,1)$.

	The algebraic monodromy group of $\Vbb$ over $M_1$ is determined by the representation of $\pi_1(M_1,x)$ on $\Vbb_x$. Following \cite{Rhode}, we have an evident surjective map $\pi: \Pcal\ra M_1$ where $\Pcal$ is the configuration space of five distinct points, and $\pi$ sends $x=(x_1,x_2,x_3,x_4,x_5)$ to the point in $M_1$ corresponding to the curve of equation $T_2^5+(T_0-x_1T_1)\cdots(T_0-x_5T_1)=0$ in $\Pbb^2$. One finds that $\pi_1(M_1)$ only differs from $\pi_1(\Pcal)$ by a torsion factor, and it suffices to study the algebraic monodromy group for $\pi^*\Vbb$ over $\Pcal$. On the other hand, $\pi_1(\Pcal)$ has the explicit collection of generators by Dehn twists, and the corresponding action on $\Vbb_x\otimes_\Qbb\Cbb$ generates the subgroup $\SU(3,0)\times\SU(2,1)$ in $\GL_\Cbb(\Vbb_x\otimes_\Qbb\Cbb)$, which is already the real Lie group associated to the algebraic monodromy group $\Gbf_1^\der$.

	Meanwhile,  the Galois group of $F=\Qbb(\mu_5)$ over $\Qbb$ permutes transitively the summands $H^1_i$ in $H^1\otimes_\Qbb\Cbb$, which is semi-linear with respect to the natural action of $F$ on $H^1=H^1(C,\Qbb)$ via the cyclic covering over $\Pbb^1$. We notice that $H^1_1\oplus H^1_4$ and $H^1_2\oplus H^1_3$ are already irreducible $\Cbb$-linear representations of $\Gbf_1^\der(\Rbb)=\SU(3,0)\times\SU(2,1)$ (via the factors) which are already defined over $F^+$ and are permuted by $\Gal(F^+/F)$. Hence $H^1$ is irreducible as a representation of $\Gbf_1$.


	Clearly $F\subset\End{\Gbf_1}{H^1}$. Notice also that $H^1\otimes_\Qbb\Cbb=\bigoplus H^1_i$ is defined over $F$ and the summands are permuted transitively by $\Gal(F/\Qbb)$, one concludes that $F$ equals $\End{\Gbf_1}{H^1}$.
\end{proof}

\subsubsection*{The case of $M_2$}



The constant sheaf $\Cbb_S$ on $S$ gives $\psi_{*}\Cbb_S=\bigoplus_{i=0}^4L_{S,j}$ corresponding to the characters $\chi^j$ of the cyclic covering. Thus we obtain $$H^2(S,\Qbb)\otimes_\Qbb \Cbb=H^2(S,\Cbb)=H^2(\Pbb^2,\pi_{2,*}\Cbb_S)=\bigoplus_{j=0}^4H^2(\Pbb^2,L_{S,L_{S,j}}).$$ Note that $H^2(\Pbb^2,L_{S,0})=H^2(\Pbb^2,\Cbb)$ is defined over $\Qbb$ by the $\Qbb$-HS on $H^2(\Pbb^2,\Qbb)$, we obtain a decomposition $H^2(S,\Qbb)=H^2(\Pbb^2,\Qbb)\oplus H^2$ with $H^2$ a $\Qbb$-HS of weight 2 such that $H^2\otimes_\Qbb \Cbb=\bigoplus_{j=1}^4H^2_j$ using $H^2_j=H^2(\Pbb^2,L_{S,j})$. Since the characters take values in $F$, the decomposition is defined over $F$.

On $M_2$ we have the natural local system $\Wbb$ whose fiber at $y$ is $H^2$, supporting a $\Qbb$-PVHS of weight 2. We also have $\Wbb\otimes_\Qbb F=\bigoplus_{i=1}^4\Mbb_i$ into local systems of $F$-vector spaces along the characters.

\begin{proposition}\label{MT-M_2}
	Let $y$ be a general point in $M_2$. Then the Hodge decomposition for $H^2=\Wbb_y$ is of the form
	$H^2\otimes_\Qbb\Cbb=H^{2,0}\oplus H^{1,1}\oplus H^{0,2}$ with

	\begin{itemize}
		\item $H^{2,0}=\bigoplus_{j=1}^4 H^0(\Pbb^2,\Omega^2_{\Pbb^2}(5)\otimes \Ocal(-j))=H^0(\Pbb^2,\Ocal(1))\oplus H^0(\Pbb^2,\Ocal)$, summands of dimension 3, 1 respectively;

		\item $H^{0,2}=\bigoplus_{j=1}^{4}H^2(\Pbb^2,\Ocal(-j))=H^2(\Pbb^2,\Ocal(-3))\oplus H^2(\Pbb^2,\Ocal(-4))$, summands of dimension 1,3 respectively;

		\item $H^{1,1}=\bigoplus_{j=1}^4H^1(\Pbb^2,\omega^1_{\Pbb^2}\otimes \Ocal(-j))$ with $\omega^1_{\Pbb^2}=\Omega^1_{\Pbb^2}(\log D)$, $D$ being the ramification divisor for the cyclic cover $\pi_S$, and summands of dimension 10, 12, 12, 10 respectively.
	\end{itemize}
	The summands in $H^2\otimes_\Qbb F=\bigoplus_{j=1}^4 H^2_j$ are of dimension 13, $j=1,2,3,4$:
	\begin{itemize}
		\item  $H^2_1\otimes_F\Cbb=H^0(\Pbb^2,\Ocal(1))\oplus H^1(\Pbb^2,\omega^1_{\Pbb^2}(-1))$, $H^2_4\otimes_F\Cbb=H^2(\Pbb^2,\Ocal(-4))\oplus H^1(\Pbb^1,\omega^1_{\Pbb^2}(-4))$;
		\item $H^2_2\otimes_F\Cbb=H^0(\Pbb^2,\Ocal)\oplus H^1(\Pbb^2,\omega^1_{\Pbb^2}(-2))$, $H^2_3\otimes_F\Cbb=H^{2}(\Pbb^2,\Ocal(-3))\oplus H^1(\Pbb^2,\omega^1_{\Pbb^2}(-3))$;
	\end{itemize}


	On the $\Cbb$-vector spaces $H^2_1$ and $H^2_4$ there exist natural Hermitian form of signature $(3,10)$, and similarly on $H^2_2$ and $H^2_3$ of signature $(1,12)$. They descend to Hermitian forms defined over the CM extension $F/F^+$.

	The Mumford-Tate group $\Gbf_2$ for $H^2$ is a reductive $\Qbb$-group, whose derived part equals the algebraic monodromy group for $\Wbb$ over $M_2$,  and is realized as a $\Qbb$-subgroup of $\Res_{F^+/\Qbb}\Hbf_2$ with $\Hbf_2$ the automorphism group of some Hermitian form $h_2: F^{13}\times F^{13}\ra F$, of signature $(3,10)$ and $(1,12)$ along the two different real embeddings $F^+\mono\Rbb$, so that $\Gbf_2(\Rbb)\isom \SU(3,10)\times \SU(1,12)$. We also have $\mathrm{End}_{\Gbf_2}(H^2)=F$.
\end{proposition}

\begin{proof}[Sketch of proof]

	We first study the Hodge decomposition in $H^2\otimes_\Qbb\Cbb$.

	It is easy to determine that $H^{2,0}=H^0(\Pbb^2,\Ocal(1))\oplus H^0(\Pbb^2,\Ocal)$ and $h^0(\Pbb^2,\Ocal(1))=3$, $h^0(\Pbb^2,\Ocal)=1$. Similarly, one has $H^{0,2}= H^2(\Pbb^2,\Ocal(-3))\oplus H^2(\Pbb^2,\Ocal(-4)) \cong H^0(\Pbb^2,\Ocal) \oplus H^0(\Pbb^2,\Ocal(1))$ with dimensions 1,3. Thus we shall focus on the calculation of dimension of eigenspaces of $H^{1,1}$:
	\[
	H^{1,1}=\bigoplus_{j=1}^4H^1(\Pbb^2,\Omega^1_{\Pbb^2}(\log D) \otimes \Ocal(-j)).
	\]

	From the short exact sequence
	\[
	0 \to \Omega^1_{\Pbb^2} \to \Omega^1_{\Pbb^2}(\log D) \to \Ocal_D \to 0
	\]
	one obtains
	\[
	0 = H^1(\Omega^1_{\Pbb^2}(-i)) \to {H^1(\Omega^1_{\Pbb^2}(\log D)(-i))} \to H^1(D, \Ocal_D(-i)) \to H^2(\Omega^1_{\Pbb^2}(-i)) \to 0
	\]
	which implies
	\[
	h^1(\Omega^1_{\Pbb^2}(\log D)(-i)) = h^1(D, \Ocal_D(-i)) - h^2(\Omega^1_{\Pbb^2}(-i)).
	\]

	\begin{itemize}
		\item $h^1(D, \Ocal_D(-i))$: By Riemann-Roch
		\[
		h^1(D, \Ocal_D(-i)) = -\mathrm{deg}_D\Ocal_D(-i) - 1 + g(D) = 5i-1 + \frac{(5-1)(5-2)}{2} = 5(i+1).
		\]
		\item $h^2(\Omega^1_{\Pbb^2}(-i))$: By the Euler exact sequence
		\[
		0 \to \Omega^1_{\Pbb^2} \to \Ocal(-1)^{\oplus 3} \to \Ocal \to 0
		\]
		one obtains
		\[
		0=H^1(\Ocal(-i)) \to H^2(\Omega^1_{\Pbb^2}(-i)) \to H^2(\Ocal(-1-i)^{\oplus 3}) \to H^2(\Ocal(-i)) \cong H^0(O(i-3)) \to 0.
		\]
		Thus
		\[
		\begin{array}{ll}
			h^2(\Omega^1_{\Pbb^2}(-i)) &= 3 \cdot h^2(\Ocal(-1-i)) - h^0(\Ocal(i-3)) \\
			& = 3 \cdot h^0(\Ocal(i-2)) - h^0(\Ocal(i-3)).
		\end{array}
		\]
		In summary,
		\[
		h^1(\Omega^1_{\Pbb^2}(\log D)(-i)) = 5(i+1) - 3 \cdot h^0(\Ocal(i-2)) + h^0(\Ocal(i-3)).
		\]

		\begin{center}
			\begin{tabular}{ c| c| c|c }
				& $H^{2,0}$ & $H^{1,1}$ & $H^{0,2}$ \\
				\hline
				$i=1$ & 3 & 10 & 0 \\
				\hline
				$i=2$ & 1 & 12 & 0 \\
				\hline
				$i=3$ & 0 & 12 & 1 \\
				\hline
				$i=4$ & 0 & 10 & 3
			\end{tabular}
		\end{center}

	\end{itemize}


	Clearly the decomposition of $H^2\otimes\Cbb$ is compatible with the $F$-action, and each summand $H^2_i$ is defined over $F$. This is actually a Hodge structure with CM symmetry in the sense of \cite{geeman} (we avoid the name ``of CM type'' to prevent possible ambiguities): each $H^2_i$ is an $F$-vector space suporting a Hodge structure with CM symmetry by $F$, and is polarized by an Hermitian form of signature as indicated in the table. Moreover, the Galois action of $\Gal(F/\Qbb)$ permutes them transitively sending $H^2_\chi$ to $H^2_{{\chi\circ\tau}}$ etc.  so that $H^2_1\oplus H^2_4$ and $H^2_2\oplus H^2_3$ are both defined over $F^+$.

	The Mumford-Tate group $\Gbf_2$ preserves the symmetry by $F$ on $H^2$ as a Hodge structure, hence it also preserves the CM symmetry by $F$ of the natural $F$-structure on $H^2_i$. Thus $\Gbf_2^\der$ is a $\Qbb$-subgroup of $\Res_{F^+/\Qbb}\Hbf_2$, where $\Hbf_2$ is the $F^+$-group of automorphisms (of determinant 1) of an Hermitian form $h_2: F^{13}\times F^{13}\ra F$ descended from the polarizations on $H^2_i$.


	The equality $\End{\Gbf_2}{H^2}=F$ is similar to the case of $M_1$ because the four pieces $H^2_i$ are permuted transitively by $\Gal(F/\Qbb)$.
\end{proof}

\begin{remark}
(1)	In \cite{yu_zheng} Yu and Zheng studied half-twisted Hodge structures of ball type (i.e. whose period spaces are Hermitian symmetric domains of $\SU(1,n)$-type), and they conjectured that the period map should be an open embedding under suitable conditions. In our setting we expect that the derived Mumford-Tate group should be exactly $\Res_{F^+/\Qbb}\Hbf_2$ as in the lemma above, and the period map should be generically injective.

(2) The $\Qbb$-group $\Res_{F^+/\Qbb}$ in the lemma above is already of Hermitian type, in the sense that the associated period domain is of Hermitian type (i.e. of Hodge level 1 in the languageof Section 3). Thus the period domain of $M_2$ described by $\Gbf_2$ is also of Hermitian type, and the period map takes value in a Shimura variety.
\end{remark}

Note that $\End{\Gbf_1}{H^1}\otimes_\Qbb\End{\Gbf_2}{H^2}=F\otimes_\Qbb F$ is NOT a field, hence  $H^1\otimes H^2$ is NOT irreducible as a representation of the linear group $\Gbf_1\times\Gbf_2$ over $\Qbb$. The $\Qbb$-linear subspace $V$ of $\mu_5$-invariants in $H^1\otimes_\Qbb H^2$ naturally supports a $\Qbb$-HS, and we may characterize $V\otimes_\Qbb F$ as the $\mu_5$-invariant subspace in $$(H^1\otimes_\Qbb H^2)\otimes_\Qbb F=(\bigoplus_{i=1}^4H^1_i)\otimes_F(\bigoplus_{j=1}^4H^2_j).$$ Since $\mu_5$ acts on $H^1_i\otimes_FH^2_j$ by the character $\chi^{i+j}$, we obtain $$V\otimes_\Qbb F=H^1_1\otimes H^2_4\oplus H^1_4\otimes H^2_1\oplus H^1_2\otimes H^2_3\oplus H^1_3\otimes H^2_2,$$ and the action of $\tau$ on $(H^1\otimes_\Qbb H^2)\otimes_\Qbb F$ via $F$ induces a permutation $\tau^*$ on the direct summands: $$H^1_1\otimes H^2_4\overset{\tau^*}{\lra}H^1_2\otimes H^2_3\overset{\tau^*}{\lra}H^1_{4}\otimes H^2_1\overset{\tau^*}{\lra}H^1_3\otimes H^2_2\overset{\tau^*}{\lra}H^1_1\otimes H^2_4.$$ Since $\Gal(F/\Qbb)$ acts on these components transitively, and $(H^1_1\otimes_FH^2_4)\otimes_F\Cbb$ is irreducible as a $\Cbb$-linear representation of $U(3,0)\times U(3,10)$, we deduce that $V$ is irreducible as a representation of $\Gbf_1\times \Gbf_2$, and $\Gbf_1\times \Gbf_2$ is the Hodge group for $V$ as a $\Qbb$-HS. 

In particular, the $F$-subspaces $A=H^1_1\otimes H^2_4\oplus H^1_4\otimes H^2_1$ and $B=H^1_2\otimes H^2_3\oplus H^1_3\otimes H^2_2$ are both stable under the action of $\tau^2$, and descend to $F^+$, i.e. $A=A^+\otimes_{F^+}F$ and $B=B^+\otimes_{F^+}F$ for $F^+$-linear subspace $A^+=A^{\tau^2}$ and $B^+=B^{\tau^2}$. From $\tau^*(A)=B$ and $\tau^*(B)=A$ we deduce that neither $A^+$ nor $B^+$ descends further to $\Qbb$: rather we have $V\otimes_\Qbb F^+=A^+\oplus B^+$ with $A^+$ and $B^+$ permuted by the non-trivial automorphism of $F^+$ sending $t+t^4$ to $t^2+t^3$.

\subsubsection{Verification of unobstructedness and specialness}
Let $M:=\mathscr M_{5,3}$, the moduli space of smooth quintic hypersurfaces in $\mathbb{P}^4$. Choose a smooth compactification $\overline M\supset M$ with $D_\infty=\overline M\backslash M$. Let $(E,\theta)\supset(E,\theta)_0$ be the Deligne's canonical extension.
By (\ref{incl}), for any morphism $f:X\ra \overline M$, one has $$H^0(X,f^\ast T_{\overline M}(-\log D_\infty))\subset H^0(X,f^\ast (E^{0,3}\otimes E^{2,1}))\text{\ and}$$ $$\dim_\mathbb C H^0(X,f^\ast T_{\overline M}(-\log D_\infty))\leq \dim_\mathbb C H^0(X,f^\ast (E^{0,3}\otimes E^{2,1})).$$

It is shown in \cite{VZ-06} that the $\mathbb Q$-VHS decomposes as
\[
R^3h_*\mathbb Q_{\mathcal Z}= \mathbb V\;\oplus\; \mathbb T,
\]
where
\[
\mathbb V:=\bigl(R^1g_{1*}\mathbb Q_{Z_1}\boxtimes R^2g_{2*}\mathbb Q_{Z_2}\bigr)^{\mu_5}
\]
(the superscript $\mu_5$ denotes the $\mathbb Z_5$-invariant part) and $\mathbb T$ is a Tate twist of a weight-1 $\mathbb Q$-VHS arising from the birational modifications.\\[.1cm]

\paragraph{Hodge bundles on the product.}
Let $M_1$, $M_2$ be as in \S\ref{VZ-example-M_i}. Since $(E,\theta)$ is of Calabi-Yau type, by Proposition \ref{MT-M_1} and Proposition \ref{MT-M_2} we have
\begin{align*}
    E^{0,3}|_{\overline{M_1}\times \overline{M_2}}=(E^{0,1}_1)_2\boxtimes(E^{0,2}_2)_3,
\end{align*}
where $\overline{M_i}$ is the corresponding projective compactification.

Similarly, one has
\begin{align*}
     E^{2,1}|_{\overline{M_1}\times \overline{M_2}}=(E^{1,0}_1\boxtimes E^{1,1}_2)^{\mu_5}\oplus(E^{0,1}_1\boxtimes E^{2,0}_2)^{\mu_5}\oplus E^{2,1}_\mathbb T,
\end{align*}
where $(E_\mathbb T,\theta_\mathbb T)$ is the Hodge bundle associated with $\mathbb T$ and
\begin{align*}
    (E^{1,0}_1\boxtimes E^{1,1}_2)^{\mu_5}=&(E_1^{1,0})_1\boxtimes(E_2^{1,1})_4\oplus (E_1^{1,0})_2\boxtimes(E_2^{1,1})_3\oplus (E_1^{1,0})_3\boxtimes(E_2^{1,1})_2;\\
    (E^{0,1}_1\boxtimes &E^{2,0}_2)^{\mu_5}=(E_1^{0,1})_3\boxtimes(E_2^{2,0})_2\oplus (E_1^{0,1})_4\boxtimes(E_2^{2,0})_1.
\end{align*}

Now, take a general point $t\in M_2$, and consider the log morphism
\begin{align*}
    \varphi_t:=\varphi|_{\overline{M_1}\times\{t\}}:(\overline{M_1},D_1)\ra (\overline M,D_\infty).
\end{align*}
Then we have
\begin{align*}
    T_{[\varphi_t]}\Hom((\overline{M_1},D_1),(\overline M,D_\infty))=&H^0(\overline{M_1},\varphi_t^\ast T_{\overline M}(-\log D_\infty))\\
    \subset &H^0(\overline{M_1},(E^{0,3}\otimes E^{2,1})|_{\overline{M_1}\times\{t\}}).
\end{align*}
Note that $E^{0,3}|_{\overline{M_1}\times\{t\}}=(E_1^{0,1})_2=(E^{1,0}_1)_3^\ast$,
and by Proposition \ref{MT-M_1} and Proposition \ref{MT-M_2}, we have
\begin{align*}
      (E^{1,0}_1\boxtimes E^{1,1}_2)^{\mu_5}|_{\overline{M_1}\times\{t\}}=&(E_1^{1,0})_1^{\oplus10}\oplus (E_1^{1,0})_2^{\oplus12}\oplus (E_1^{1,0})_3^{\oplus12};\\
    (E^{0,1}_1\boxtimes E^{2,0}_2)^{\mu_5}&|_{\overline{M_1}\times\{t\}}=(E_1^{0,1})_3\oplus (E_1^{0,1})_4^{\oplus3}.
\end{align*}

\paragraph{Irreducibility of the factors.}
\begin{proposition}\label{prop:irred}
All local systems $(R^1g_{1*}\mathbb Q_{Z_1})_i$, $(R^2g_{2*}\mathbb Q_{Z_2})_i$ ($i=1,2,3,4$) and $\mathbb T$ are irreducible and have pairwise different ranks.
Consequently, by Schur's lemma and Simpson's correspondence,
\[
\Hom\bigl((E_1)_i,(E_1)_j\bigr)=0\ \text{for }i\neq j,\qquad
\Hom(E_1,E_{\mathbb T})=0.
\]
\end{proposition}
\begin{proof}
The irreducibility and distinctness of ranks follow from [VZ, Lemma~4.1] together with Propositions~\ref{MT-M_1} and~\ref{MT-M_2}.  The vanishing of Homs then follows from Schur's lemma and Simpson's correspondence {\cite[Theorem~1 and Lemma~3.5]{Simp}}.
\end{proof}

\paragraph{Unobstructedness for the first factor.}
Fix a general point $t\in M_2$ and consider the restricted log map
\[
\varphi_t:=\varphi|_{\overline M_1\times\{t\}} : (\overline M_1,D_1)\longrightarrow(\overline M,D_\infty).
\]
By the above,
\[
T_{[\varphi_t]}\Hom\bigl((\overline M_1,D_1),(\overline M,D_\infty)\bigr)
\;\subseteq\; H^0\!\bigl(\overline M_1,\;(E^{0,3}\otimes E^{2,1})|_{\overline M_1\times\{t\}}\bigr).
\]
Now $E^{0,3}|_{\overline M_1\times\{t\}}=(E_1^{0,1})_2=(E_1^{1,0})_3^\vee$, and using the decomposition of $E^{2,1}$ we obtain
\[
\begin{aligned}
(E^{0,3}\otimes E^{2,1})|_{\overline M_1\times\{t\}}
=&\ (E_1^{1,0})_3^\vee\otimes\Bigl((E_1^{1,0})_1^{\oplus10}\oplus(E_1^{1,0})_2^{\oplus12}\oplus(E_1^{1,0})_3^{\oplus12}\Bigr)\\
&\ \oplus\ (E_1^{1,0})_3^\vee\otimes\Bigl((E_1^{0,1})_3\oplus(E_1^{0,1})_4^{\oplus3}\Bigr).
\end{aligned}
\]

Applying Proposition~\ref{prop:irred}, the only non-zero Homs occur between identical factors, hence
\[
H^0\!\bigl(\overline M_1,\;(E^{0,3}\otimes E^{2,1})|_{\overline M_1\times\{t\}}\bigr)
\;\cong\; \Hom\bigl((E_1^{1,0})_3,\,(E_1^{1,0})_3^{\oplus12}\bigr)
\;\cong\; \mathbb C^{12}.
\]
(The term $\Hom((E_1^{1,0})_3,(E_1^{0,1})_3)$ vanishes because it would give a section of $T_{\overline M_1}(-\log D_1)$, and $M_1$ is of ball-quotient type.)
Thus
\[
\dim_{\mathbb C}H^0\!\bigl(\overline M_1,\;(E^{0,3}\otimes E^{2,1})|_{\overline M_1\times\{t\}}\bigr)=12=\dim M_1,
\]
which establishes the unobstructedness of $\varphi_t$.\\[.1cm]

\paragraph{Unobstructedness for the second factor.}
Now fix a general point $s\in M_1$ and consider
\[
\varphi_s:=\varphi|_{\{s\}\times\overline M_2}:(\overline M_2,D_2)\longrightarrow(\overline M,D_\infty).
\]
We have $E^{0,3}|_{\{s\}\times\overline M_2}=(E_2^{0,2})_3\cong(E_2^{2,0})_2^\vee$, and
\[
\begin{aligned}
(E^{2,1})|_{\{s\}\times\overline M_2}
=&\ (E_2^{1,1})_4^{\oplus3}\oplus(E_2^{1,1})_3^{\oplus2}\oplus(E_2^{1,1})_2\\
&\ \oplus\ (E_2^{2,0})_2^{\oplus2}\oplus(E_2^{2,0})_1^{\oplus3}.
\end{aligned}
\]

Again by Proposition~\ref{prop:irred},
\[
H^0\!\bigl(\overline M_2,\;(E^{0,3}\otimes E^{2,1})|_{\{s\}\times\overline M_2}\bigr)
\;\cong\; \Hom\bigl((E_2^{2,0})_2,\,(E_2^{1,1})_2\bigr)
\;\oplus\; \Hom\bigl((E_2^{2,0})_2,\,(E_2^{2,0})_2^{\oplus2}\bigr),
\]
and the second summand is $\mathbb C^{2}$.

\begin{claim}
$\Hom\bigl((E_2^{2,0})_2,\,(E_2^{1,1})_2\bigr)=0$.
\end{claim}
\begin{proof}
Consider the natural embedding $\iota:M_1\hookrightarrow M_2$ induced by the cyclic cover construction \cite{VZ}.
Because $T_{\overline M}(-\log D_\infty)|_{M_2}$ is semi-negative, the restriction map
\[
H^0\!\bigl(\overline M_2,\;T_{\overline M}(-\log D_\infty)|_{M_2}\bigr)
\;\hookrightarrow\; H^0\!\bigl(\overline M_2,\;\iota^*T_{\overline M}(-\log D_\infty)|_{M_2}\bigr)
\]
is injective.  Hence it suffices to show
\[
\Hom\bigl(\iota^*(E_2^{2,0})_2,\;\iota^*(E_2^{1,1})_2\bigr)=0.
\]

On the smooth part $M_2$ there is a Kodaira-Spencer isomorphism
\[
\delta:T_{M_2}\;\stackrel{\sim}{\longrightarrow}\; (E_2^{2,0})_2^\vee\otimes(E_2^{1,1})_2
\]
associated to the factor $(R^2g_{2*}\mathbb Q_{Z_2})_2$.  Although $\delta$ does not extend to an isomorphism of the logarithmic bundles, its restriction to $M_1$ does extend over a complete curve $H\subset\overline M_1$ that avoids the boundary $D_1$ (such a curve exists because $M_1$ is a ball-quotient surface whose Baily-Borel compactification has only finitely many cusps).

Over $H\subset M_1$ we therefore have
\[
\Hom\bigl(\iota^*(E_2^{2,0})_2,\;\iota^*(E_2^{1,1})_2\bigr)|_H
\;\cong\; H^0\!\bigl(H,\;\iota^*( (E_2^{2,0})_2^\vee\otimes(E_2^{1,1})_2)|_H\bigr)
\;\cong\; H^0\!\bigl(H,\;\iota^*T_{M_2}|_H\bigr)=0,
\]
since $\iota^*T_{M_2}|_H$ is strictly negative.  This proves the claim.
\end{proof}

Thus
\[
\dim_{\mathbb C}H^0\!\bigl(\overline M_2,\;(E^{0,3}\otimes E^{2,1})|_{\{s\}\times\overline M_2}\bigr)=2=\dim M_1,
\]
establishing unobstructedness for $\varphi_s$.\\[.1cm]

\paragraph{Conclusion.}
Both restricted maps $\varphi_t$ and $\varphi_s$ satisfy the unobstructedness condition of Conjecture~\ref{unob-conj}.  By Proposition~\ref{spe}, the subvariety $(H_1,S_1)\times(H_2,S_2)$ is therefore special with respect to the $\mathbb Z$-PVHS $R^3h_*\mathbb Z_{\mathcal Z}$.  Hence the Viehweg-Zuo example provides a concrete instance where the specialness of bi-Hom schemes holds.


\end{document}